\documentclass[11pt]{amsart}
\usepackage{pdfsync}
\usepackage{fullpage}
\usepackage{amsmath}
\usepackage{amssymb}
\usepackage{MnSymbol}
\usepackage{verbatim}
\usepackage[linktocpage]{hyperref}
\usepackage{enumitem}

\newtheorem{df}{Definition}[section]
\newtheorem{thm}[df]{Theorem}

\newtheorem{prop}[df]{Proposition}
\newtheorem{rem}[df]{Remark}

\newtheorem{lem}[df]{Lemma}

\newtheorem{cor}[df]{Corollary}
\newtheorem{conj}[df]{Conjecture}

\newcommand{\ga}{\gamma}

\newcommand{\lam}{\lambda}
\newcommand{\ti}{\tilde}
\newcommand{\mc}{\mathcal}
\newcommand{\mb}{\mathbb}

\newcommand{\pr}{^{\prime}}
\newcommand{\eq}{eqnarray*}

\title{Domination results in $n$-Fuchsian fibers in the moduli space of Higgs bundles}
\author{Song Dai\textsuperscript{1}}

\address{Song Dai\\
Center for Applied Mathematics\\
Tianjin University\\
No.92 Weijinlu Nankai District\\
Tianjin\\
P.R.China 300072}

\email{song.dai@tju.edu.cn}

\author{Qiongling Li\textsuperscript{2}}

\address{Qiongling Li\\
Chern Institute of Mathematics and LPMC\\
Nankai University\\
No. 94 Weijinlu Nankai District\\
Tianjin\\
P.R.China 300071}

\email{qiongling.li@nankai.edu.cn}
\date{}

\begin{document}
\maketitle
\begin{abstract}
In this article, we show some domination results on the Hitchin fibration, mainly focusing on the $n$-Fuchsian fibers. More precisely, we show the energy density of associated harmonic map of an $n$-Fuchsian representation dominates the ones of all other representations in the same Hitchin fiber, which implies the domination of topological invariants: translation length spectrum and entropy. As applications of the energy density domination results, we obtain the existence and uniqueness of equivariant minimal (or maximal) surfaces in certain product Riemannian (or pseudo-Riemannian) manifold. Our proof is based on establishing an algebraic inequality generalizing a GIT theorem of Ness on the nilpotent orbits to general orbits.
\end{abstract}
\tableofcontents
\section{Introduction}
Let $S$ be an oriented closed surface of genus at least $2$ and $\Sigma=(S,J)$ be a Riemann surface structure on $S$. The celebrated non-Abelian Hodge correspondence developed by Hitchin \cite{Hitchin87}, Simpson \cite{Simpson88}, Corlette \cite{Corlette}, and Donaldson \cite{Donaldson}, is a homeomorphism between the representation variety $\mc{M}_{\text{Betti}}(S)$ of reductive representations from $\pi_1(S)$ into $SL(n,\mathbb C)$ and the moduli space $\mathcal M_{\text{Higgs}}(\Sigma)$ of polystable $SL(n,\mb{C})$-Higgs bundles over $\Sigma.$ The correspondence is transcendental since it involves solving the Hitchin equation for Higgs bundles, which is a nonlinear second-order elliptic system. In this paper, without analyzing the Hitchin equation, we investigate the properties of corresponding representations under certain algebraic restrictions of Higgs bundles. This sheds light on understanding part of the non-Abelian Hodge correspondence.

The properties we are going to deduce are in terms of domination results on the translation length spectrum. This is motivated from the result by Deroin-Tholozan \cite{DominationFuchsian} that any $SL(2,\mathbb C)$-representation can be dominated by some Fuchsian representation using harmonic map method (for $SL(2,\mathbb R)$-representations, the result is also proved independently by Gu\'eritaud-Kassel-Wolff \cite{GKW} with a different method). Recently, such domination results are generalized to complete surfaces by Sagman \cite{Sagman} and surfaces with boundary by Gupta-Su \cite{GuptaSu}. One can view the domination results in this paper as a generalization to the $SL(n,\mathbb C)$-case as much as possible. If we move to consider $SL(n,\mathbb C)$-representations, the main disadvantage is that the associated symmetric space of $SL(n,\mathbb C)$ is no longer negatively curved, a key property being used in Deroin-Tholozan's work. However, we manage to recover the property of being negatively curved for certain Higgs bundles and thus are able to make use the techniques developed in Deroin-Tholozan's work, 
\subsection{Main results}
An $SL(n,\mb{C})$-Higgs bundle over the Riemann surface $\Sigma$ is a pair $(E,\phi)$, where $E$ is a holomorphic rank $n$ vector bundle of trivial determinant and $\phi$ is a trace-free $End(E)$-valued holomorphic $1$-form. Let $K_\Sigma$ be the canonical bundle of $\Sigma$. The Hitchin fibration is a map $p:\mc{M}_{\text{Higgs}}(\Sigma) \rightarrow\bigoplus_{i=2}^nH^0(\Sigma, K_{\Sigma}^i)$ and the Hitchin section is constructed in Hitchin \cite{Hitchin92} explicitly using the principal $3$-dimensional Lie subalgebra. Under the non-Abelian Hodge correspondence, the Hitchin section corresponds to a connected component in the representation variety of $SL(n,\mathbb R)$, called the Hitchin component. Elements in the Hitchin component are called Hitchin representations, which is the main subject in the higher Teichm\"uller theory. The Teichm\"uller space $\mathcal T(S)$ consists of Fuchsian representations from $\pi_1(S)$ to $PSL(2,\mathbb R)$, which can always be lifted to $SL(2,\mathbb R)$. Composing with the irreducible representation $\tau_n: SL(2,\mathbb R)\rightarrow SL(n,\mathbb R)$, $\mathcal T(S)$ embeds naturally into the Hitchin component, as the sublocus consisting of $\tau_n\circ j$ which will be called \textit{$n$-Fuchsian representations}. The Hitchin fiber containing an $n$-Fuchsian representation is call \textit{an $n$-Fuchsian fiber}.

From the non-Abelian Hodge theory, for every representation $\rho\in\mc{M}_{\text{Betti}}(S)$, there exists a $\rho$-equivariant harmonic map $f: \ti{\Sigma}\rightarrow X:=SL(n,\mb{C})/SU(n)$, where $X$ is equipped with the $SL(n,\mathbb C)$-invariant metric induced by the rescaled Killing form on $sl(n,\mathbb C)$. Denote $e(f)$ as the energy density of $f$, $g_{f_{\rho}}$ as the pullback metric of $f$.  Let $j: \pi_1(S)\rightarrow SL(2,\mathbb R)$ be a Fuchsian representation. From Wolf \cite{TeichOfHarmonic} and Hitchin \cite{Hitchin87}, for every holomorphic quadratic differential $q_2$ on $\Sigma$, there is a unique Fuchsian representation $j$ up to conjugacy, so that the Hopf differential of the unique $j$-equivariant harmonic map $f_j:\ti{\Sigma}\rightarrow \mb{H}^2$ is a lift of $q_2$ to $\ti\Sigma$.

In the following theorem, we show that an $n$-Fuchsian representation dominates other representations in the same Hitchin fiber in the geometric and topological sense. For a representation $\rho:\pi(S)\rightarrow SL(n,\mathbb C)$, denote by $\mathbb P(\rho)$ the composition of $\rho$ with the natural projection from $SL(n,\mathbb C)$ to $PSL(n,\mathbb C)$.

\begin{thm}\label{Theorem1Intro}(Theorem \ref{dommain})
Suppose $\rho\in \mc{M}_{\text{Betti}}(S)$ is in an $n$-Fuchsian fiber of $\mathcal M_{Higgs}(\Sigma)$ containing $\tau_n\circ j$, then \\
(1) the energy density satisfies $e(f)< e(f_{\tau_n\circ j})$;\\
(2) the pullback metric satisfies $g_{f}< g_{f_{\tau_n\circ j}}$;\\
(3) the translation length spectrum satisfies $l_\rho< \lambda\cdot l_{\tau_n\circ j}$ for some positive constant $\lambda<1$; \\
(4) the energy satisfies $E(f)< E(f_{\tau_n\circ j});$ \\
(5) the entropy satisfies $h(\rho)>h(\tau_n\circ j)=\sqrt{\frac{6}{n^3-n}}$,\\
unless $\mathbb P(\rho)=\mathbb P(\tau_n\circ j)$.
\begin{rem}
For $\mathbb P(\rho)=\mathbb P(\tau_n\circ j)$, we mean $\rho$ is conjugate to $(\tau_n\circ j)\cdot\mu_{(n)}$ for some unitary representation $\mu_{(n)}:\pi_1(S)\rightarrow \mathfrak{G}_{(n)}=\{e^{\frac{2k\pi\sqrt{-1}}{n}},k=1,\cdots,n\}\cdot I_n $, in which case, it has the same harmonic map and the same translation length spectrum as $\tau_n\circ j$.
\end{rem}
\end{thm}
\begin{rem}
In the case of $SL(2,\mathbb C)$, Theorem \ref{Theorem1Intro} were shown by Deroin and Tholozan \cite{DominationFuchsian}. Note that in this case, every Hitchin fiber is an $2$-Fuchsian fiber. 
\end{rem}
\begin{rem}
The second author in \cite{QLnil} shows a more refined domination result inside the nilpotent cone.
\end{rem}
\begin{rem}Potrie and Sambarino \cite{HitchinEntropy} showed that for any Hitchin representation $\rho: \pi_1(S)\rightarrow SL(n,\mathbb R)$, one has the entropy $h(\rho)\leq h(\tau_n\circ j)=\sqrt{\frac{6}{n^3-n}}$ and the equality holds only if $\rho$ is $n$-Fuchsian. We can see that the $n$-Fuchsian fibers possess an opposite behavior comparing to the Hitchin section in the Hitchin fibration.
\end{rem}

The Higgs bundles in $n$-Fuchsian fibers has characteristic polynomial 
$$\det(\lambda I-\phi)=(\lambda^2-(n-1)^2q_2)\cdots(\lambda^2-(n-2[\frac{n}{2}])^2q_2)\lambda^{n-2[\frac{n}{2}]}.$$ More generally, for the Hitchin fibers with the characteristic polynomial $\det(\lambda I-\phi)$ is either\\
(1) $(\lambda^2-a_1^2q_2)\cdots (\lambda^2-a_{[\frac{n}{2}]}^2q_2)\lambda^{n-2[\frac{n}{2}]}$ for $q_2\in H^0(\Sigma, K_{\Sigma}^2)$ and $a_i\in \mb{R}^{>0}$ are distinct; or\\
(2) $(\lambda-b_1\omega)\cdots (\lambda-b_n\omega)$ for $\omega\in H^0(\Sigma, K_{\Sigma})$ and $b_i\in \mb{R}$ are distinct, \\
we show the domination results in Theorem \ref{dommainUniReal}. 

Suppose the Higgs field is of rank at most $2$ everywhere, we also show the domination results in Theorem \ref{dommainRank2}. 

\subsection{Geometric applications}
One nice application of the energy density domination result in Theorem \ref{Theorem1Intro} is to study the associated equivariant minimal (maximal) surface in certain product Riemannian (pseudo-Riemannian) manifold. 

Let $\bar\tau_n$ be the induced map from $\mathbb H^2$ to $X=SL(n,\mathbb C)/SU(n)$ by $\tau_n$ and $g_n$ be the normalized invariant Riemannian metric on $X$ such that $\bar{\tau}_{n}^*g_{n}=g_{\mb{H}^2}$. Then $(f_j,f_{\rho})$ gives a $(j,\rho)$-equivariant harmonic map
$$(f_j,f_{\rho})^+:\ti{\Sigma}\rightarrow \big(\mathbb H^2\times X, g_{\mb{H}^2}+g_n\big),\quad
(f_j,f_{\rho})^-:\ti{\Sigma}\rightarrow \big(\mathbb H^2\times X, g_{\mb{H}^2}-g_n\big).$$
Since $f_j$ is a diffeomorphism, $(f_j,f_{\rho})^{\pm}$ must be an embedding.
The Hopf differential of $(f_j,f_{\rho})^{\pm}$ is $\text{Hopf}\big((f_j,f_{\rho})^{\pm}\big)=\text{Hopf}(f_j)\pm\text{Hopf}(f_{\rho}).$ Using Theorem \ref{Theorem1Intro}, we will show that the composed map $f\circ f_j^{-1}$ is area-decreasing if $\text{Hopf}(f_j,f_{\rho})^{+}=0$; distance-decreasing if $\text{Hopf}(f_j,f_{\rho})^{-}=0$. \\

\textbf{Minimal surface:} Suppose $\text{Hopf}(f_j)=-\text{Hopf}(f_{\rho})=q_2$, then the product map $(f_j,f_{\rho})^+$ is conformal. Together with the harmonicity, $(f_j,f_{\rho})^+$ gives a $(j,\rho)$-equivariant embedded minimal surface. We obtain the following proposition by making use of a result of Lee-Wang in \cite{LeeWang}, which states that if $f_{\rho}\circ f_{j}^{-1}$ is area-decreasing, then the minimal surface is stable. 
\begin{prop}(Proposition \ref{ministable})
Let $q_2$ be a holomorphic quadratic differential on $\Sigma$. Let $j, \hat j $ be Fuchsian representations which correspond to $q_2, -q_2$ respectively. Suppose $\rho\in \mc{M}_{\text{Betti}}(S)$ is in the $n$-Fuchsian fiber of $\mathcal M_{Higgs}(\Sigma)$ containing $\tau_n\circ \hat j$, then the  $(j,\rho)$-equivariant embedded minimal surface $(f_j,f_{\rho})^+:\ti{\Sigma}\rightarrow \big(\mathbb H^2\times X, g_{\mb{H}^2}+g_n\big)$ is stable.
\end{prop}

In particular, suppose $G$ is a semisimple Lie group of rank $1$, the sectional curvature of the symmetric space $X_G$ associated to $G$ is strictly negative. Denote $g_{-c}$ as the rescaling metric of $g_{X_G}$ such that the maximum of the sectional curvature of $g_{-c}$ is $-c$.
\begin{prop}\label{lab}(Proposition \ref{minirank1})
Let $j$ be a Fuchsian representation and $\rho:\pi_1\rightarrow G$ be an irreducible representation to a reductive Lie group of rank $1$. Suppose $\rho$ does not preserve any geodesic arc in $X_G$. Then for $c\geq 1$, there exists a unique $(j,\rho)$-equivariant minimal surface $f:\ti{S} \rightarrow \big(\mathbb H^2\times X_G,g_{\mb{H}^2}+g_{-c}\big)$.
\end{prop}
\begin{rem}
1. Proposition \ref{lab} is closely related to Labourie's conjecture in \cite{LabourieEnergy} on the uniqueness of equivariant minimal surface for Hitchin representations and maximal representations. Labourie's conjecture is an important problem in higher Teichm\"uller theory and there are lots of studies on it, e.g. \cite{AlessandriniCollier, Collier, CollierTholozanToulisse, LabourieFlat, LabourieCyclic, Loftin}. It still remains open for Hitchin representation into real split Lie groups of rank $\geq 3$ and maximal representations into Hermitian Lie groups of rank $\geq 3$. \\
2. For $G=SL(2,\mb{R})$, $\rho$ being Fuchsian, Proposition \ref{lab} recovers the theorem of Schoen \cite{Schoen}, i.e., Labourie's conjecture holds for Hitchin representations into $PSL(2,\mathbb R)\times PSL(2,\mathbb R)$. For $G=SL(2,\mb{C})$, let $Rep^*(SL(2,\mathbb C))$ denote the space of conjugacy classes of irreducible no-elementary representations of $\pi_1(S)$ into $SL(2,\mathbb C)$. Proposition \ref{lab} implies for each representation $\sigma\in\mathcal T(S)\times Rep^*(SL(2,\mathbb C))$, Labourie's conjecture holds, that is, there exists a unique $\sigma$-equivariant minimal surface in $\mathbb H^2\times \mathbb H^3$. 
\end{rem}
\textbf{Maximal surface:} Suppose $\text{Hopf}(f_j)=\text{Hopf}(f_{\rho})=q_2$, then the product map $(f_j,f_{\rho})^-$ is conformal. If $f_j^*g_{\mb{H}^2}>f_{\rho}^*g_{n}$, the image of $\ti{\Sigma}$ is spacelike.
Together with the harmonicity, $(f_j,f_{\rho})^-$ gives a $(j,{\rho})$-equivariant embedded spacelike maximal surface. We obtain the following proposition by making use of a result in Tholozan \cite{Tholozan} which showed the uniqueness of the conformal class of the maximal surface under the condition that $f_{\rho}\circ f_{j}^{-1}$ is strictly distance-decreasing.
\begin{prop}(Proposition \ref{maximalsurface})
Suppose $\rho\in \mc{M}_{\text{Betti}}(S)$ is in the same Hitchin fiber as $\tau_n\circ j$ in $\mc{M}_{\text{Higgs}}(\Sigma)$. Suppose $\mathbb P(\rho)\neq \mathbb P(\tau_n\circ j)$, then $(f_j,f_{\rho})^{-}:\ti{\Sigma}\rightarrow \big(\mb{H}^2\times X,g_{\mb{H}^2}-g_n\big)$ gives a $(j,\rho)$-equivariant embedded spacelike maximal surface. 

Moreover, the conformal class $[\Sigma]\in \mathcal T(S)$ is unique among all the $(j,\rho)$-equivariant spacelike maximal surfaces.
\end{prop}

\subsection{Key step} The key step in proving our main theorem is establishing an inequality generalizing a theorem of Ness on the adjoint orbit. This result characterizes the critical property of the standard $sl(2,\mb{C})$, which has its own interests in the orbit theory. Consider the function $K_0: sl(n,\mathbb C)\setminus \{0\}\rightarrow \mathbb R$ given by $K_0(A)=\frac{|[A, A^*]|^2}{|A|^4},$ where $|A|^2=\text{tr}(AA^*)$. Denote $\mathcal O_A$ as the $SL(n,\mb{C})$-adjoint orbit of $A$. Restricting $K_0$ on $\mathcal O_A$ for $A$ being nilpotent, Ness in \cite{N} proved the following theorem of geometric invariant theory and the precise statement is in Schmid-Vilonen \cite{SV}. 
\begin{thm}(Ness \cite{N}, Schmid-Vilonen \cite{SV}) 
For a nilpotent matrix $A\in sl(n,\mathbb C)\setminus \{0\}$, $A$ is a critical point of the function $K_0$ on the orbit $\mathcal O_A$ if and only if there exists a real number $a$, $a<0$, such that
\begin{equation*}
[[A, A^*], A]=aA, \quad\text{and}\quad [[A, A^*], A^*]=-aA^*.
\end{equation*}
The set of the critical points is non-empty and consists of a single $SU(n)\times \mathbb C^*$-orbit.

Moreover, the function $K_0$ on $\mathcal O_A$ assumes its minimum value exactly on the critical set.
\end{thm}
However, when the matrix $A$ is not nilpotent, for instance diagonalizable, $K_0$ always has minimum $0$ on the orbit $\mc{O}_A$ and fails to detect any special unitary orbit. We would like to generalize $K_0$ to apply to other orbits. Motivated by the curvature formula of the symmetric space $SL(n,\mb{C})/SU(n)$, we define a function $K:sl(n,\mathbb C)\setminus Z\rightarrow \mathbb{R}$ given by
$$K(A)=\frac{|[A,A^*]|^2}{|A|^4-|\text{tr}(A^2)|^2},$$
where $Z=\{A\in sl(n,\mathbb C):~|A|^4-|\text{tr}(A^2)|^2=0\}$. The function $K$ coincides with Ness' function $K_0$ for $A$ being nilpotent. 

\begin{thm} (Theorem \ref{GeneralizationOfNessTheoremSl2})
Suppose $A\in sl(n,\mathbb C)$ is not conjugate to any element in $W=\{A\in sl(n,\mathbb C):~[A,A^*]=0\}$. Then $A$ is a critical point of the function $K$ on $\mathcal O_A\setminus Z$ if and only if  $A,A^*,[A,A^*]$ generate a three-dimensional Lie subalgebra, which is $SU(n)$-conjugate to a standard $sl(2,\mathbb{C})$.

Moreover, if $A$ is of even Jordan type, the function $K$ on $\mathcal O_A\setminus Z$ assumes its minimum value exactly on the critical set.
\end{thm}

\subsection{Further questions}
Theorem \ref{Theorem1Intro} is closely related to the following conjecture.
\begin{conj}(Dai-Li \cite{DaiLi2})\label{HitchinFiber}
Inside each Hitchin fiber of the moduli space $\mathcal M_{\text{Higgs}}(\Sigma)$, the Hitchin section maximizes the energy density of the corresponding harmonic maps.
\end{conj}
\begin{rem}
Theorem 1.1 actually proves the conjecture for all $n$-Fuchsian fibers. In \cite{DominationFuchsian}, the result of Deroin and Tholozan implied  this conjecture for $n=2$. In \cite{DaiLi2}, the authors showed this conjecture for cyclic $SL(n,\mathbb R)$-Higgs bundle with $n=3,4$.
\end{rem}
As a corollary of Proposition \ref{maximalsurface}, we study the structure of the $n$-Fuchsian fibers of $\tau_n\circ j$ when the Riemann surface varies.

\begin{cor}(Proposition \ref{n-fiber})\label{n-fiberIntro}
Let $\rho\in \mc{M}_{\text{Betti}}(S)$ such that $\mathbb P(\rho)\neq \mathbb P(\tau_n\circ j)$, then there is at most one Riemann surface structure $[\Sigma]\in \mathcal T(S)$ such that $\rho$ is in the same Hitchin fiber of $\tau_n\circ j$ in $\mathcal M_{Higgs}(\Sigma)$.
\end{cor}
We conjecture Proposition \ref{n-fiberIntro} holds for general Hitchin fibers.
\begin{conj}
Let $\hat{\rho}$ be a Hitchin representation. Let $\rho\in \mc{M}_{\text{Betti}}(S)$ such that $\mathbb P(\rho)\neq \mathbb P(\hat \rho)$, then there is at most one Riemann surface structure $[\Sigma]\in \mathcal T(S)$ such that $\rho$ is in the same Hitchin fiber of $\hat{\rho}$ in $\mathcal M_{Higgs}(\Sigma)$.
\end{conj}

\subsection{Organization} In Section \ref{key}, we recall Ness' theorem on the nilpotent orbits in Section \ref{NessTheorem}. And then we generalize this result to the general case in Theorem \ref{GeneralizationOfNessTheoremSl2} in Section \ref{infi}. In Section \ref{keyproof}, we prove Theorem \ref{GeneralizationOfNessTheoremSl2}. In Section \ref{domination results}, under Proposition \ref{IntrinsicCurvatureImpliesEnergyDensity}, we show our main theorem on the domination results of the $n$-Fuchsian representations in Theorem \ref{dommain} in Section \ref{n-Fuchsian}. We also show the domination results in some other cases in Section \ref{others}. In Section \ref{ProofOfProposition}, we prove Proposition \ref{IntrinsicCurvatureImpliesEnergyDensity}. In Section \ref{applications}, we show some applications of the domination results.

\subsection*{Acknowledgement} The second author wants to thank Nicolas Tholozan for the helpful discussion on Ness' theorem in the early stage of this article and to thank Brian Collier for the helpful discussion on the minimal surfaces. The first author is supported by NSF of China (No.11871283 and No.11971244). The second author acknowledges support from Nankai Zhide Foundation.

\section{A Generalization of Ness' Theorem}\label{key}
In this section, we recall some results of a function on nilpotent orbits firstly introduced by Ness. We generalize Ness' function to arbitrary orbits and state a theorem similar to Ness' theorem. This result plays a key role in this article. The proof will be postponed to the next section. We first review some basic knowledge on the relationship between partitions and nilpotent orbits, one can refer to Section 3.1 in the book of Collingwood and McGovern \cite{CollingwoodMcGovern}. 

\subsection{$sl(2, \mathbb C)$ in  $sl(n, \mathbb C)$}\label{liealg}

Set $e=\left(\begin{array}{cc}
0 & 1\\
0 & 0
\end{array}\right), \tilde{e}=\left(\begin{array}{cc}
0 & 0\\
1 & 0
\end{array}\right), x=\left(\begin{array}{cc}
1 & 0\\
0 & -1
\end{array}\right)$. Then $e,\tilde{e},x$ form an $sl(2, \mathbb C)$-triple, that is, they satisfy
$$[e,\tilde{e}]=x,~[x,e]=2e,~[x,\tilde{e}]=-2\tilde{e}.$$

There is a canonical irreducible representation $\tau_n: SL(2,\mathbb{C})\rightarrow SL(n,\mathbb{C})$. It is defined as follows.  Identifying $\mathbb C^2, \mathbb{C}^n$ with the homogeneous polynomials in $(X,Y)$ of degree 1, $n-1$ respectively.
Then $\tau_n$ is defined as the induced action of the natural action of $SL(2,\mb{C})$ on $\mb{C}^2$, that is, for $g\in SL(2,\mathbb{C})$, $\tau_n(g): P(X,Y)\mapsto P(gX,gY)$.

The differential of $\tau_n$ at $I\in SL(2,\mathbb C)$ gives a Lie algebra representation $$j_n:=d\tau_n|_{I}: sl(2,\mathbb{C})\rightarrow sl(n,\mathbb{C}).$$
Choose the basis  of the space of homogeneous polynomials in $(X, Y)$ of degree $n-1$ as $$(X^{n-1},\cdots,\sqrt{C^{k-1}_{n-1}}X^{n-k}Y^{k-1},\cdots,Y^{n-1}).$$
For $A\in sl(2,\mathbb{C})$, then the images of $e, \tilde e, x$ under $j_n: sl(2,\mathbb C)\rightarrow sl(n,\mathbb C)$ are
\begin{eqnarray*}
e_n=\left(
\begin{array}{ccccc}
0 & r_1 & & &\\
& 0 & r_2 & &\\
& & \ddots & \ddots &\\
& & & 0 & r_{n-1}\\
& & & & 0
\end{array}
\right),
\tilde{e}_n=\left(
\begin{array}{ccccc}
0 & & & &\\
r_1 & 0 & & &\\
& r_2 & 0 & &\\
& & \ddots & \ddots &\\
& & & r_{n-1} & 0
\end{array}
\right),
x_n=\left(
\begin{array}{cccccc}
n-1 & & & &\\
& n-3 &  & &\\
& & \ddots & &\\
& & & 3-n &\\
& & & & & 1-n
\end{array}
\right).
\end{eqnarray*}
where $r_k=\sqrt{k(n-k)}$.

\begin{lem}\label{eigen}
Let $M$ be a nonzero element in $ j_n(sl(2,\mathbb{C}))$, then either $M$ has eigenvalues $$\{t(n-1),t(n-3),\cdots,t(3-n),t(1-n)\}, ~t\in \mathbb{C}^*,$$ or $M$ is nilpotent of rank $n-1$.
\end{lem}
\begin{proof}
Let $M=j_n(\tilde{M})$, for $\tilde{M}\in sl(2,\mathbb C)$. Then either $\tilde{M}$ has eigenvalues $\{t,-t\}, t\in\mathbb C^*$ with eigenvector $\tilde{X},\tilde{Y}$, or $\tilde{M}$ is nilpotent and nonzero. In the first case, $\tilde{X}^{n-k}\tilde{Y}^{k-1}$ is the eigenvector of $M$ with eigenvalue $2k-n-1$, for $1\leq k\leq n$. In the latter case, suppose $\tilde{M}\tilde{X}\neq 0$, then $M^i\cdot \tilde{X}^{n-1}$, $i=0,\cdots,n-1$ forms a desired basis such that $M$ is nilpotent of rank $n-1$.
\end{proof}

A \textbf{partition} of $n$ is a non-increasing array $\pi=(n_1,\cdots,n_n)$ of integers  $n_1\geq n_2\geq \cdots\geq n_n$ satisfying $n_i\geq 0, \sum\limits_{p=1}^nn_p=n$. Sometimes we omit the zeros, and use the superscript to denote the multiple, for example $(2,2,1)=(2^2,1)$. Denote $\mathcal{P}_n$ as the space of all partitions of $n$. The space $\mathcal{P}_n$ has a natural partial ordering, called the dominance ordering. Given $\pi=(n_1,\cdots, n_n),\pi^{\prime}=(n_1^{\prime}, \cdots, n_n^{\prime})$ two partitions of $n$, $\pi$ is said to \textbf{dominate} $\pi^{\prime}$ ($\pi\geq\pi^{\prime}$) if for $1\leq p\leq n$, $\sum\limits_{i=1}^pn_i\geq\sum\limits_{i=1}^pn_i^{\prime}.$ For example, in the case $n=4$, $(4)>(3,1)>(2,2)> (2,1,1)> (1,1,1,1)$.

\begin{df}
Given $\pi=(n_1,\cdots, n_s)\in \mathcal P_n$,  the image of $j_{\pi}=(j_{n_1},\cdots,j_{n_s}):sl(2,\mathbb{C})\rightarrow sl(n,\mathbb{C})$ is called the standard $sl(2,\mathbb{C})$ of type $\pi$.
\end{df}

A basis of the standard $sl(2,\mathbb{C})$ of type $\pi$ is given by
\begin{equation*}
E^{\pi}=\text{diag}(e_{n_1},\cdots,e_{n_s}),~\tilde{E}^{\pi}
=\text{diag}(\tilde{e}_{n_1},\cdots,\tilde{e}_{n_s}),~X^{\pi}
=\text{diag}(x_{n_1},\cdots,x_{n_s}).
\end{equation*}
Note that $\{E^{\pi}, \tilde E^{\pi}, X^{\pi}\}$ form a $sl(2,\mathbb C)$-triple. And the matrix $E^{\pi}$ is clearly a nilpotent element of $sl(n,\mathbb C)$. In fact, we have the following well-known result, see \cite{FH} for references.
\begin{prop}\label{SLstandard}
Every $sl(2,\mathbb{C})$ copy in $sl(n,\mathbb{C})$ is $SL(n,\mathbb{C})$-conjugate to a standard $sl(2,\mathbb{C})$.
\end{prop}
Denote $\Lambda_{n}=(n-1,n-3,\cdots,3-n,1-n)$. Then from Lemma \ref{eigen} the eigenvalues of a non-nilpotent element in a standard $sl(2,\mathbb C)$ has the form $c(\Lambda_{n_1},\cdots,\Lambda_{n_s})$ for some $c\in\mathbb C^*$. It may happen that two standard $sl(2,\mathbb C)$ from different partitions give the same form up to a factor.
\begin{lem}\label{2crit}
Suppose the eigenvalues of $A\in sl(n,\mathbb C)$ can be expressed in more than one way as $c(\Lambda_{n_1},\cdots,\Lambda_{n_s})$ for some $c\in\mathbb C^*$. Then it must be the case
$$c(\Lambda_{2m_1+1},\cdots,\Lambda_{2m_s+1}) \quad\text{and}\quad 2c(\Lambda_{m_1},\Lambda_{m_1+1},\cdots,\Lambda_{m_s},\Lambda_{m_s+1}).$$
\end{lem}
\begin{proof} Suppose the eigenvalues of $A$ can be expressed as $c(\Lambda_{n_1},\cdots,\Lambda_{n_s})$ and $c^{\prime}(\Lambda_{n_1^{\prime}},\cdots,\Lambda_{n_s^{\prime}})$ for some $c, c^{\prime}\in\mathbb C^*$.  Consider $d=\min\limits_{\lambda_i\neq \lambda_j}|\lambda_i-\lambda_j|$, $\lambda_i$'s are the elements of $(\Lambda_{n_1},\cdots,\Lambda_{n_s})$. Then $d=1$ or $2$. $d^{\prime}$ is similarly defined. Then $cd=c^{\prime}d^{\prime}$. It is easy to see $c\neq c^{\prime}$. So we assume $d=2$, $d^{\prime}=1$. Since $d^{\prime}=1$, we see $0$ is in $(\Lambda_{n_1^{\prime}},\cdots,\Lambda_{n_s^{\prime}})$. So $0$ is also in $(\Lambda_{n_1},\cdots,\Lambda_{n_s})$. Together with $d=2$, the elements in $(\Lambda_{n_1},\cdots,\Lambda_{n_s})$ are all even, in other words, $n_i$'s are all odd. Let $n_i=2m_i+1$. Then $(\Lambda_{2m_1+1},\cdots,\Lambda_{2m_s+1})$ is uniquely expressed as $2(\Lambda_{m_1},\Lambda_{m_1+1},\cdots,\Lambda_{m_s},\Lambda_{m_s+1})$. We finish the proof.
\end{proof}
\subsection{Nilpotent orbits and Ness' theorem}\label{NessTheorem}

Recall the Cartan decomposition of $sl(n,\mathbb C)$ is $sl(n,\mathbb C)=su(n)\oplus \sqrt{-1} su(n)$ and the Cartan involution is $\sigma(X)=-X^*$, where $X^*=\overline X^T$. Using the rescaled Killing form $B(X, Y)=\text{tr}(XY)$ on $sl(n,\mathbb R)$ and the Cartan involution, we then have an $SU(n)$-invariant Hermitian inner product on $sl(n,\mathbb C)$ by
\begin{equation*}
\big<X, Y\big>=-B(X, \sigma(Y))=\text{tr}(XY^*),\quad \text{for} ~X, Y\in sl(n,\mathbb C).
\end{equation*} As usual, $|X|^2$ denotes $\big<X,X\big>$.
Ness in \cite{N} defined a map $m: sl(n, \mathbb C)\rightarrow \sqrt{-1}su(n)$ by
\begin{equation*}
\big<m(\xi),\eta\big>=\frac{1}{2|\xi|^2}(\frac{d}{dt}|Ad(\exp(t\eta)\xi|^2)|_{t=0}\quad \text{for $\xi, \eta\in sl(n,\mathbb C)$},
\end{equation*} which measures the change of the square norm of a vector under the adjoint action. 
Ness in \cite{N} showed that $\sqrt{-1}m: sl(n,\mathbb C)\rightarrow su(n)$ is the moment map for the induced action of $SU(n)$ on $\mathbb P(sl(n,\mathbb C))$. One may consider the function $K_0: sl(n,\mathbb C)\rightarrow \mathbb R$ given by
\begin{equation*}
K_0(A)=|m(A)|^2=\frac{|[A, A^*]|^2}{|A|^4}.
\end{equation*}

Denote by $\mathcal N$ the space of nilpotent matrices inside $sl(n,\mathbb C)$ and by $\mathcal O_A$ the adjoint orbit of $A\in sl(n, \mathbb C)$. Ness proved the following theorem.
\begin{thm}(Theorem 6.1 and 6.2 in Ness \cite{N} and Lemma 2.11 in Schmid-Vilonen \cite{SV})\label{Ness}
For a nilpotent matrix $A\in sl(n,\mathbb C)$, $A\neq 0$,\\
(1) $A$ is a critical point of the function $K_0$ on its adjoint orbit $\mathcal O_A$ if and only if there exists a real number $a$, $a<0$, such that
\begin{equation}\label{Rigidity}
[[A, A^*], A]=aA, \quad\text{and}\quad [[A, A^*], A^*]=-aA^*.
\end{equation}
The set of the critical points is non-empty and consists of a single $SU(n)\times \mathbb C^*$-orbit.
\\
(2) The function $K_0$ on $\mathcal O_A$ achieves its minimum value exactly on the critical set.
\end{thm}
We will provide a proof of Theorem \ref{Ness} in Section \ref{ProofNess} which is different from the original proof in Ness \cite{N}.

For each $\pi=(n_1,\cdots,n_k)\in \mathcal{P}_n$, we associate a constant \begin{equation*}
C_{\pi}:=K_0(E^{\pi})
=\frac{12}{\sum\limits_{p=1}^kn_p(n_p^2-1)}.
\end{equation*}
The constant $C_{\pi}$ has monotonicity with respect to the partial order of $\pi$, which is proved in \cite{QLnil}.
\begin{lem}\label{ql1}
If $\pi_1, \pi_2\in\mathcal P_n$ satisfy $\pi_1<\pi_2$, then the constants satisfy $C_{\pi_1}> C_{\pi_2}$.
\end{lem}
\begin{df}
For a nilpotent matrix $A\in sl(n,\mathbb C)$, we say it is of \textbf{Jordan type} $\pi\in\mathcal P_n$ if the block sizes of $A$'s Jordan normal form give the partition $\pi$ of $n$.
\end{df}

Theorem \ref{Ness} gives a lower bound of $K_0$ with respect to the Jordan type.
\begin{prop}\label{JordanInequality}
Suppose $A\in \mathcal N$ is of Jordan type at most $\pi\in \mathcal{P}_n$, then
$$K_0(A)\geq C_{\pi}, $$
and equality holds if and only $A$ is $SU(n)$-conjugate to $c\cdot E^{\pi}$, for some constant $c\in\mathbb C^*$.
\end{prop}
\begin{proof}
Apply Theorem \ref{Ness} to our case that $A$ is nilpotent of Jordan type $\pi^{\prime}\in \mathcal{P}_n$ for some $\pi^{\prime}\leq \pi$. Since $E^{\pi^{\prime}}$ satisfies Equation (\ref{Rigidity}), we obtain that all the minimum points are $SU(n)$-conjugate to $c\cdot E^{\pi^{\prime}}$ for some constant $c\in \mathbb C^*$ and hence $K_0(A)\geq K_0(E^{\pi^{\prime}})=C_{\pi^{\prime}}.$ From the monotonicity in Lemma \ref{ql1}, we have $C_{\pi^{\prime}}\geq C_{\pi}$ and hence $K_0(A)\geq C_{\pi}.$ The rigidity also follows from Theorem \ref{Ness}.
\end{proof}

Note that among $\mathcal{P}_n$, $\lambda=(n)$ is the absolute maximum. Therefore, we have an immediate corollary of Proposition \ref{JordanInequality}.
\begin{cor}\label{nilmax}
For every $A\in \mathcal N$, we have
$$K_0(A)\geq C_{(n)}=\frac{12}{n(n^2-1)}, $$
and equality holds if and only if $A$ is $SU(n)$-conjugate to $c\cdot e_n$, for some constant $c\in \mathbb C^*$.
\end{cor}

\subsection{From nilpotent elements to $sl(2,\mathbb C)$-copies}\label{infi}
\begin{df}
Let $\mathfrak{s}$ be an $sl(2,\mathbb{C})$-copy in $sl(n,\mathbb{C})$. The Jordan type of $\mathfrak{s}$ is defined to be the Jordan type of the nilpotent elements in $\mathfrak{s}$.
\end{df}

From the previous section, the minimum of the function $K_0(A)$ is powerful detecting an $SU(n)$-orbit inside an $SL(n,\mathbb C)$-orbit of a nilpotent matrix.  However, when the matrix $A$ is not nilpotent, for instance diagonalizable, $K_0$ always has minimum $0$ on the orbit $\mc{O}_A$ and fails to detect any special unitary orbit. We would like to generalize $K_0$ to apply to other orbits. Motivated by the curvature formula of the symmetric space $SL(n,\mb{C})/SU(n)$, see Lemma \ref{CurvatureAlgebraicFunction}, we define a function $K:sl(n,\mathbb C)\setminus Z\rightarrow \mathbb{R}$,
\begin{equation}\label{FunctionK}
K(A)=\frac{|[A,A^*]|^2}{|A|^4-|\text{tr}(A^2)|^2},
\end{equation}
where $Z=\{A\in sl(n,\mathbb C):~|A|^4-|\text{tr}(A^2)|^2=0\}$. The function $K(A)$ coincides with Ness' function $K_0(A)=\frac{|[A,A^*]|^2}{|A|^4}$ for $A$ being nilpotent.
By the Cauchy inequality, $|\big<A,A^*\big>|\leq |A||A^*|$ and so $|A|^4-|\text{tr}(A^2)|^2\geq 0$. Then $K$ takes nonnegative value in $\mathbb R$. Denote $W=\{A\in sl(n,\mathbb C):~[A,A^*]=0\}$.
\begin{lem}\label{uni-real}
(1) The function $K$ is invariant under scaling and unitary conjugation. \\
(2) $W=\{U^{-1}\text{diag}(\lambda_1,\cdots,\lambda_n)U, ~\lambda_1,\cdots,\lambda_n\in \mathbb{C}, U\in SU(n)\}$ \\
(3) $Z=\{U^{-1}\text{diag}(c\lambda_1,\cdots,c\lambda_n)U, ~\lambda_1,\cdots,\lambda_n\in \mathbb{R}, ~c\in \mathbb{C}, U\in SU(n)\}$.
\end{lem}
\begin{proof}
Part (1) is by direct calculation. Part (2) is from basic linear algebra. For Part (3), by the Cauchy inequality, if $A\in Z$, then there exists $a\in \mathbb C$ such that $A=aA^*$.  So $|a|^2=1$ and set $a=e^{2i\theta}$. Then $e^{-i\theta}A=(e^{-i\theta}A)^*$ for $A\in Z$. Then from basic linear algebra we finish the proof.
\end{proof}

To make sure $\mathcal{O}_A\setminus Z$ is not empty, we always need to assume $A$ is not a scalar matrix.
\begin{df}
Let $(\lambda_1,\cdots,\lambda_n)\in \mathbb{C}^n$. If $(\lambda_1,\cdots,\lambda_n)=c(\mu_1,\cdots,\mu_n)$ for $\mu_i\in \mathbb{R}, i=1,\cdots, n,$ and $c\in \mathbb{C}$, then we call $(\lambda_1,\cdots,\lambda_n)$ is uni-real.
\end{df}
Fix $A\in sl(n,\mathbb C)$, we study the minimum (infimum) of the function $K$ on the adjoint orbit $\mathcal{O}_A\setminus Z$. From Lemma \ref{uni-real}, $\min\limits_{\mathcal{O}_A\setminus Z} K=0$ if and only if $A$ is diagonalizable and the eigenvalues of A are not uni-real.

\begin{df}
A partition $\pi=(n_1,\cdots,n_n)\in\mathcal P_n$ is said to be even, if $n_1,\cdots, n_n$ have the same parity. Correspondingly, a nilpotent matrix $A$ or an $\text{sl}(2,\mathbb{C})$-copy $\mathfrak{s}$ in $\text{sl}(n,\mathbb{C})$ is said to be even, if it is of Jordan type of an even $\pi\in \mathcal P_n$. A partition $\pi=(n_1,\cdots,n_n)\in\mathcal P_n$ is said to be odd, if it is not even.
\end{df}
We generalize Theorem \ref{Ness} as follows and postpone the proof in the next section.
\begin{thm}\label{GeneralizationOfNessTheoremSl2}
Let $A\in sl(n,\mathbb{C})$, which is not a scalar matrix. Suppose $A$ is not $SL(n,\mb{C})$-conjugate to an element in $W\setminus Z$. Then

(1) $A$ is a critical point of the function $K$ in $\mathcal O_A\setminus Z$ if and only if  $A,A^*,[A,A^*]$ generate a three-dimensional Lie subalgebra $\mathfrak{s}$, which is $SU(n)$-conjugate to a standard $sl(2,\mathbb{C})$.

(2) On the standard $sl(2,\mathbb{C})$ of type $\pi\in \mathcal P_n$, the function $K\equiv C_{\pi}$ outside $Z$.

(3) Suppose $A$ is nilpotent, then the function $K$ on $\mathcal O_A\setminus Z$ achieves its minimum value at $A$ if and only if $A$ is $SU(n)$-conjugate to an element in a standard $sl(2,\mathbb{C})$.

(4) Suppose $A$ is not nilpotent, then the function $K$ on $\mathcal O_A\setminus Z$ achieves its minimum value at $A$ if and only if $A$ is $SU(n)$-conjugate to an element in an even standard $sl(2,\mathbb{C})$.
\end{thm}

From Theorem \ref{GeneralizationOfNessTheoremSl2}, analogous to Proposition \ref{JordanInequality}, we have the following proposition.
\begin{prop}\label{typeKestimate}
For $A\in sl(n,\mathbb C)$, $A\notin Z$, if $A$ is conjugate to an element in a standard $sl(2,\mathbb C)$ whose Jordan type is even and at most $\pi\in \mathcal P_n$, then $K(A)\geq C_{\pi}$. 

Equality holds if and only if $\pi$ is even and $A$ is $SU(n)$-conjugate to an element in the standard $sl(2,\mathbb C)$ of Jordan type $\pi$.
\end{prop}
\begin{proof}
Suppose $A$ is conjugate to an element in the standard $sl(2,\mathbb C)$ of even Jordan type $\pi^{\prime}\in \mathcal P_n$. From the part (2) and (4) in the Theorem \ref{GeneralizationOfNessTheoremSl2}, we see that this standard $sl(2,\mathbb C)$ achieves the minimum of $K$, which is $C_{{\pi}^{\prime}}$. Since $\pi\geq \pi^{\prime}$, together with Lemma \ref{ql1}, we obtain $K(A)\geq C_{\pi^{\prime}}\geq C_{\pi}$. The rigidity also follows from Theorem \ref{GeneralizationOfNessTheoremSl2}.
\end{proof}

Since $(n)\in \mathcal P_n$ is the absolute maximum and is even, 
we have an immediately corollary.
\begin{cor}
For $A\in sl(n,\mathbb C)\setminus Z$, if $A$ is conjugate to an element in a standard $sl(2,\mathbb C)$ whose Jordan type is even, then $K(A)\geq C_{(n)}=\frac{12}{n(n^2-1)}$. 

Equality holds if and only if $A$ is $SU(n)$-conjugate to an element in the standard $sl(2,\mathbb C)$ of Jordan type $(n)$.
\end{cor}
From Lemma \ref{eigen}, we have the following corollary which will be used later.
\begin{cor}\label{maincor}
Let $A\in sl(n,\mathbb C)\setminus Z$. Suppose $A$ has the same eigenvalues as $te_n+\ti{e}_n$ for some $t\in \mb{C}$. Then $K(A)\geq C_{(n)}=\frac{12}{n(n^2-1)}$.
\end{cor}
If the rank of $A$ is at most $2$, we have the following corollary.
\begin{cor}\label{rank2Inequality} For $A\in sl(n,\mathbb C)\setminus Z$, if $A$ is of rank at most $2$, then $K(A)\geq \frac{1}{2}.$

Equality holds if and only if $A$ is $SU(n)$-conjugate to an element in the standard $sl(2,\mathbb C)$ of Jordan type $(3,1,\cdots, 1)\in \mathcal P_n$.
\end{cor}
\begin{proof}
Since $A$ is at most rank $2$ and $\text{tr}A=0$, the Jordan normal form has the following types:
$$
\text{diag}(J_3,0,\cdots,0);~\text{diag}(J_2,J_2,0,\cdots,0);~
\text{diag}(\lambda,-\lambda,0,\cdots,0),\lambda\neq 0;~\text{diag}(J_2,0,\cdots,0);~0.
$$
Notice that if $A$ is in the orbit of $\text{diag}(\lambda,-\lambda,0,\cdots,0),\lambda\neq 0$, i.e. $\frac{\lambda}{2}\text{diag}(\text{diag}(2,0,-2),0,\cdots,0),\lambda\neq 0$, then $A$ is in an even $sl(2,\mathbb C)$ copy of Jordan type $(3,1,\cdots,1)$. From Proposition \ref{typeKestimate}, $K(A)\geq C_{(3,1,\cdots,1)}$. We have the similar estimates in the other nilpotent cases. Together with Lemma \ref{ql1}, noticing $(3,1,\cdots,1)>(2,2,1,\cdots,1)$, we obtain $K(A)\geq C_{(3,1,\cdots,1)}=\frac{1}{2}$. The rigidity also follows from Theorem \ref{GeneralizationOfNessTheoremSl2}.
\end{proof}

\section{Proof of Theorem \ref{GeneralizationOfNessTheoremSl2}}\label{keyproof}
We prove Theorem \ref{GeneralizationOfNessTheoremSl2} by considering the infimum of the function $K$. We consider the following three kinds of candidates of the infimum of $K$ on $\mathcal{O}_A\setminus Z$:\\
(a) the critical values in the interior of $\mathcal{O}_A\setminus Z$,\\
(b) the inferior limit when $A$ approaches to $Z$,\\
(c) the inferior limit when $A$ approaches to the boundary of $\mathcal{O}_A$ or infinity.

\subsection{Critical points of $K$ in $\mathcal{O}_A\setminus Z$}\label{sec crit}
First we calculate the values of $K$ on a standard $sl(2,\mathbb{C})$.
\begin{lem}\label{C(K)}
On the standard $sl(2,\mathbb{C})$ of type $\pi\in \mathcal P_n$, the function $K\equiv C_{\pi}$ outside $Z$.
\end{lem}
\begin{proof} 
A basis of the standard $sl(2,\mathbb{C})$ of type $\pi=(n_1,\cdots,n_s)$ is
$$E=\text{diag}(e_{n_1},\cdots,e_{n_s}),~\tilde{E}=\text{diag}(\tilde{e}_{n_1},\cdots,\tilde{e}_{n_s}),~X=\text{diag}(x_{n_1},\cdots,x_{n_s}).$$
Then $E,\tilde{E},X$ are orthogonal to each other and
\begin{eqnarray*}
|E|^2=|\tilde{E}|^2=\sum\limits_{i=1}^{s}|e_{n_i}|^2=\sum\limits_{i=1}^{s}\sum\limits_{k=1}^{n_i} k(n_i-k)=\sum\limits_{i=1}^{s}(\frac{n_i^2(n_i+1)}{2}-\frac{n_i(n_i+1)(2n_i+1)}{6})=\sum\limits_{i=1}^{s}\frac{n_i^3-n_i}{6},\\
|X|^2=\sum\limits_{i=1}^{s}|x_{n_i}|^2=\sum\limits_{i=1}^{s}\sum\limits_{k=1}^{n_i} (n_i-2k+1)^2
=\sum\limits_{i=1}^{s}(-(n_i+1)^2n_i+\frac{2n_i(n_i+1)(2n_i+1)}{3})=\sum\limits_{i=1}^{s}\frac{n_i^3-n_i}{3}.
\end{eqnarray*}
Let $A=aE+b\tilde{E}+cX$, $a,b,c\in \mathbb{C}$. Then $A^*=\bar{b}E+\bar{a}\tilde{E}+\bar{c}X$. And
\begin{eqnarray*}
[A,A^*]&=&(|a|^2-|b|^2)[E,\tilde{E}]+(a\bar{c}-c\bar{b})[E,X]+(b\bar{c}-c\bar{a})[\tilde{E},X]\\
&=&2(c\bar{b}-a\bar{c})E+2(b\bar{c}-c\bar{a})\tilde{E}+(|a|^2-|b|^2)X,\\
|[A,A^*]|^2&=&4(|c\bar{b}-a\bar{c}|^2+|b\bar{c}-c\bar{a}|^2)|E|^2+(|a|^2-|b|^2)^2|X|^2=2(4|a\bar{c}-\bar{b}c|^2+(|a|^2-|b|^2)^2)|E|^2.\\
|A|^2&=&(|a|^2+|b|^2)|E|^2+|c|^2|X|^2=(|a|^2+|b|^2+2|c|^2)|E|^2,\\
\big<A,A^*\big>&=&c^2|X|^2+2ab|E|^2=2(c^2+ab)|E|^2,\\
|A|^4-|\big<A,A^*\big>|^2&=&(|a|^2+|b|^2+2|c|^2)^2|E|^4-4|c^2+ab|^2|E|^4=(4|a\bar{c}-\bar{b}c|^2+(|a|^2-|b|^2)^2)|E|^4.
\end{eqnarray*}
So $K(A)=\frac{|[A,A^*]|^2}{|A|^4-|\big<A,A^*\big>|^2}=\frac{2}{|E|^2}=C_{\pi}$, outside the points such that $4|a\bar{c}-\bar{b}c|^2+(|a|^2-|b|^2)^2=0$ which lie in $Z$.
\end{proof}
Now we calculate the variation formula. The following Lemma is useful, whose proof is by direct calculation. Recall the Hermitian inner product is defined as $\big<X, Y\big>=\text{tr}(XY^*)$.
\begin{lem}\label{basic0}
$\big<[A,X],Y\big>=\big<X,[A^*,Y]\big>$, $[X,Y]^*=[Y^*,X^*]$, $\big<X^*,Y^*\big>=\overline{\big<X,Y\big>}$.
\end{lem}

The following inequality is the key to characterize the critical point.
\begin{lem}\label{ineq}
Let $X_1,X_2,X_3\in \mathbb{C}^m$ and $\big<\cdot,\cdot\big>$ be the standard Hermitian inner product. Then
\begin{eqnarray}\label{Equation0}
&&|X_1|^2|X_2|^2|X_3|^2+2\text{Re}(\big<X_1,X_2\big>\big<X_2,X_3\big>\big<X_3,X_1\big>)\nonumber\\
&\geq& |X_1|^2|\big<X_2,X_3\big>|^2+|X_2|^2|\big<X_3,X_1\big>|^2+|X_3|^2|\big<X_1,X_2\big>|^2.
\end{eqnarray}
Equality holds if and only if $X_1,X_2,X_3$ are linearly dependent.
\end{lem}
\begin{proof} Consider the term
$X_1\wedge X_2\wedge X_3=\sum\limits_{\sigma\in\mathcal S_3}(-1)^{sign(\sigma)}X_{\sigma(1)}\otimes X_{\sigma(2)}\otimes X_{\sigma(3)}.$  The inequality follows from $|X_1\wedge X_2\wedge X_3|^2\geq 0$. In fact
\begin{eqnarray*}
&&\big<X_1\wedge X_2\wedge X_3,X_1\wedge X_2\wedge X_3\big>\\
&=&\big<\sum_{\sigma\in\mathcal S_3}(-1)^{sign(\sigma)}X_{\sigma(1)}\otimes X_{\sigma(2)}\otimes X_{\sigma(3)}, \sum_{\tau\in\mathcal S_3}(-1)^{sign(\tau)}X_{\tau(1)}\otimes X_{\tau(2)}\otimes X_{\tau(3)}\big>\\
&=&6\big<X_1\otimes X_2\otimes X_3, \sum_{\tau\in\mathcal S_3}(-1)^{sign(\tau)}X_{\tau(1)}\otimes X_{\tau(2)}\otimes X_{\tau(3)}\big>\\
&=&6\big<X_1\otimes X_2\otimes X_3,X_1\otimes X_2\otimes X_3-X_1\otimes X_3\otimes X_2+X_2\otimes X_3\otimes X_1-X_2\otimes X_1\otimes X_3\\&&+X_3\otimes X_1\otimes X_2-X_3\otimes X_2\otimes X_1\big>\\
&=&6(|X_1|^2|X_2|^2|X_3|^2+2\text{Re}(\big<X_1,X_2\big>\big<X_2,X_3\big>\big<X_3,X_1\big>)-|X_1|^2|\big<X_2,X_3\big>|^2\\&&-|X_2|^2|\big<X_3,X_1\big>|^2-|X_3|^2|\big<X_1,X_2\big>|^2).
\end{eqnarray*}
So we show the inequality. And equality holds if and only if $X_1\wedge X_2\wedge X_3=0$ if and only if $X_1,X_2, X_3$ are linearly dependent.
\end{proof}

Now we show the characterization of the critical points of $K$.
\begin{prop}\label{critical}
Let $A\in sl(2,\mathbb{C})$, $A\notin Z$, $[A,A^*]\neq 0$. The following statements are equivalent.\\
(1) The point $A$ is a critical point of the function $K$ on its orbit $\mathcal O_A\setminus Z$;\\ 
(2) $A,A^*,[A,A^*]$ generate a three-dimensional Lie subalgebra $\mathfrak{s}$, which is $SU(n)$-conjugate to a standard $sl(2,\mathbb{C})$.
\end{prop}
\begin{proof}
First we show (1) implies (2).

Step 1: We claim that if $A$ is a critical point of $K$ on $\mathcal O_A\setminus Z$, then $A,A^*,[A,[A,A^*]]$ are linearly dependent. Consider a family $A_t=T^{-1}_tAT_t$ in $\mathcal O_A$, where $T_0=Id$, $\frac{d}{dt}\Big|_{t=0}T_t=M$. Then $H:=\frac{d}{dt}\Big|_{t=0}A_t=[A,M]$.  By using Lemma \ref{basic0},
\begin{eqnarray*}
\frac{d}{dt}\Big|_{t=0}[A_t,A_t^*]&=&[\frac{d}{dt}\Big|_{t=0}A_t, A^*]+[A, \frac{d}{dt}\Big|_{t=0}A_t^*]=[H,A^*]+[A,H^*],\\
\frac{d}{dt}\Big|_{t=0}|[A_t,A_t^*]|^{2}&=&2\text{Re}\big<\frac{d}{dt}\Big|_{t=0}[A_t,A_t^*], [A,A^*]\big>=2\text{Re}\big<[H,A^*]+[A,H^*],[A,A^*]\big>\\&=&4\text{Re}\big<[A,[A^*,A]],H\big>,\\
\frac{d}{dt}\Big|_{t=0}|A_t|^{4}&=&2|A|^2\cdot \frac{d}{dt}\Big|_{t=0}|A_t|^{2}=4|A|^2\text{Re}\big<A,\frac{d}{dt}\Big|_{t=0}A_t\big>=4\text{Re}\big<|A|^2A,H\big>,\\
\frac{d}{dt}\Big|_{t=0}|\text{tr}(A_t^2)|^2&=&\frac{d}{dt}\Big|_{t=0}|\text{tr}(A^2)|^2=0.
\end{eqnarray*}
If $A$ is a critical point of $K(A)$ on $\mathcal O_A\setminus Z$, then
\begin{eqnarray*}
\frac{d}{dt}\Big|_{t=0}K(A_t)=\frac{1}{(|A|^4-|\text{tr}(A^2)|^2)^2}\big(4\text{Re}\big<[A,[A^*,A]]
,H\big>(|A|^4-|\text{tr}(A^2)|^2)-|[A,A^*]|^24\text{Re}\big<|A|^2A,H\big>\big).
\end{eqnarray*}
Since $A$ is a critical point, we have
\begin{equation*}
4\text{Re}\big<H,~(|A|^4-|\text{tr}(A^2)|^2)[A,[A^*,A]]-|[A,A^*]|^2|A|^2A\big>=0.
\end{equation*}
Recall $H=[A,M]$, then
\begin{equation}\label{Variationfomula}
\text{Re}\Big<M,~~~(|A|^4-|\text{tr}(A^2)|^2)[A^*,[A,[A^*,A]]]-|[A,A^*]|^2|A|^2[A^*,A]\Big>=0.
\end{equation}
Now $M$ is an arbitrary matrix in $sl(n,\mathbb{C})$, set $M=[A,A^*]$ and we obtain from Equation (\ref{Variationfomula}),
\begin{equation*}
4\text{Re}\big<[A, [A,A^*]],~(|A|^4-|\text{tr}(A^2)|^2)[A,[A^*,A]]-|[A,A^*]|^2|A|^2A\big>=0.
\end{equation*}
Using Lemma \ref{basic0}, we have
\begin{equation}\label{Equation2}(|A|^4-|\text{tr}(A^2)|^2)|[A,[A^*,A]]|^2=|[A,A^*]|^4|A|^2.\end{equation}
Applying Lemma \ref{ineq}, letting $X_1=[A,[A^*,A]]$, $X_2=A$, $X_3=A^*$, inequality (\ref{Equation0}) becomes
$$|A|^4|[A,[A^*,A]]|^2\geq |\text{tr}(A^2)|^2|[A,[A^*,A]]|^2+|[A,A^*]|^4|A|^2.$$ and \begin{equation}\label{Equation3}
(|A|^4-|\text{tr}(A^2)|^2)|[A,[A^*,A]]|^2\geq|[A,A^*]|^4|A|^2.
\end{equation}
Comparing with Equation (\ref{Equation2}), the equality of Equation (\ref{Equation3}) holds, meaning that there exist $a, b, c\in \mathbb{C}$ not all vanishing, such that
\begin{equation*}
a[A,[A^*,A]]+bA+cA^*=0.
\end{equation*}
So we finish the proof of Step 1.

Step 2: We claim that $A, A^*$ generate a $3$-dimensional Lie subalgebra $\mathfrak s$. If $a=0$, then $A\in Z$. So $a\neq 0$, which means $[A,[A^*,A]]\in \text{span}\{A,A^*\}$. By conjugation, we have $[A^*,[A^*,A]]$ also lies in the vector space spanned by $\{A,A^*\}$. So $A,A^*,[A,A^*]$ generate a Lie subalgebra $\mathfrak{s}$ which is spanned by $\{A,A^*,[A,A^*]\}$ as a complex vector space.
To see it is three-dimensional, if not, then there exists $a, b, c\in \mathbb{C}$ not all vanishing, such that $a[A^*,A]+bA+cA^*=0.$ By Lemma \ref{basic0}, $\big<[A^*,A],A\big>=\big<[A^*,A],A^*\big>=0$, so we have $a|[A^*,A]|^2=0$. Since $a\neq0$, we have $[A^*,A]=0$, contradiction. So we finish the proof of the Step 2.

Step 3: We claim that there is a complex Lie algebra representation $\rho: sl(2,\mathbb{C})\rightarrow sl(n,\mathbb{C})$ commuting with the conjugate transpose operator $*$, such that the image of $\rho$ is the three-dimensional Lie subalgebra $\mathfrak{s}$ in Step 2. First we notice that since $\mathfrak{s}$ is $*$-invariant, it must contain a nonzero Hermitian matrix $M$ (for example $[A,A^*]$). Consider the adjoint representation $\text{ad}(M)$ on $\mathfrak{s}$. Since $M$ is Hermitian, we see $\text{ad}(M)$ is also Hermitian and hence has only real eigenvalues. Suppose $[M,P]=\lambda P$, where $P\in \mathfrak{s}$, $\lambda\in \mathbb{R}$. Since $[A,A^*]\neq 0$, $\mathfrak{s}$ is not commutative. We can choose $\lambda\neq 0$. By rescaling $M$, we may assume $\lambda=2$. Then $[M,P]=2P$, $[M,P^*]=-2P^*$ and $[M,M]=0$. Since $P, P^*, M$ are eigenvectors of distinct eigenvalues of $\text{ad}(M)$ and form a basis of $\mathfrak{s}$. So we assume $[P,P^*]=aM+bP+cP^*$. Taking $*-$operation on both sides, we have $a\in \mathbb{R}$ and $b=\bar{c}$. By using the Jacobian identity $[M,[P,P^*]]+[P,[P^*,M]]+[P^*,[M,P]]=0$, we see $bP-\bar{b}P^*=0$. If $b\neq 0$, then $[P,P^*]=0$, which implies $b=0$. So $b$ must be zero and $[P,P^*]=aM$. Then
$$a|M|^2=\big<[P,P^*],M \big>=\big<P^*,[P^*,M] \big>=2|P^*|^2.$$
It is obvious that $a>0$ and by rescaling $P$ we may assume $a=1$. Letting $\rho(e)=P$, $\rho(\tilde{e})=P^*$, $\rho(x)=M$, and by complex linear extension, we obtain a representation $\rho: sl(2,\mathbb{C})\rightarrow sl(n,\mathbb{C})$. From the construction, it is clear that $\rho$ is a $*$-equivariant complex Lie algebra representation. So we finish the proof of the Step 3.

Step 4: We claim that $\mathfrak{s}$ is $SU(n)$-conjugate to a standard $sl(2,\mathbb{C})$ and finish the proof of the direction from (1) to (2). From the terminology of Sekiguchi \cite{Se}, $(M,P,P^*)$ is a strictly normal S-triple. Then by Lemmas 1.4 and 1.5 in \cite{Se}, our claim holds. For the convenience of the readers, we give a proof briefly in our setting. Since $\rho$ is $*$-equivariant, the orthogonal complement of a $*$-invariant subspace is also $*$-invariant. So we may assume $\rho$ is irreducible. From Proposition \ref{SLstandard}, there exists $g\in SL(n,\mathbb{C})$ such that $\rho(v)=g^{-1}j(v)g$ for every $v\in sl(2,\mathbb{C})$, where $j$ is the canonical one. Since $\rho$ and $j$ are both $*$-equivariant, we have $$\rho(v^*)=(\rho(v))^*=(g^{-1}j(v)g)^*=g^*(j(v))^*(g^*)^{-1}=g^*j(v^*)(g^*)^{-1}=g^*g\rho(v^*)g^{-1}(g^*)^{-1}.$$
So $\rho (g^*g)=(g^*g)\rho$. Since $\rho$ is irreducible, then by Schur's Lemma, $g^*g=\lambda I$ for some $\lambda\in \mathbb{C}$. By taking trace, we see $\lambda$ is a positive real number. So by rescaling we may assume $g^*g=I$, which means $g\in U(n)$. Since $g\in SL(n,\mathbb C)$, $g\in SU(n)$. So we finish the proof of this direction.

Next we show (2) implies (1). From Equation \ref{Variationfomula}, we want to show
\begin{equation}\label{Variationfomula1}
(|A|^4-|\text{tr}(A^2)|^2)[A^*,[A,[A^*,A]]]-|[A,A^*]|^2|A|^2[A^*,A]=0.
\end{equation}
Notice that this equation is invariant under $SU(n)$ adjoint action. We may assume $A$ is in a standard $sl(2,\mathbb{C})$. As the calculation in Lemma \ref{C(K)}, we have
\begin{\eq}
A&=&aE+b\ti{E}+cX,\\
~[A,A^*]&=&2(c\bar{b}-a\bar{c})E+2(b\bar{c}-c\bar{a})\tilde{E}+(|a|^2-|b|^2)X,\\
~[A,[A^*,A]]&=&2(|a|^2a-|b|^2a-2\bar{b}c^2+2|c|^2a)E+2(-|a|^2b+|b|^2b+2|c|^2b-2\bar{a}c^2)\ti{E}\\
&&+2(-2ab\bar{c}+|a|^2c+|b|^2c)X,\\
~[A^*,[A,[A^*,A]]]&=&2(|a|^2+|b|^2+2|c|^2)\big(2(c\bar{b}-a\bar{c})E
+2(b\bar{c}-c\bar{a})\tilde{E}+(|a|^2-|b|^2)X\big).
\end{\eq}
From Lemma \ref{C(K)}, we have $|[A,A^*]|^2=C_{\pi}(|A|^4-|\text{tr}(A^2)|^2)$ and $|A|^2=\frac{2(|a|^2+|b|^2+2|c|^2)}{C_{\pi}}$. So Equation (\ref{Variationfomula1}) follows. We finish the whole proof.
\end{proof}

\subsection{Inferior limit at $Z$}\label{Z}
Let $A\in sl(n,\mathbb C)$. Suppose $A$ is not nilpotent. We will discuss the nilpotent case in Section \ref{nilpotentcase}. Let $A_i$, $i=1,2,\cdots$ be a sequence in $\mathcal{O}_A\setminus Z$. Suppose $A_i$ has a limit point in $Z$. Then from the lemma below, if $A$ is diagonalizable, then the eigenvalues are uni-real.
\begin{lem}\label{closed}
Let $D=\text{diag}(\lambda_1,\cdots,\lambda_n)$, then its orbit $\mathcal{O}_D$ is closed in $sl(n,\mathbb C)$.
\end{lem}
\begin{proof}   Suppose $A_i\in \mathcal{O}_D$, $i=1, 2, \cdots$, and $A_i\rightarrow A_\infty$ in $sl(n,\mathbb C)$. We need to show $A_\infty\in \mathcal{O}_D$. Let $f(\lambda)$ be the minimal polynomial of $\mathcal{O}_A$. Then $f(\lambda)$ has no multiple root and $f(A_i)=0$. By taking the limit, $f(A_{\infty})=0$. So the minimal polynomial of $A_{\infty}$ also has no multiple root. Consider the characteristic polynomial of $\mathcal O_A$, $$\chi(\lambda)=\det(\lambda I-A_i)=\prod\limits_{i=1}^{n}(\lambda-\lambda_i).$$ By taking the limit, the characteristic polynomial of $A_{\infty}$ is also $\chi(\lambda)$. Together with the minimal polynomial having no multiple root, we obtain $A_\infty\in \mathcal{O}_A$.
\end{proof}

In fact, if $A$ is diagonalizable and $A_i$ approaches to a point $P\in Z$, then from the closedness of $\mathcal{O}_A$, we have $P\in \mathcal{O}_A$. From Lemma \ref{uni-real}, the eigenvalues of $P$ are uni-real, which implies the eigenvalues of $A$ are uni-real.

Now we consider the inferior limit of $K$ in the orbit of $A$ as approaching to $Z$. Denote by $J_k^{\lambda}$ the matrix $\left(\begin{array}{cccc}
\lambda & & &\\
1 & \lambda &\\
 & \ddots & \ddots \\
 &  & 1 & \lambda
\end{array}\right)$ of size $k$.
\begin{prop}\label{Zvalue}(Inferior limit as approaching to $Z$)\\
(1) Let $D=\text{diag}(\lambda_1,\cdots,\lambda_n)$, $\lambda_i\in\mathbb{R}$, $i=1,\cdots,n$. Suppose  $D$ is not a scalar matrix, then
\begin{eqnarray*}
\liminf_{A\rightarrow Z,~A\in \mathcal{O}_D\setminus Z}K(A)=\frac{\min\limits_{\lambda_i\neq \lambda_j}|\lambda_i-\lambda_j|^2}{\sum\limits_{i=1}^n|\lambda_i|^2},
\end{eqnarray*}
(2) Let $D$ be $c\cdot X^{\pi}$ for some $\pi\in \mathcal P_n$ and $c\in \mathbb C^*$, then \begin{eqnarray*}
\liminf\limits_{A\rightarrow Z,~A\in \mathcal{O}_D\setminus Z}K(A)=\Big\{
\begin{array}{cc}C_{\pi},\quad \text{ if $\pi$ is even},\\
\frac{C_{\pi}}{4}, \quad \text{ if $\pi$ is odd}.\end{array}
\end{eqnarray*}
(3)  Let $J=\text{diag}(J_{k_1}^{\lambda_1},\cdots,J_{k_p}^{\lambda_p})$ and neither nilpotent nor diagonal, then $$\liminf\limits_{A\rightarrow Z,~A\in \mathcal O_J\setminus Z}K(A)=0.$$
\end{prop}
\begin{proof} For the proof of Part (1):  Suppose the sequence $A_i=g_iDg_i^{-1}$ is approaching a point $P$ in $Z$. Since $\mathcal{O}_A$ is closed, $P$ is also in this orbit, set $P=CDC^{-1}$, where $C$ is unitary. Then the family $C^{-1}g_iDg_i^{-1}C$ converges to $D$ which has the same value of $K$ since $K$ is invariant under unitary conjugation. So without loss of generality, we can assume the sequence $A_i=g_iDg_i^{-1}$ is approaching $D$. From linear algebra, we know for any $g\in SL(n,\mathbb C)$, there is a unitary matrix $U$ and a Hermitian positive matrix $R$ such that $g=RU$. Since $K$ is invariant under unitary conjugation, $K(gDg^{-1})=K(RDR^{-1})$. So we may assume $g_i$ is Hermitian positive. Notice that the convergence is with respect to the topology of $\mathcal{O}_D$ which is induced from the natural topology of $sl(n,\mathbb{C})$. Since the orbit $\mathcal{O}_D$ is closed, it must be locally compact. Then from Theorem 3.2 in \cite{Helgason}, $\mathcal{O}_D$ is homeomorphism to the homogeneous space $SL(n,\mathbb{C})/H$, where $H$ is the subgroup of $SL(n,\mathbb{C})$ fixing $D$. The topology of the homogeneous space is from quotient topology. So we have $A_i=g_iDg_i^{-1}$ converges to $D$ in the quotient topology, which means there is sequence $g_i^\prime\in g_iH$ such that $g_i^\prime$ approaches to identity. Since $D$ is invariant under the action of $H$, we may assume $g_i$ is approaching to identity. Let $g_i=e^{B_i}$. Then it is enough to consider the family $A_i=e^{B_i}De^{-B_i}$ where $B_i$ is Hermitian and approaching to $0$.

From now on, we omit the subscript $i$. Consider $A(t)=\text{Ad}_{tB}D=e^{tB}De^{-tB}$, then
\begin{equation}\label{Taylor}
A^{(n)}(0)=(\text{ad}_{B})^{n}D.
\end{equation}
Suppose $B$ is approaching to $0$, then
\begin{equation*}
A(1)=A=e^{B}De^{-B}=D+[B,D]+O(|B||[B,D]|).
\end{equation*}
As assumption $B^*=B$, $D^*=D$, then $[B,D]^*=[D^*,B^*]=-[B,D]$. Then
\begin{eqnarray*}
&&[A,A^*]=[[B,D],D]+[D,-[B,D]]+O(|B||[B,D]|)=2[[B,D],D]+O(|B||[B,D]|),\\
&&|[A,A^*]|^2=4|[[B,D],D]|^2+O(|B||[B,D]|^2).
\end{eqnarray*}
From Equation \ref{Taylor}, we know that
\begin{equation*}
A=e^BDe^{-B}=D+\sum\limits_{i=1}^{\infty}\frac{1}{i!}(\text{ad}_{B})^i D
=D+[B,D]+\frac{1}{2}[B,[B,D]]+O(|B|^2|[B,D]|).
\end{equation*}
Note that $\big<D,[B,D]\big>=\big<[D,D^*],B\big>=0.$ Also, use $B^*=B$, then
\begin{eqnarray*}
|A|^2&=&|D|^2+|[B,D]|^2+\text{Re}\big<D,[B,[B,D]]\big>+
\text{Re}\big<D,\sum\limits_{i=3}^{\infty}\frac{1}{i!}(\text{ad}_B)^i D\big>+O(|B||[B,D]|^2)\\
&=&|D|^2+2|[B,D]|^2+\text{Re}\big<[B,D],\sum\limits_{i=2}^{\infty}\frac{1}{(i+1)!}(\text{ad}_B)^iD\big>
+O(|B||[B,D]|^2)\\
&=&|D|^2+2|[B,D]|^2+O(|B||[B,D]|^2),\\
\big<A,A^*\big>&=&\text{tr}(A^2)=\text{tr}(D^2)=|D|^2,\\
|A|^4-|\big<A,A^*\big>|^2&=&4|D|^2|[B,D]|^2+O(|B||[B,D]|^2).
\end{eqnarray*}

We obtain
\begin{eqnarray*}
K(A)=\frac{|[A,A^*]|^2}{|A|^4-|\big<A,A^*\big>|^2}=\frac{4|[[B,D],D]|^2+O(|B||[B,D]|^2)}
{4|D|^2|[B,D]|^2+O(|B||[B,D]|^2)}
\end{eqnarray*}
For the inferior limit of $K(A)$,
\begin{eqnarray*}
K(A)&=&\frac{4|[[B,D],D]|^2+O(|B||[B,D]|^2)}{4|D|^2|[B,D]|^2+O(|B||[B,D]|^2)}\\
&=&\frac{4\sum\limits_{i,j}|\lambda_i-\lambda_j|^2|[B,D]_{ij}|^2
+O(|B||[B,D]|^2)}{4\sum\limits_{i=1}^n|\lambda_i|^2|[B,D]|^2+O(|B||[B,D]|^2)}\\
&\geq&\frac{4\min\limits_{\lambda_i\neq \lambda_j}|\lambda_i-\lambda_j|^2\cdot\sum\limits_{\lambda_i\neq \lambda_j}|[B,D]_{ij}|^2
+O(|B||[B,D]|^2)}{4\sum\limits_{i=1}^n|\lambda_i|^2\cdot |[B,D]|^2+O(|B||[B,D]|^2)}\\
&=&\frac{4\min\limits_{\lambda_i\neq \lambda_j}|\lambda_i-\lambda_j|^2\cdot|[B,D]|^2
+O(|B||[B,D]|^2)}{4\sum\limits_{i=1}^n|\lambda_i|^2\cdot |[B,D]|^2+O(|B||[B,D]|^2)}.
\end{eqnarray*}
So as $B\rightarrow 0$,
\begin{eqnarray*}
\liminf\limits_{B\rightarrow 0}K(A)\geq \frac{\min\limits_{\lambda_i\neq \lambda_j}|\lambda_i-\lambda_j|^2}{\sum\limits_{i=1}^n|\lambda_i|^2}.
\end{eqnarray*}
To see the equality, suppose $(i_0,j_0)$ achieves $\min\limits_{\lambda_i\neq \lambda_j}|\lambda_i-\lambda_j|$, choose $B_0$ satisfying $(B_{0})_{i_0j_0}=(B_{0})_{j_0i_0}=1$ and 0 for other entries. Let $B_t=tB_0$. Then from the discussion above, $K(B_t)$ gives the desired limit $\frac{\min\limits_{\lambda_i\neq \lambda_j}|\lambda_i-\lambda_j|^2}{\sum\limits_{i=1}^n|\lambda_i|^2}$ when $t\rightarrow 0$.\\

For the proof of Part (2): It follows from direct calculation, notice that $|X^{\pi}|^2=\frac{4}{C_\pi}$.\\

For the proof of Part (3): Denote by $J_{k,t}^{\lambda}$ the matrix $\left(\begin{array}{cccc}
\lambda & & &\\
t & \lambda &\\
 & \ddots & \ddots \\
 &  & t & \lambda
\end{array}\right)$  of size $k$, $t>0$, which is conjugate to $J_k^{\lambda}$. Then we consider $A_t=\text{diag}(J_{k_1,t}^{\lambda_1},\cdots,J_{k_p,t}^{\lambda_p})$ which is conjugate to $J$. Let \begin{equation*}
D=\text{diag}(\tilde{\lambda}_1,\tilde{\lambda}_2,\cdots,\tilde{\lambda}_{n-1},\tilde{\lambda}_n)
=\text{diag}(\lambda_1,\lambda_1,\cdots,\lambda_p,\lambda_p),
\end{equation*}
then $A_t=D+tM_0,$ where $M_0$ is nilpotent and nonzero since $E$ is not diagonal. Since $M_0$ is nilpotent, $\big<M_0,M_0^*\big>=\text{tr}(M_0^2)=0.$ Also, $\big<D,M_0\big>=0$, and $[M_0, D]=0, [M_0^*,D]=0.$ So we have
\begin{eqnarray*}
[A_t,A_t^*]&=&t^2[M_0,M_0^*],\\
|A_t|^2&=&\big<A_t, A_t\big>=\big<D+tM_0,D+tM_0\big>=|D|^2+t^2|M_0|^2,\\
\big<A_t, A_t^*\big>&=&\big<D+tM_0,D^*+tM_0^*\big>=\big<D,D^*\big>=|D|^2,\\
|A_t|^4-|\big<A_t,A_t^*\big>|^2&=&2t^2|D|^2|M_0|^2+t^4|M_0|^4.
\end{eqnarray*}
We then obtain
\begin{eqnarray*}
K(A_t)=\frac{|[A_t,A_t^*]|^2}{|A_t|^4-|\big<A_t,A_t^*\big>|^2}
=\frac{t^2|[M_0,M_0^*]|^2}{2|D|^2|M_0|^2+t^2|M_0|^4}.
\end{eqnarray*}
Since $J$ is not nilpotent, $D$ is not a zero matrix, so $\lim\limits_{t\rightarrow 0}K(A_t)=0$. We finish the proof.
\end{proof}

\subsection{Limit behavoir at boundary or infinity}\label{limit}

In this subsection, we study the limit behavior when $A$ approaches the boundary of $\mathcal O_A$ or infinity, in other words, the boundary of $\mathbb P(\mathcal O_A) \subset \mathbb P(sl(n,\mathbb C))$. 

Given an $n\times n$ matrix $A$, denote by $\pi(A)=(m_1,\cdots, m_{n_1})\in \mathcal P_n$, where $m_a$ is the degree of the $a$-th invariant factor $d_a(\lambda)$ of $A$. For $A$ being nilpotent, $\pi(A)$ coincides with its Jordan type. So one can view $\pi(A)$ as a generalization of Jordan type from nilpotent matrix to a general matrix.

The following proposition shows that $\pi(A)$ satisfy a lower semi-continuous property in $\mathbb P(sl(n,\mathbb C))$.
\begin{prop}(Lower semi-continuity of $\pi(A)$)\label{LowerSemiContinuity}
Suppose $A_i$ is a sequence of matrices in the adjoint orbit $\mathcal O_{A}$ and $c_i\in \mathbb C^*$ is a sequence of constants such that $\frac{A_i}{c_i}\rightarrow A_{\infty}\in sl(n,\mathbb C)$, then $\pi(A_{\infty})\leq \pi(A).$

In the case $c_i$ is bounded, equality holds if and only if $A_{\infty}\in \mathcal O_A$.

In the case $c_i$ is unbounded, then $A_{\infty}$ is nilpotent. And $\pi(A)$ can be achieved for suitable $A_i,c_i$.
\end{prop}

\begin{proof}
We divide the proof into two cases: the case $c_i$ is bounded and the case $c_i$ is unbounded. Before going into the proof, we first review the elementary factors and invariant factors of a matrix.

Let $p\in \mathbb{N}_+$, $n_1\geq \cdots\geq n_p\in \mathbb{N}_+$, and $\lambda_1,\cdots,\lambda_p\in \mathbb{C}$ be distinct. Let
\begin{equation*}
k_{11}\geq \cdots \geq k_{n_{1}1},~\cdots~,k_{1b}\geq \cdots \geq k_{n_bb},~\cdots~,k_{1p}\geq \cdots \geq k_{n_{p}p}
\end{equation*}
be positive integers satisfying $\sum\limits_{b=1}^{p}\sum\limits_{a=1}^{n_b}k_{ab}=n$. Suppose $A$ has Jordan normal form as
\begin{equation*}
\text{diag}(J_{k_{11}}^{\lambda_1},\cdots, J_{k_{n_11}}^{\lambda_1},~\cdots~, J_{k_{1b}}^{\lambda_b},\cdots,
J_{k_{n_bb}}^{\lambda_b},~\cdots~, J_{k_{1p}}^{\lambda_p},\cdots, J_{k_{n_pp}}^{\lambda_p}).
\end{equation*}
The elementary factors $e_{jb}(A)$, invariant factors $d_j(A)$ of $A$ are respectively
\begin{eqnarray*}
e_{jb}(\lambda)&=&(\lambda-\lambda_b)^{k_{jb}}, \quad\text{ for } j=1, \cdots, n_b\\
d_j(\lambda)&=&(\lambda-\lambda_1)^{k_{j1}}(\lambda-\lambda_2)^{k_{j2}}\cdots (\lambda-\lambda_p)^{k_{jp}},\quad\text{ for }j=1,\cdots, n_1.
\end{eqnarray*}
Let $m_j:=\deg d_j(\lambda)=k_{j1}+\cdots+k_{jp}$. Then $\pi(A)=(m_1,m_2,\cdots,m_{n_1})\in \mathcal P_n$. Note that
\begin{equation*}A-\lambda_bI=\text{diag}(J^{\lambda_1-\lambda_b}_{k_{11}},\cdots,J^{\lambda_1-\lambda_b}_{k_{n_11}},~\cdots~,
J^{0}_{k_{1b}},\cdots,J^{0}_{k_{n_bb}},~\cdots~,
J^{\lambda_p-\lambda_b}_{k_{1p}},\cdots,J^{\lambda_p-\lambda_b}_{k_{n_pp}}),
\end{equation*}
and
\begin{equation}\label{BasicForm}
e_{jb}(A)=(A-\lambda_b I)^{k_{jb}},\quad d_j(A)=(A-\lambda_1I)^{k_{j1}}(A-\lambda_2I)^{k_{j2}}\cdots (A-\lambda_pI)^{k_{jp}},
\end{equation}
where we use the convention $k_{ji}=0$ for $j>n_i$. Since the rank of $(J_{k}^{\mu})^l$ is
\begin{eqnarray}\label{rank}
\text{rk}\{(J^{\mu}_k)^l\}=
\Big\{
\begin{array}{cc}
\max \{k-l,0\} & \text{ for }\mu=0\\
k & \text{ for }\mu\neq 0,
\end{array}
\end{eqnarray}
we have
\begin{eqnarray}
\text{rk}(e_{jb}(A)|_{V_{jb}})&=&\sum_{a=1}^j\text{rk}\{(J_{k_{ab}}^0)^{k_{jb}}\}
=\sum\limits_{a=1}^j(k_{ab}-k_{jb})=(\sum\limits_{a=1}^j k_{ab})-jk_{jb},\nonumber\\
\text{rk}(e_{jb}(A))&=&(\sum\limits_{a=1}^j k_{ab})-jk_{jb}+n-l_b,\label{ElementaryRank}
\end{eqnarray}
where $V_{jb}$ is the subspace corresponding to the Jordan block $J_{jb}^{\lambda_b}$ and $l_b$ is the multiple of $\lambda_b$ in the characteristic polynomial of $A_\infty$ (and $A_i$). Then
\begin{equation}\label{InvariantRank}
\text{rk}(d_{j}(A))=\sum\limits_{b=1}^p\sum\limits_{a=1}^{j}\text{rk}\{(J_{k_{ab}}^0)^{k_{jb}}\}
=\sum\limits_{b=1}^{p}\sum\limits_{a=1}^{j}(k_{ab}-k_{jb})
=\sum\limits_{a=1}^{j}\sum\limits_{b=1}^{p}(k_{ab}-k_{jb})=(\sum\limits_{a=1}^{j}m_a)-jm_j.
\end{equation}

Case 1: $c_i$ is bounded. Then we may assume it has a limit $c$. So we can also view $\frac{1}{c}A_{\infty}$ as the limit of $A_i$.
Without loss of generality, we can assume $A_i\rightarrow A_{\infty}\in sl(n,\mathbb C)$.

We make use of elementary factors. Since $A_{\infty}$ is a limit of $A_i\in \mathcal  O_A$, $A_{\infty}$ has the same characteristic polynomial, which means the same eigenvalue set of $A$ as well as multiplicities. For each $1\leq b\leq p$, denote
\begin{equation*}
k_{1b}^{\infty}\geq k_{2b}^{\infty}\geq \cdots \geq k_{n_bb}^{\infty}
\end{equation*}
as the degrees of the elementary factors of $A_{\infty}$ for eigenvalue $\lambda_b$ and $m_j^{\infty}=\sum\limits_{b=1}^{p}k_{jb}^{\infty},~j=1, \cdots, n_1$ the degree of the invariant divisor of $A_{\infty}$. Note that $\sum\limits_{a=1}^{n_b}k_{ab}^{\infty}=\sum\limits_{a=1}^{n_b}k_{ab},$ since the multiplicity of eigenvalue $\lambda_b$ of $A_{\infty}$ is the same as the one of $A$. It is enough to show the following claim.

Claim: For each fixed $b$, $\sum\limits_{a=1}^sk_{ab}^{\infty}\leq \sum\limits_{a=1}^sk_{ab}, \forall s\geq 1.$

Because $\sum\limits_{a=1}^sm_a=\sum\limits_{a=1}^s\sum\limits_{b=1}^pk_{ab}
=\sum\limits_{b=1}^p(\sum\limits_{a=1}^sk_{ab}), $ and  $\sum\limits_{a=1}^sm_a^{\infty}=\sum\limits_{a=1}^s\sum\limits_{b=1}^pk_{ab}^{\infty}
=\sum\limits_{b=1}^p(\sum\limits_{a=1}^sk_{ab}^{\infty}),$ then the claim implies that $\sum\limits_{a=1}^sm_a^{\infty}\leq \sum\limits_{a=1}^sm_a.$

We will prove the claim by induction. From (\ref{BasicForm}), we have $e_{jb}(A_{\infty})=(A_{\infty}-\lambda_bI)^{k_{jb}}$ and $A_{\infty}-\lambda_bI$ is conjugate to \begin{equation*}
\text{diag}(J_{k_{11}^{\infty}}^{\lambda_1-\lambda_b},\cdots, J_{k_{n_11}^{\infty}}^{\lambda_1-\lambda_b},~\cdots~, J_{k_{1b}^{\infty}}^{0},\cdots,
J_{k_{n_bb}^{\infty}}^{0},~\cdots~, J_{k_{1p}^{\infty}}^{\lambda_p-\lambda_b},\cdots, J_{k_{n_pp}^{\infty}}^{\lambda_p-\lambda_b}).
\end{equation*}
From (\ref{rank}), we have
\begin{equation}\label{LimitSequence0}
\text{rk}(e_{jb}(A_\infty))=\sum\limits_{a=1}^j\text{rk}\{(J_{k_{ab}^{\infty}}^0)^{k_{jb}}\}
+n-l_b,
\end{equation}
Since the rank is a lower semi-continuous function on the space of matrices, we obtain \begin{equation}\label{LimitSequence1}
\text{rk}(e_{jb}(A_\infty))\leq \lim\limits_{i\rightarrow \infty} \text{rk}(e_{jb}(A_i))=(\sum\limits_{a=1}^{j}k_{ab})-jk_{jb}+n-l_b,
\end{equation}
where the equality follows from Equation (\ref{ElementaryRank}).

First we show $k_{1b}^{\infty}\leq k_{1b}$. 
From Equation (\ref{LimitSequence0}),
$$\text{rk}(e_{1b}(A_{\infty}))=\text{rk}(J_{k_{1b}^{\infty}}^0)^{k_{1b}}+n-l_b
=\max\{k_{1b}^{\infty}-k_{1b},0\}+n-l_b
\geq k_{1b}^{\infty}-k_{1b}+n-l_b.$$
Since $A_i\in\mathcal{O}_{A}$, from Equation (\ref{LimitSequence1}), we have $$\text{rk}(e_{1b}(A_\infty))\leq\lim\limits_{i\rightarrow \infty}\text{rk}(e_{1b}(A_i))=n-l_b.$$
So $k_{1b}^{\infty}\leq k_{1b}$.

Assume that  $\sum\limits_{a=1}^{j-1}k_{ab}^{\infty} \leq \sum\limits_{a=1}^{j-1}k_{ab},$ we are going to show that  $\sum\limits_{a=1}^{j}k_{ab}^{\infty}\leq \sum\limits_{a=1}^{j}k_{ab}.$  It suffices to show the case when $k_{jb}^{\infty}>k_{jb}$ since the statement follows immediately from assumption when $k_{jb}^{\infty}\leq k_{jb}$.

Since $k_{jb}^{\infty}\geq k_{jb}$, then $k_{ab}^{\infty}\geq k_{jb}$ for $1\leq a\leq j$. Using the fact that $\text{rk}\{(J_k^0)^l\}=k-l$ for $k\geq l$, from Equation (\ref{LimitSequence0}), we have
\begin{equation}\label{LimitSequence2}
\text{rk}(e_{jb}(A_\infty))=(\sum\limits_{a=1}^{j}k_{ab}^{\infty})-jk_{jb}+n-l_b.
\end{equation}

Combining Equation (\ref{LimitSequence1}) and (\ref{LimitSequence2}), we obtain that  $\sum\limits_{a=1}^{j}k_{ab}^{\infty}\leq \sum\limits_{a=1}^{j}k_{ab}.$

For $\pi(A_\infty)=\pi(A)$, the Jordan norm form of $A_{\infty}$ coincides with the the one of $A_{0}$, then $A_\infty\in\mathcal{O}_{A}$. Conversely it is clear. \\

Case 2: $c_i$ is unbounded. We make use of invariant factors.
Suppose $c_i\rightarrow \infty$,  $\hat{A_i}=\frac{A_i}{c_i}\rightarrow N$. Since $\text{tr} (N^k)=\lim\limits_{i\rightarrow \infty}\text{tr} (\hat{A_i}^k)=\lim\limits_{i\rightarrow \infty}\frac{\text{tr} (A^k)}{c_i^k} =0,\text{ for }k=1,\cdots,n,$ then $N$ is nilpotent.
Suppose the Jordan type of $N$ is $\pi(N)=(k_1,\cdots, k_n)$. We want to show that $\pi(N)\leq \pi(A)$ by induction.

We know that the invariant factors of $\hat{A_i}$ are
\begin{equation*}
d_{j}^{\hat{A_i}}(\lambda)=(\lambda-\frac{\lambda_1}{c_i})^{k_{j1}}
(\lambda-\frac{\lambda_2}{c_i})^{k_{j2}}\cdots (\lambda-\frac{\lambda_p}{c_i})^{k_{jp}},\text{ for }j=1,\cdots, n_1.
\end{equation*}
For each $1\leq j\leq n_1$, we have
\begin{equation}\label{NilpotentProduct}
\lim_{i\rightarrow \infty}d_{j}^{\hat{A_i}}(\hat{A_i})= N^{m_j}.
\end{equation}
Since the Jordan type of $N$ is $\pi(N)=(k_1,\cdots, k_n)$, $N$ is conjugate to $\text{diag}(J_{k_1}^0, \cdots, J_{k_n}^0)$ and
\begin{equation}\label{NilpotentMultiple}
\text{rk}(N^{m_j})= \sum\limits_{a=1}^{j}\text{rk}(J_{k_a}^0)^{m_j}.
\end{equation}

First we show $k_1\leq m_1$. Since the first invariant polynomial is the minimal polynomial, from Equation (\ref{NilpotentProduct}), we have $N^{m_1}=0.$  So from Equation (\ref{NilpotentMultiple}), $0=\text{rk}(N^{m_1})=\text{rk}(J_{k_1}^0)^{m_1}=\max\{k_1-m_1,0\}$. So $k_1\leq m_1.$

Assume that  $\sum\limits_{a=1}^{j-1}k_a \leq \sum\limits_{a=1}^{j-1}m_a,$ we are going to show that  $\sum\limits_{a=1}^{j}k_a \leq \sum\limits_{a=1}^{j}m_a.$ \\

We only need to show the case that $k_j>m_j$ since the inequality follows immediately from the assumption when $k_j\leq m_j$. Since the rank is a lower semi-continuous function on the space of matrices, we obtain
\begin{equation}\label{RankNSequence1}
\text{rk}(N^{m_j})=\text{rk}(\lim_{i\rightarrow \infty}d_{j}^{\hat{A_i}}(\hat{A_i}))\leq \lim\limits_{i\rightarrow \infty}\text{rk} (d_{j}^{\hat{A_i}}(\hat{A_i}))=(\sum\limits_{a=1}^{j}m_a)-jm_j,
\end{equation}
where the left equality follows from Equation (\ref{NilpotentProduct}) and the right equality follows from Equation (\ref{InvariantRank}).
Since $k_j\geq m_j$,  $k_a\geq m_j$ for $1\leq a\leq j$. By using the fact $\text{rk}(J_k^0)^l=k-l$ for $k\geq l$, from Equation (\ref{NilpotentMultiple}), we have
\begin{equation}\label{RankNSequence2}
\text{rk}(N^{m_j})=(\sum\limits_{a=1}^{j}k_a)-jm_j.
\end{equation}
Combining Equation (\ref{RankNSequence1}) and (\ref{RankNSequence2}), we obtain $\sum\limits_{a=1}^{j}k_a \leq \sum\limits_{a=1}^{j}m_a.$ So we finish the proof of the induction.

To see the equality, we construct a sequence $A_i\in \mathcal O_{A}$ as follows. Let $A^j=\text{diag}(J_{k_{j1}}^{\lambda_1},\cdots,J_{k_{jq}}^{\lambda_q})$ be the submatrix of $A$ corresponding to the invariant factor $d_j(\lambda)=(\lambda-\lambda_1)^{k_{j1}}(\lambda-\lambda_2)^{k_{j2}}\cdots (\lambda-\lambda_q)^{k_{jq}}$ for $k_{jq}>0$. Define $A_{i}^{j}=\text{diag}(\lambda_1I_{k_{j1}},\cdots,\lambda_qI_{k_{jq}})+iJ^{0}_{m_j}$, where $m_j=\deg(d_j(\lambda))=\sum\limits_{l=1}^{q}k_{jl}$. Let $A_i=\text{diag}(A^1_i,\cdots,A^{n_1}_i)$. By direct calculation, the $\lambda$-matrix of $A_i$ is equivalent to the $\lambda$-matrix of $A$. From the $\lambda$-matrix theory, $A_i$ is similar to $A$. Let $N=\lim\limits_{i\rightarrow \infty}\frac{A_i}{|A_i|}=\text{diag}(J^0_{m_1},\cdots,J^0_{m_{n_1}}).$ Then $N$ is nilpotent of Jordan type $\pi(A)$.
\end{proof}

\subsection{Nilpotent case}\label{nilpotentcase}
\label{ProofNess} We are ready to give a new proof of Theorem \ref{Ness}.

\begin{proof} (of Theorem \ref{Ness})
The Part (1) of Theorem \ref{Ness} follows from Proposition \ref{critical} and the fact that the nilpotent elements in a standard $sl(2,\mathbb C)$ satisfy Equation (\ref{Rigidity}).

For Part (2), we prove it by induction on the partial order of the partitions of $n$.

Suppose $\pi(N)=(2,1,\cdots, 1)\in \mathcal P_n.$ Note that this partition is a global minimum among all partitions except $(1,\cdots, 1)$, which means $N=0$. Suppose there is a sequence $N_i$ inside the orbit such that $K(N_i)$, $i=1,2,\cdots$ approaches to the infimum. Consider a limit $[N_\infty]$ of $[N_i]\in \mathbb P(sl(n,\mathbb C))$, which means there exists a sequence $c_i$ such that $\frac{N_i}{c_i}\rightarrow N_{\infty}$. From Proposition \ref{LowerSemiContinuity}, $\pi(N_{\infty})\leq \pi(N)$. But $\pi(N)$ is already the minimum among partitions except $(1,\cdots,1)$, so $\pi(N_{\infty})=\pi(N)$ and $N_{\infty}\in \mathcal O_N$. So the infimum must be achieved in the interior. From Proposition \ref{critical}, $K(N)$ achieves its minimum $C_{\pi(N)}$ exactly at an $SU(n)$-adjoint orbit of  $E^{\pi(N)}$.

Assume that for any $\pi_1<\pi$, if $N$ is nilpotent of Jordan type $\pi_1$,  the function $K$ on the orbit $\mathcal O_N$ achieves its minimum $C_{\pi_1}$ exactly at an $SU(n)$-adjoint orbit of $E^{\pi_1}$.  Now we are going to show the same statement holds if $N$ is nilpotent of Jordan type $\pi$.

Suppose there is a sequence $N_i$ inside the orbit such that $K(N_i)$, $i=1,2,\cdots$ approaches to the infimum. Consider a limit $[N_\infty]$ of $[N_i]\in \mathbb P(sl(n,\mathbb C))$. Since $K$ is scaling invariant, $\lim\limits_{i\rightarrow\infty}K(N_i)=K(N_\infty)$. From Proposition \ref{LowerSemiContinuity}, $\pi(N_{\infty})\leq \pi(N)=\pi$. If $\pi(N_{\infty})<\pi(N)$, then by the assumption step in the induction, we have $K(N_{\infty})\geq C_{\pi(N_{\infty})}$. Since $\pi(N_{\infty})<\pi(N)$, from Lemma \ref{ql1}, $C_{\pi(N_{\infty})}> C_{\pi(N)}$. So we have $K(N_{\infty})> C_{\pi}$, which is is impossible since from Proposition \ref{critical} and Lemma \ref{C(K)}, $C_{\pi}$ is a critical value of $K$, but $K(N_{\infty})$ is the infimum of $K$ on $\mathcal O_A$. If $\pi(N_{\infty})= \pi(N)$, then $N_{\infty}\in\mathcal{O}_{N}$. From Proposition \ref{critical}, the minimum is $C_{\pi}$. Therefore, we finish the proof.
\end{proof}

\subsection{General case}

In this subsection, we solve the infimum problem of $K$ in general cases, which implies Theorem \ref{GeneralizationOfNessTheoremSl2}. First we continue the discussion on the inferior limit of $K(A)$ when $A$ approaches to the boundary of $\mathcal{O}_{A}$ or infinity.

For the boundary, the nilpotent case is discussed in Section \ref{nilpotentcase}. If $A$ is neither diagonalizable nor nilpotent, from the proof of Part (3) in Proposition \ref{Zvalue}, we see the inferior limit of $K(A)$ is zero when $A$ approaches to the boundary of $\mathcal{O}_{A}$. For $A$ being diagonalizable, since the orbit is closed, there is no boundary.

Now we consider the situation when $A$ approaches to infinity.
\begin{prop}\label{inftyvalue0}
Let $A\in sl(n,\mathbb{C})$, which is not a scalar matrix. Let $\pi(A)$ be the partition of $n$ determined by the degree of the invariant factor of $A$. Then
\begin{equation*}
\liminf\limits_{B\in \mathcal{O}_{A}\setminus Z,~|B|\rightarrow\infty}K(B)= C_{\pi(A)}.
\end{equation*}
\end{prop}
\begin{proof} Considering subsequence, we assume $\frac{A_n}{|A_n|}\rightarrow N$ where $N$ is nilpotent. Then \begin{equation*}
\lim\limits_{n\rightarrow \infty} K(A_n)=\lim\limits_{n\rightarrow \infty}K(\frac{A_n}{|A_n|})=K(N).
\end{equation*}
From Theorem \ref{Ness}, $K(N)\geq C_{\pi(N)}.$ From Proposition \ref{LowerSemiContinuity}, $\pi(A)\geq \pi(N)$. Then from Lemma \ref{ql1},  $C_{\pi(A)}\leq C_{\pi(N)}$. So we obtain  $\lim\limits_{n\rightarrow \infty}K(A_n)\geq C_{\pi(A)}.$

To see the equality, from Proposition \ref{LowerSemiContinuity}, we see that there exists a sequence $A_i\in \mathcal O_A, c_i\in \mathbb C^*$ unbounded  such that $\frac{A_i}{c_i}$ limits to a nilpotent matrix $N$ with $\pi(N)=\pi(A)$. Then there exists a $g\in sl(n,\mathbb C)$ such that $N=g^{-1}E^{\pi(A)}g$ and so the family $\frac{gA_ig^{-1}}{c_i}$ limits to $E^{\pi}$. Therefore,  \begin{equation*}
\lim\limits_{n\rightarrow \infty} K(gA_ig^{-1})=\lim\limits_{n\rightarrow \infty} K(\frac{gA_ig^{-1}}{c_i})=K(E^{\pi(A)})=C_{\pi(A)}.
\end{equation*}
So we finish the proof.
\end{proof}

To consider the infimum of $K$ on $\mathcal{O}_A\setminus Z$, we summarize what we have obtained now. For $A$ being nilpotent, it is discussed in Section \ref{nilpotentcase}. For $A$ being neither diagonalizable nor nilpotent, the infimum is $0$ from Section \ref{Z}. For $A$ being diagonalizable but the eigenvalues being not uni-real, from Section \ref{infi} the infimum is $0$ and it is also the minimum. Now we consider the situation that $A$ is diagonalizable and the eigenvalues are uni-real. From the sections above, we see there two possible values as the infimum, $C_{\pi(A)}$ in Section \ref{critical}, \ref{limit} and $\frac{\min\limits_{\lambda_i\neq \lambda_j}|\lambda_i-\lambda_j|^2}{\sum\limits_{i=1}^n|\lambda_i|^2}$ in Section \ref{Z}. The next lemma gives the comparison of these two values. Denote $\Lambda_{n}=(n-1,n-3,\cdots,3-n,1-n)$.
\begin{lem}\label{Zinfty}
Suppose $A\in sl(n,\mathbb{C})$ is diagonalizable and is not a scalar matrix. Suppose the eigenvalues $(\lambda_1,\cdots, \lambda_n)$ of $A$ are uni-real. The degree of the invariant factors $\pi(A)=(n_1,\cdots,n_s)$ gives a partition of $n$. Then
$$\frac{\min\limits_{\lambda_i\neq \lambda_j}(\lambda_i-\lambda_j)^2}{\sum\limits_{i=1}^n\lambda_i^2}\leq C_{\pi(A)}.$$
Equality holds if and only if the eigenvalues have the form $c(\Lambda_{n_1},\cdots,\Lambda_{n_s})$ with even $\pi(A)$.
\end{lem}
\begin{proof} Suppose the eigenvalues of $A$ are uni-real, we assume the eigenvalues have the form $(\lambda_1,\cdots,\lambda_n)$ with $\lambda_i\in \mathbb{R}$, $i=1,\cdots,n$, $c\in\mathbb{C}^*$.  Since $A\in sl(n,\mathbb{C})$ is diagonalizable, each invariant factor $d_i$ has distinct eigenvalues $(\lambda_1^i, \cdots, \lambda_{n_i}^i)$. Without loss of generality, we can assume $\min\limits_{\lambda_i\neq \lambda_j}|\lambda_i-\lambda_j|\geq 2$. It is enough to show for each $i$, $\sum\limits_{l=1}^{n_i}\lambda_{l}^2\geq \frac{n_i^3-n_i}{3}$. It reduces to show the following claim.

Claim: Suppose $\lambda_1>\lambda_2>\cdots>\lambda_n$ and $\lambda_i-\lambda_{i+1}\geq 2$, $i=1,\cdots,n-1$, then $\sum\limits_{i=1}^n\lambda_i^2\geq \frac{n^3-n}{3}.$ Equality holds if and only if $(\lambda_i)=\Lambda_n$.

To minimize $\sum\limits_{i=1}^n\lambda_i^2$, we may assume $\lambda_i-\lambda_{i+1}=2$, $i=1,\cdots,n-1$. In fact, for $\lambda_i>\lambda_{i+1}\geq 0$, set $\lambda_{i+1}^{\prime}=\lambda_{i+1},~\lambda_i=\lambda_{i+1}+2$, for $\lambda_{i+1}<\lambda_{i}\leq 0$, set $\lambda_{i}^{\prime}=\lambda_{i},~\lambda_{i+1}=\lambda_{i}-2$, for $\lambda_i> 0>\lambda_{i+1}$, set $\lambda_{i}^{\prime}=\frac{2}{\lambda_i-\lambda_{i+1}}\lambda_{i},~\lambda_{i+1}^{\prime}=\frac{2}{\lambda_i-\lambda_{i+1}}\lambda_{i+1}$. Then $\sum\limits_{i=1}^n\lambda_i^2$ decreases preserving the condition $\lambda_1>\lambda_2>\cdots>\lambda_n$ and $\lambda_i-\lambda_{i+1}\geq 2$. So we assume $\lambda_1,\cdots,\lambda_n$ is an arithmetic progression with common difference $2$. Let $a=\frac{\lambda_1+\lambda_n}{2}$. Then for $n=2m+1$ odd, $\sum\limits_{i=1}^n\lambda_i^2=a^2+\sum\limits_{k=1}^m\big((a+2k)^2+(a-2k)^2\big)$; for $n=2m$ even, $\sum\limits_{i=1}^n\lambda_i^2=\sum\limits_{k=1}^m\big((a+2k-1)^2+(a-2k+1)^2\big).$ Notice that the terms of odd power of $a$ vanish and the coefficients of even power of $a$ is positive. So $\sum\limits_{i=1}^n\lambda_i^2$ is minimized when $a=0$, in which case $\sum\limits_{i=1}^n\lambda_i^2=\frac{n^3-n}{3}$. So we finish the proof of the claim and the lemma.
\end{proof}

Now we are in the position to give the full answer to the infimum problem.
\begin{thm}\label{main}
Let $A\in sl(n,\mb{C})$, which is not a scalar matrix. Consider $K(B)=\frac{|[B,B^*]|^2}{|B|^4-|\text{tr}(B^2)|^2}$ on $\mc{O}_{A}\setminus Z$. \\
(1) Suppose $A$ is diagonalizable, and the eigenvalues $(\lam_1,\cdots,\lam_n)$ are not uni-real. Then $K$ achieves its minimum $0$ exactly at the $SU(n)$-adjoint orbit of $\text{diag}(\lambda_1,\cdots,\lambda_n)$ and has no other critical points;\\
(2) Suppose $A$ is diagonalizable, and the eigenvalues have the form $c(\Lambda_{2m_1+1},\cdots,\Lambda_{2m_s+1})$, $c\in \mb{C}^*$, which gives an even partition $\pi=(2m_1+1,\cdots,2m_s+1)\in\mc{P}_n$. The eigenvalues also correspond to $2c(\Lambda_{m_1},\Lambda_{m_1+1},\cdots,\Lambda_{m_s},\Lambda_{m_s+1})$, which gives another partition $\pi^{\prime}=(m_1,m_1+1,\cdots,m_s,m_s+1)\in\mc{P}_n$. Then $K$ achieves its minimum $C_{\pi}$ exactly at the $SU(n)$-adjoint orbit of $j_{\pi}(sl(2,\mathbb C))\setminus Z$. $K$ has other critical points, which is the $SU(n)$-adjoint orbit of $j_{\pi^{\prime}}(sl(2,\mathbb C))\setminus Z$ of the same critical value $C_{\pi^{\prime}}=4C_{\pi}$.\\
(3) Suppose $A$ is diagonalizable, and the eigenvalues have the form $c(\Lambda_{2m_1},\cdots,\Lambda_{2m_s})$, $c\in \mb{C}^*$, which gives an even partition $\pi=(2m_1,\cdots,2m_s)\in\mc{P}_n$. Then $K$ achieves its minimum $C_{\pi}$ exactly at the $SU(n)$-adjoint orbit of $j_{\pi}(sl(2,\mathbb C))\setminus Z$. $K$ has no other critical points.\\
(4) Suppose $A$ is diagonalizable, and the eigenvalues have the form $c(\Lambda_{n_1},\cdots,\Lambda_{n_s})$, $c\in \mb{C}^*$, which gives a partition $\pi=(n_1,\cdots,n_s)\in\mc{P}_n$. Suppose $\pi$ is odd, and $\pi\neq (m_1,m_1+1,\cdots,m_l,m_l+1)$. Then $K$ can't achieve its minimum in the interior of $\mathcal{O}_{A}\setminus Z$ and the infimum of $K$ is $\frac{C_{\pi}}{4}$. The critical points of $K$ is the $SU(n)$-adjoint orbit of $j_{\pi}(sl(2,\mathbb C))\setminus Z$ of the same critical value $C_{\pi}$.\\
(5) Suppose $A$ is diagonalizable, and the eigenvalues $(\lam_1,\cdots,\lam_n)$ are uni-real but can't have the form $c(\Lambda_{n_1},\cdots,\Lambda_{n_s})$, $c\in \mb{C}^*$. Then $K$ can't achieve its minimum in the interior of $\mc{O}_{A}\setminus Z$. The infimum of $K$ is $\frac{\min\limits_{\lambda_i\neq \lambda_j}(\lambda_i-\lambda_j)^2}{\sum\limits_{i=1}^n\lambda_i^2}$. $K$ has no critical points.\\
(6) Suppose $A$ is nilpotent, let $\pi=(n_1,\cdots,n_s)$ be its Jordan type. Then $K$ achieves its minimum $C_{\pi}$ exactly at the $SU(n)$-adjoint orbit of $j_{\pi}(sl(2,\mathbb C))$. $K$ has no other critical points.\\
(7) Suppose $A$ is neither diagonalizable nor nilpotent. Then the infimum of $K$ is 0. $K$ has no critical points.
\end{thm}
\begin{proof} For the critical points, from Proposition \ref{critical} and Lemma \ref{eigen}, they can only happen in the case (2)(3)(4)(6). The critical value is from Lemma \ref{C(K)}. Now we consider the infimum. For the case (6), it follow from Theorem \ref{Ness}. For the case (7), it follows from the Part (3) of Proposition \ref{Zvalue}. For the case (1), it follows from Lemma \ref{uni-real}. Now we assume $A$ is diagonalizable and the eigenvalues $(\lam_1,\cdots,\lam_n)$ are uni-real. From Lemma \ref{closed}, $\mc{O}_A$ is closed. So the infimum may happen in the following three situations: (a) interior of $\mc{O}_A$, (b) $Z$, (c) infinity. From Lemma \ref{C(K)}, Propositions \ref{critical}, \ref{Zvalue}, and \ref{inftyvalue0}, we see the possible infimum values are $C_\pi$ and $\frac{\min\limits_{\lambda_i\neq \lambda_j}(\lambda_i-\lambda_j)^2}{\sum\limits_{i=1}^n\lambda_i^2}$. For the case (5), the eigenvalues $(\lam_1,\cdots,\lam_n)$ can't have the form $c(\Lambda_1,\cdots,\Lambda_{n_s})$, so it follows from Lemma \ref{Zinfty}. Suppose $(\lam_1,\cdots,\lam_n)=c(\Lambda_1,\cdots,\Lambda_{n_s})$, from Lemma \ref{2crit}, if it has more than one expression in this manner, then it must be $$(\lam_1,\cdots,\lam_n)=c(\Lambda_{2m_1+1},\cdots,\Lambda_{2m_s+1})= 2c(\Lambda_{m_1},\Lambda_{m_1+1},\cdots,\Lambda_{m_s},\Lambda_{m_s+1}).$$
This is the situation in the case (2). Then it follows from Lemma \ref{Zinfty}. The critical value $C_{\pi^{\prime}}=4C_{\pi}$ follows from direct calculation. For the case (3), we see that $\pi$ is also an even partition, so it follows from Lemma \ref{Zinfty}. Finally, for the case (4), the partition $\pi$ is odd and can't be written in an even partition, so from Lemma \ref{Zinfty} the infimum is $\frac{\min\limits_{\lambda_i\neq \lambda_j}(\lambda_i-\lambda_j)^2}{\sum\limits_{i=1}^n\lambda_i^2}$, which is $\frac{C_{\pi}}{4}$ from the Part (2) of Proposition \ref{Zvalue}.
\end{proof}

Theorem \ref{GeneralizationOfNessTheoremSl2} follows from Theorem \ref{main}.
\begin{proof}(of Theorem \ref{GeneralizationOfNessTheoremSl2})
The statement (1) follows from Proposition \ref{critical}. The statement (2) follows from Proposition \ref{C(K)}. The statement (3) follows from Theorem \ref{Ness}. For the statement (4), since $A\notin W\setminus Z$ and not nilpotent, we rule out the case (1) and case (6) in Theorem \ref{main}. We check each case in Theorem \ref{main}. Then we find the minimum is achieved only in the case (2) and case (3), that is the partition $\pi$ is even. So we finish the proof.
\end{proof}

Together with Lemma \ref{Zinfty}, we obtain
\begin{cor}\label{eurhb}
Suppose $A$ is diagonalizable with uni-real eigenvalues $(\lam_1,\cdots,\lam_n)$ and is not a scalar matrix, then $K\geq\frac{\min\limits_{\lambda_i\neq \lambda_j}(\lambda_i-\lambda_j)^2}{\sum\limits_{i=1}^n\lambda_i^2}$ on $\mc{O}_A\setminus Z$. The equality holds if and only if the eigenvalues have the form $c(\Lambda_{n_1},\cdots,\Lambda_{n_s})$ for some even partition $(n_1,\cdots,n_s)\in\mc{P}_n$.
\end{cor}

\section{Domination results}\label{domination results}
In this section, we first briefly recall some preliminaries in the the non-Abelian Hodge theory and higher Teichm\"uller theory, see \cite{LiSIGMA} for more details, and then prove the main theorems.
\subsection{Hitchin fibration}\label{setup}

Let $S$ be an oriented closed surface with genus at least $2$. Denote the fundamental group $\pi_1(S,p)$ of $S$ by $\pi_1$. Let $\Sigma=(S,J)$ be a Riemann surface structure on $S$ and $K_{\Sigma}$ be the canonical line bundle of $\Sigma$. Denote $X$ as the symmetric space $SL(n,\mathbb C)/SU(n)$ equipped with the Riemannian metric $g_X$ induced by the normalized Killing form on $sl(n,\mathbb C)$, i.e. $\big<A,B\big>=2 \text{tr}(AB)$ for $A, B\in sl(n, \mathbb C)$. We do this normalization to make $SL(2,\mb{R})/SO(2)$ of constant curvature $-1$.
\begin{df}
An $SL(n,\mb{C})$-Higgs bundle over $\Sigma$ is a pair $(E,\phi)$, where $E$ is a holomorphic vector bundle over $\Sigma$ of rank $n$ satisfying $\det E=\mathcal O$ and $\phi\in H^0(\Sigma,\text{End}(E)\bigotimes K_{\Sigma})$ is a trace-free holomorphic $\text{End}(E)$-valued $1$-form.
\end{df}
We consider the moduli space $\mc{M}_{\text{Higgs}}(\Sigma)$ consisting of gauge equivalent classes of polystable $SL(n,\mb{C})$-Higgs bundles over $\Sigma$. 
From the the non-Abelian Hodge theory \cite{Corlette}\cite{Donaldson}\cite{Hitchin87}\cite{Simpson88}, the moduli space $\mc{M}_{\text{Higgs}}(\Sigma)$ is homeomorphic to  $\mc{M}_{\text{Betti}}(S)$ and $\mc{M}_{\text{Harmonic}}(\Sigma)$:\begin{itemize}
\item $\mc{M}_{\text{Betti}}(S)$ is the moduli space consisting of conjugacy classes of reductive representations $\rho: \pi_1\rightarrow SL(n,\mathbb C)$; 
\item $\mc{M}_{\text{Harmonic}}(\Sigma)$ is the moduli space consisting of equivalent pairs $(\rho,f)$, where $\rho$ is a representation from $\pi_1$ to $SL(n,\mb{C})$ and $f$ is a $\rho$-equivariant harmonic map from the universal cover $\ti{\Sigma}$ to $X$. \end{itemize} 
We usually abuse the notation to denote both the equivalent class and the representative element. $\mc{M}_{\text{Betti}}(S)$ is also called the $SL(n,\mb{C})$-representation variety.

The moduli space $\mc{M}_{\text{Harmonic}}(\Sigma)$ can be also described as the moduli space of harmonic bundles $(E,\phi,h)$, where $(E,\phi)$ is an $SL(n,\mb{C})$-Higgs bundle over $\Sigma$ and $h$ is a harmonic metric compatible with the $SL(n,\mathbb C)$ structure, solving the Hitchin equation
$F({\nabla_h})+[\phi,\phi^{*_h}]=0,$
where $\nabla_h$ is the Chern connection uniquely determined by $h, \bar\partial_E$, $F({\nabla_h})$ is the curvature of $\nabla
_h$ and $\phi^{*_h}$ is the Hermitian adjoint of $\phi$ with respect to $h$.

In particular, from \cite{TeichOfHarmonic} and \cite{Hitchin87}, for every holomorphic quadratic differential $q_2$ on $\Sigma$, there is a unique Fuchsian representation $j: \pi_1\rightarrow SL(2,\mathbb R)$ which corresponds to the Higgs bundle $(K_{\Sigma}^{\frac{1}{2}}\oplus K_{\Sigma}^{-\frac{1}{2}}, \begin{pmatrix}0&q_2\\1&0\end{pmatrix})$. So for fixed $\Sigma$, the Fuchsian representations are parameterized by $H^0(\Sigma,K_{\Sigma}^2)$.

For a Fuchsian representation $j:\pi_1\rightarrow SL(2,\mathbb R)$, the representation $\tau_n\circ j:\pi_1\rightarrow SL(n,\mathbb R)$ is called an \textit{$n$-Fuchsian representation}. Hitchin representations are the representations $\rho:\pi_1\rightarrow PSL(n,\mb{R})$ which can be deformed to $n$-Fuchsian representations in the moduli space of $PSL(n,\mathbb R)$-representations.

Each Higgs bundle $(E,\phi)$ in $\mc{M}_{\text{Higgs}}(\Sigma)$ corresponds to a pair $(\rho,f)\in \mathcal{M}_{\text{Harmonic}}(\Sigma)$. We consider the pullback metric $g_f=f^*g_X$, which is a (possibly degenerate) symmetric $2$-tensor. From the $\rho$-equivariancy, $g_f$ descends to $\Sigma$, still denoted as $g_f$. Let $g_0$ be the uniformization hyperbolic metric over $\Sigma$. From \cite{LiSIGMA}, the Hopf differential and the energy density of $f$ are given by
\begin{equation}\label{expression}
\text{Hopf}(f):=g_f^{2,0}=2\text{tr}(\phi^2),\quad e(f)\cdot g_0:=\frac{1}{2}|df|^2_{g_X}=2\text{tr}(\phi\phi^{*_h}),
\end{equation}
where $h$ is the harmonic metric. Notice that the formulae here differ by $n$ from the formulae in \cite{LiSIGMA} because of the renormalization of the Riemannian metric on $X$. The pullback metric $g_f$ is decomposed into $(2,0)+(1,1)+(0,2)$-parts as
\begin{equation}\label{201102}
g_f=\text{Hopf}(f)+e(f)\cdot g_0+\overline{\text{Hopf}(f)}.
\end{equation}

Let $p$ be an immersed point of $f$. Denote $\kappa_f(p)$ as the extrinsic sectional curvature of the tangent plane $f_*(T_p\ti\Sigma)$ at $f(p)$. We omit the subscript $f$ if there is no confusion. The following lemma relates the extrinsic curvature $\kappa$ with the function $K$ defined in Section \ref{infi}.
\begin{lem}\label{CurvatureAlgebraicFunction}
Let $(E,\phi)$ be a polystable $SL(n,\mb{C})$-Higgs bundle over $\Sigma$ and $(\rho,f)$ be the corresponding holonomy representation and harmonic map. Let $p$ be an immersed point of $f$, then
$$\kappa(p)=-\frac{1}{2n}K(\Phi),$$ where $\Phi\in sl(n,\mathbb C)$ is the matrix presentation of $\phi(\frac{\partial}{\partial z}):E_p\rightarrow E_p$ for some local coordinate $z$ and a unitary frame of $E_p$ with respect to the harmonic metric $h$. The function $K: sl(n,\mathbb C)\rightarrow \mathbb R$
is defined in Equation (\ref{FunctionK}) in Section \ref{infi}.
\end{lem}
\begin{proof}
Denote by $\ti\phi$ the transformation map $\phi(\frac{\partial}{\partial z}): E_p\rightarrow E_p$. From \cite{LiSIGMA}, we have that $p$ is immersed if and only if $|\text{tr}(\ti\phi\ti\phi^{*_{h}})|^2>|\text{tr}(\ti\phi^2)|^2$, and  $$\kappa(p)=-\frac{1}{2n}\frac{\text{tr}([\ti\phi,\ti\phi^{*_{h}}]^2)}{|\text{tr}(\ti\phi\ti\phi^{*_{h}})|^2-|\text{tr}(\ti\phi^2)|^2}.$$

Under a unitary frame of the fiber $E_p$ with respect to $h$, the metric $h=I$ and $\ti\phi$ has a matrix presentation $\Phi\in sl(n,\mathbb{C})$, then \begin{equation*}
\kappa(p)=-\frac{1}{2n}\frac{\text{tr}([\Phi,\Phi^*]^2)}{|\Phi|^4-|\text{tr}(\Phi^2)|^2}=-\frac{1}{2n}K(\Phi).
\end{equation*}
\end{proof}

To apply the estimates of $K$, we study the eigenvalues of $\phi$. The Hitchin fibration in fact characterizes the eigenvalue information. The Hitchin fibration is a map
$p:\mc{M}_{\text{Higgs}}(\Sigma) \rightarrow\bigoplus\limits_{j=2}^nH^0(\Sigma, K_{\Sigma}^i)$ given by
$$p\big((E,\phi)\big)=(p_2(\phi),\cdots,p_n(\phi)),$$
where $p_i$ is an $SL(n,\mathbb C)$-invariant homogeneous polynomial on $sl(n,\mathbb C)$ of degree $i$ for $i=2, \cdots, n$. Two Higgs bundles $(E_1,\phi_1)$ and $(E_2,\phi_2)$ being in the same Hitchin fibers is equivalent to $\phi_1$ and $\phi_2$ having the same characteristic polynomial $\det(\lambda I_n-\phi)$, in particular the same eigenvalues at every point. Note that two Higgs bundles sharing the same Hitchin fiber is independent of the choice of the polynomials $p_2,\cdots, p_n$.

Hitchin \cite{Hitchin92} defined a section $s_p$ of this fibration, whose image exactly corresponds to the Hitchin representations from the non-Abelian Hodge theory. For suitable choice of $p_i$'s, the Hitchin section $s_p$ is given by mapping $(q_2,\cdots,q_n)\in \bigoplus\limits_{j=2}^nH^0(\Sigma, K_{\Sigma}^i)$ to \begin{eqnarray}&&\Big(E=\bigoplus\limits_{k=1}^nK_{\Sigma}^{\frac{n+1-2k}{2}},\quad \phi=\begin{pmatrix}0&r_1q_2&r_1r_2q_3&\cdots&(\prod_{i=1}^{n-2}r_i)q_{n-1}&(\prod_{i=1}^{n-1}r_i)q_n
\\r_1&0&r_2q_2&\cdots&\cdots&(\prod_{i=2}^{n-1}r_i)q_{n-1}\\&r_2&0&r_3q_2&\cdots&\vdots\\ &&\ddots&\ddots&\ddots&\vdots\\ &&&r_{n-2}&0&r_{n-1}q_2\\ &&&&r_{n-1}&0\end{pmatrix}\Big),\nonumber
\end{eqnarray}
with $r_k=\sqrt{k(n-k)}$. For the Fuchsian representation $j$ which corresponds to $q_2$, the $n$-Fuchsian representation $\tau_n\circ j$ corresponds to the Higgs bundle $s_p(q_2,0,\cdots,0)$, in other words, $\phi=\ti{e}_n+q_2e_n$, where $e_n,\ti{e}_n$ are defined in Section \ref{liealg}. So the $\text{Hopf}(f)=\frac{2(n^3-n)}{3}q_2=\frac{8}{C_{(n)}}q_2.$

Given a partition $\pi=(\lambda_1^{k_1},\cdots, \lambda_r^{k_r})\in \mc{P}_n$, define $$\tau_{\pi}:SL(2,\mb{R})\xrightarrow{(\overbrace{\tau_{\lambda_1},\cdots, \tau_{\lambda_1}}^{\text{$k_1$-times}},\cdots\cdots, \overbrace{\tau_{\lambda_r},\cdots,\tau_{\lambda_r}}^{\text{$k_r$-times}})} \prod\limits_{i=1}^{r}SL(\lambda_i,\mb{R})^{k_i}\hookrightarrow SL(n,\mb{C}).$$
Define the subgroup $\mathfrak{G}_{\pi}$ of $SU(n)$ as
\begin{equation*}
\mathfrak{G}_\pi=\{A\in \text{diag}(U(k_1)\otimes I_{\lambda_1},\cdots,U(k_r)\otimes I_{\lambda_r})\cap SU(n)\},
\end{equation*}
which lies in the centralizer of $\tau_{\pi}$ inside $SL(n,\mathbb C)$. Given two representations $j:\pi_1\rightarrow SL(2,\mb{R})$ and $ \mu_{\pi}:\pi_1\rightarrow \mathfrak{G}_{\pi}$, there is a natural well-defined representation $(\tau_{\pi}\circ j)\cdot \mu_{\pi}:\pi_1\rightarrow SL(n,\mb{C})$, $\gamma\mapsto(\tau_{\pi}\circ j)(\gamma)\cdot \mu_{\pi}(\gamma),$ where the multiplication is the matrix multiplication.

For a Fuchsian representation $j:\pi_1\rightarrow SL(2,\mathbb R)$, denote  $f_j:\ti\Sigma\rightarrow \mathbb H^2$ as the corresponding $j$-equivariant harmonic map, which is in fact a diffeomorphism.
Denote $$\bar{\tau}_{\pi}: \mb{H}^2\rightarrow X$$ as the induced map from $\tau_{\pi}$, which is injective. From Theorem 7.2 in \cite{Helgason}, $\bar{\tau}_{\pi}$ is a totally geodesic map. Then $f_{\tau_{\pi}\circ j}=\bar{\tau}_{\pi}\circ f_j$ is a harmonic map which is equivariant with respect to the representation $(\tau_{\pi}\circ j)\cdot \mu_{\pi}$ for any representation $\mu_{\pi}: \pi_1\rightarrow \mathfrak{G}_{\pi}$ and it is a totally geodesic embedding.

\subsection{Domination results in $n$-Fuchsian fibers}\label{n-Fuchsian}
Together with the curvature formula Lemma \ref{CurvatureAlgebraicFunction} and the algebraic inequality in Theorem \ref{main}, we obtain the estimate of the extrinsic curvature. The following Proposition \ref{IntrinsicCurvatureImpliesEnergyDensity} is the key reason why the extrinsic curvature of the equivariant harmonic maps will deduce the domination results for the harmonic maps and the associated representations. We will postpone the proof of Proposition \ref{IntrinsicCurvatureImpliesEnergyDensity} until the end of Section \ref{ProofOfProposition}.
\begin{prop}\label{IntrinsicCurvatureImpliesEnergyDensity}
Let $f:\ti{\Sigma} \rightarrow X$ be a $\rho$-equivariant harmonic map. Suppose there is a positive constant $c$ such that the extrinsic curvature $\kappa(p)\leq -c$
for each immersed point $p$. (If there is no immersed points, $c$ can be arbitrary positive constant.) Let $j$ be the Fuchsian representation such that, the Hopf differential of the corresponding $j$-equivariant harmonic map $f_j$ is $c\cdot\text{Hopf}(f).$ Then the energy density satisfies $$e(f)\leq \frac{1}{c}e(f_j).$$
Moreover, the equality holds at one point if and only if there exists a partition $\pi\in \mathcal P_n$ and an element $x\in SL(n,\mb{C})$, such that
$$c=\frac{1}{2}C_{\pi},\quad\rho=\text{Ad}_{x^{-1}}\circ \big((\tau_{\pi}\circ j)\cdot \mu_{\pi}\big), \quad f=L_x\circ\bar\tau_{\pi}\circ f_j,$$
for some representation $\mu_{\pi}:\pi_1\rightarrow \mathfrak{G}_\pi$.
\end{prop}
\begin{rem}
Under the assumptions of Proposition \ref{IntrinsicCurvatureImpliesEnergyDensity}, using the same method, the assertion $e(f)\leq \frac{1}{c}e(f_j)$ still holds if $X=SL(n,\mb{C})/SU(n)$ is replaced by $G/K$ for any reductive Lie group $G$.
\end{rem}
We obtain the domination of the pullback metric from the domination of the energy density.
\begin{cor}
Under the assumptions of Proposition \ref{IntrinsicCurvatureImpliesEnergyDensity}, the pullback metric satisfies $f_{\rho}^*g_X\leq \frac{1}{c}f_j^*g_{\mb{H}^2}$. If the equality holds at one vector, then the same condition holds as in Proposition \ref{IntrinsicCurvatureImpliesEnergyDensity}.
\end{cor}
\begin{proof}
From Equation (\ref{201102}), the conformal factor of the $(1,1)$ part of the pullback metric is the energy density, for which we have the domination results from Proposition \ref{IntrinsicCurvatureImpliesEnergyDensity}. For the $(2,0)$ part, we have $g^{2,0}=\text{Hopf}(f)=\frac{1}{c}\text{Hopf}(f_j)$. The $(0,2)$ part is just the conjugation of $(2,0)$ part. So we finish the proof.
\end{proof}
In the following, we collect several useful concepts of a representation $\rho: \pi_1\rightarrow SL(n,\mathbb C)$.
\begin{itemize}
\item The translation length spectrum of $\pi_1$ with respect to a representation $\rho$ is defined by
\begin{equation*}
l_{\rho}(\gamma):=\inf_{x\in X}d(x,\rho(\gamma)x),\quad id\neq \gamma\in \pi_1,
\end{equation*} where $d(\cdot,\cdot)$ is the distance induced by the Riemannian metric $g_X$ on $X$.
\item The entropy of a representation $\rho$ is defined as
\begin{eqnarray*}
h(\rho):=\underset{R\rightarrow \infty}{\text{lim sup }}\frac{\text{log}(\#\{\gamma\in \pi_1|l_{\rho}(\gamma)\leq R\})}{R}.
\end{eqnarray*}Note that $h(\rho)$ can be $+\infty$.
\end{itemize}
The translation length spectrum and entropy of a representation $j:\pi_1\rightarrow SL(2,\mathbb R)$ are defined similarly.

In the following lemma, we deduce the domination of the translation length spectrum from the domination of the pullback metric for Fuchsian representations.
\begin{lem}\label{LengthSpectrum}
Let $j$ be a Fuchsian representation and $f_j$ be the corresponding $j$-equivariant harmonic map. Suppose $\rho: \pi_1\rightarrow SL(n,\mathbb C)$ is a reductive representation such that the corresponding $\rho$-equivariant harmonic map $f:\widetilde \Sigma\rightarrow X$ satisfies $f^*g_X\leq \frac{1}{c}f_j^*g_{\mb{H}^2}$, for some $c>0$. Then $l_{\rho}\leq \frac{1}{\sqrt{c}}l_{j}$.
\end{lem}
\begin{proof}
We associate to a Riemannian metric $g$ on $\ti{\Sigma}$, a length function on $\pi_1$ given by
$$l^{g}(\gamma):=\inf\limits_{x\in \ti{\Sigma}}d_{g}(x,\gamma x).$$
For any non-identity element $\ga\in \pi_1$,
\begin{eqnarray*}
l_{\rho}(\ga)=\inf\limits_{x\in X}d_X(x,\rho(\gamma)x)\leq \inf\limits_{x\in f(\widetilde{\Sigma})}d_X(x,\rho(\gamma)x)\leq\inf\limits_{x\in f(\widetilde{\Sigma})}d_{f(\widetilde{\Sigma})}(x,\rho(\gamma)x).
\end{eqnarray*}
From the $\rho$-equivariancy, $\inf\limits_{x\in f(\widetilde{\Sigma})}d_{f(\widetilde{\Sigma})}(x,\rho(\gamma)x)=\inf\limits_{y\in \widetilde{\Sigma}}d_{f^*g_X}(y,\gamma y)=l^{f^*g_X}(\gamma),$ and thus $l_{\rho}(\gamma)\leq l^{f^*g_X}(\gamma)$. By the assumption $f^*g_X\leq \frac{1}{c}f_j^*g_{\mb{H}^2}$, we have $l_{\rho}(\gamma)\leq \frac{1}{\sqrt{c}}l^{f_j^*g_{\mb{H}^2}}(\gamma).$
Since $f_j$ is a diffeomorphism, it is clear that $l_j=l^{f_j^*g_{\mb{H}^2}}$. So we finish the proof.
\end{proof}
The relation of the geometric invariants between the Fuchsian representation $j$ and the $n$-Fuchsian representation $\tau_{\pi}\circ j$ is as follows.
\begin{lem}\label{standard}
Let $j$ be a Fuchsian representation, $f_j$ be the corresponding $j$-equivariant harmonic map. Let $\pi\in\mc{P}_{n}$, $f_{\tau_{\pi}\circ j}$ be the corresponding $(\tau_{\pi}\circ j)$-equivariant harmonic map. Let $c=\frac{1}{2}C_{\pi}$. Then
$\kappa_{\bar\tau_{\pi}\circ f_j}=-c$, $\frac{1}{c}f_j^*g_{\mb{H}^2}=f_{\tau_{\pi}\circ j}^*g_X$, $l_{\tau_{\pi}\circ j}=\frac{1}{\sqrt{c}}l_j$.
\end{lem}

\begin{proof}
Since $f_{\tau_{\pi}\circ j}$ is a totally geodesic embedding, $\kappa_{\tau_{\pi}\circ j}$ is just the sectional curvature of the image of $\bar{\tau}_n$ in $X$. Then $T_{[e]}X$ can be identified with $\mathfrak{p}=\{A: A=A^*, \text{tr}A=0\}$. Recall the Killing form is given by $\big<A,B\big>=2\text{tr}(AB)$, $A,B\in \mathfrak{p}$. From Theorem 4.2 in \cite{Helgason}, the sectional curvature of the plane spanned by $A,B$ is $\kappa=-\frac{|[A,B]|^2}{|A|^2|B|^2-|<A,B>|^2}$. One can easily check $\kappa=-\frac{1}{2}K(U)$, for $U=A+\sqrt{-1}B$. From the definition of $\tau_{\pi}$, the tangent plane of the image of $f_{\tau_{\pi}\circ j}$ is spanned by
$$A=\text{diag}(e_{n_1},\cdots, e_{n_s})+\text{diag}(\ti e_{n_1},\cdots, \ti e_{n_s}), \quad B=\text{diag}(x_{n_1},\cdots, x_{n_s}).$$ 
From Lemma \ref{C(K)},
we obtain $K(U)=-\frac{1}{2}C_{\pi}$. So $\kappa_{\bar\tau_{\pi}\circ f_j}=-c$.
Therefore, $\frac{1}{c}f_j^*g_{\mb{H}^2}=f_{\tau_{\pi}\circ j}^*g_X$.

To show $l_{\tau_{\pi}\circ j}=\frac{1}{\sqrt{c}}l_j$, note that the image $f_{\tau_{\pi}\circ j}(\ti{\Sigma})$ is a totally geodesic plane $\mathcal P$ inside $X$, also acted by $\pi_1$. Clearly, $$\inf\limits_{x\in X}d_X(x,(\tau_{\pi}\circ j)(\gamma)x)\leq \inf\limits_{x\in \mathcal P}d_X(x,(\tau_{\pi}\circ j)(\gamma)x).$$ Next we show $$\inf\limits_{x\in X}d_X(x,(\tau_{\pi}\circ j)(\gamma)x)\geq \inf\limits_{x\in \mathcal P}d_X(x,(\tau_{\pi}\circ j)(\gamma)x).$$ 
Since the symmetric space $X$ is of non-positive curvature, from Proposition 2.4 on Page 176 in \cite{BH}, the projection map $p$ to $\mathcal P$ is well-defined and distance-decreasing, that is,
$$d_X(x,y)\geq d_{\mathcal P}(p(x), p(y)).$$
Since $SL(n,\mathbb C)$ acts isometrically on $X$, the projection map $p$ is also equivariant, $$p((\tau_{\pi}\circ j)(\gamma)x)=(\tau_{\pi}\circ j)(\gamma)(p(x)).$$ 
Hence
\begin{\eq}
\inf\limits_{x\in X}d_X\big(x,(\tau_{\pi}\circ j)(\gamma)x\big)&\geq& \inf\limits_{x\in X}d_{\mathcal P}\big(p(x),p((\tau_{\pi}\circ j)(\gamma)x)\big)\\ 
&=&\inf\limits_{x\in X}d_{\mathcal P}\big(p(x),(\tau_{\pi}\circ j)(\gamma)(p(x))\big)\\
&=& \inf\limits_{y\in \mathcal P}d_{\mathcal P}\big(y,(\tau_{\pi}\circ j)(\gamma)y\big).
\end{\eq}
Therefore,  $l_{\tau_{\pi}\circ j}= l^{f_{\tau_{\pi}\circ j}^*g_X}=\frac{1}{\sqrt{c}}l^{f_j^*g_{\mb{H}^2}}=\frac{1}{\sqrt{c}}l_j$.
\end{proof}

Now we are in the position to show the main theorem proving that $n$-Fuchsian representations are maximal in their Hitchin fibers as follows.
\begin{thm}\label{dommain}
Let $\tau_n\circ j$ be an $n$-Fuchsian representation. Given a Riemann surface structure $\Sigma=(S,J)$ on $S$. Suppose $\rho\in \mc{M}_{\text{Betti}}(S)$ corresponds to the Higgs bundle $(E,\phi)$ which is in the same Hitchin fiber as $\tau_n\circ j$ in $\mc{M}_{\text{Higgs}}(\Sigma)$. Let $f:\ti\Sigma\rightarrow X$ be the corresponding $\rho$-equivariant harmonic map. Then \\
(1) the energy density satisfies $e(f)<  e(\bar\tau_n\circ f_j)$;\\
(2) the pullback metric satisfies $g_{f}< g_{\bar\tau_n\circ f_j}$;\\
(3) the translation length spectrum satisfies $l_\rho\leq \lambda\cdot l_{ \tau_n\circ j}$ for some positive constant $\lambda<1$;\\
(4) the energy satisfies $E(f)< E(\bar\tau_n\circ f_j);$ \\
(5) the entropy satisfies $h(\rho)>h(\tau_n\circ j)=\sqrt{\frac{6}{n^3-n}}$,\\
unless $\rho$ is conjugate to $(\tau_n\circ j)\cdot\mu_{(n)}$ for some representation $\mu_{(n)}:\pi_1\rightarrow \mathfrak{G}_{(n)}=\{e^{\frac{2k\pi\sqrt{-1}}{n}},k=1,\cdots,n\}\cdot I_n $, in which case, it has the same harmonic map and the same translation length spectrum as $\tau_n\circ j$.
\end{thm}
\begin{proof}
Suppose $j$ corresponds to the Higgs bundle parameterized by $q_2$, then $\tau_n\circ j$ corresponds to the Higgs bundle $s_p(q_2,0,\cdots, 0)$ where the Higgs field $\phi$ has the same eigenvalues as $q_2e_n+\ti{e}_n$. The Higgs field $\phi$ has the eigenvalues of the type $(n-1,n-3,\cdots, 3-n, 1-n)\sqrt{q_2}$. From Lemma \ref{CurvatureAlgebraicFunction} and Corollary \ref{maincor}, we have at each immersed point $p$, $\kappa(p)=-\frac{1}{2}K(\Phi)\leq -\frac{1}{2}C_{(n)}=-\frac{6}{n^3-n}.$

By Proposition \ref{IntrinsicCurvatureImpliesEnergyDensity} we obtain either $e(f)<e(\tau_n\circ j)$ or $\rho$ is conjugate to $(\tau_n\circ j)\cdot \mu_{(n)}$ for some representation $\mu_{(n)}:\pi_1\rightarrow \mathfrak{G}_{(n)}$.

We only need to consider the first case, which also implies the domination of the energy. Then for any $X\in T\ti\Sigma$, $g^{(1,1)}_{\bar\tau_n\circ f_j}(X,X)-g^{(1,1)}_{f}(X,X)>0$. By the assumption, the Higgs bundles for $\rho$ and $\tau_n\circ j$ are in the same Hitchin fiber, hence $f$ and $\bar\tau_n\circ f_j$ have the same Hopf differential and thus $g_{\bar\tau_n\circ f_j}(X,X)>g_{f}(X,X).$ Since $S$ is compact, we have $g_{f}(X,X)<\lam^2g_{\bar\tau_n\circ f_j}(X,X)$ for some $0<\lam<1$, that is $g_{f}<\lam^2g_{\bar\tau_n\circ f_j}$. Then from Lemma \ref{LengthSpectrum}, we obtain $l_{\rho}\leq \lambda\cdot l_{\tau_n\circ j}$ and $\rho$ is strictly dominated by $\tau_n\circ j$.

Since for a Fuchsian representation $j$, its volume entropy $h(j)=1$, By Lemma \ref{standard}, $l_{\tau_n\circ j}=\frac{1}{\sqrt{c}}l_j$, where $c=\frac{1}{2}C_{(n)}$. Then for the $n$-Fuchsian representation $\tau_n\circ j$, its volume entropy $h(\tau_n\circ j)=\sqrt{c}.$ Since $l_{\rho}<l_{\tau_n\circ j}$, then $h(\rho)> h(\tau_n\circ j)=\sqrt{\frac{6}{n^3-n}}$. We finish the proof. 
\end{proof}
\begin{rem}
The representation $\mu_{(n)}$ appears naturally in the rigidity part, since the Higgs bundles $(E,\phi)$ and $(E,\phi)\otimes L$ give the same harmonic map and in the $SL(n,\mb{C})$ setting $L$ must be an $n$th root of the trivial bundle. We may consider $PSL(n,\mb{C})$-representations and $PSL(n,\mb{C})$-Higgs bundles to avoid $\mu_{(n)}$.
\end{rem}
\begin{rem}
In the $SL(2,\mathbb C)$ case, every Higgs bundle is in the Hitchin fiber at $(q_2)$ for some $q_2$. This case is shown by Deroin and Tholozan \cite{DominationFuchsian}.
\end{rem}

In the following proposition, we find another $n$-Fuchsian representation  which also dominates the representation $\rho$ sharing the Hitchin fiber with $\tau_n\circ j$.
\begin{prop}\label{domcon}
Let $\rho$ be in an $n$-Fuchsian fiber in $\mc{M}_{\text{Higgs}}(\Sigma)$. Let $\hat{j}$ be the Fuchsian representation corresponding to the complex structure $\hat{J}$ determined by the pullback metric $f^*g_X$. Then either $\rho$ is strictly dominated by $\tau_n\circ \hat{j}$ or $\rho$ is conjugate to $(\tau_n\circ \hat{j})\cdot\mu_{(n)}$ for some representation $\mu_{(n)}:\pi_1\rightarrow\mathfrak{G}_{(n)}$.
\end{prop}
\begin{proof}
Similar to the beginning of the proof in Theorem \ref{dommain}, at each immersed point $p$, $\kappa(p)\leq -\frac{1}{2}C_{(n)}=-\frac{6}{n^3-n}.$
Let $\hat{\Sigma}=(S,\hat{J})$. Let $\hat{g}$ be the unique uniformization hyperbolic metric over $\hat{\Sigma}$. Let $f_{\hat{j}}$ be the harmonic map corresponding to $\hat{j}$ with respect to $\hat{\Sigma}$. Then $f_{\hat{j}}^*g_{\mb{H}}=\hat{g}$. By Lemma \ref{indcurv} below, the Gaussian curvature $k$ of $f^*g_X$, satisfies $k\leq \kappa\leq c=-\frac{1}{2}C_{(n)}$. From the strong maximum principle, either $f^*g_X< \frac{1}{c}\hat{g}$ or $f^*g_X\equiv \frac{1}{c}\hat{g}$. From the rigidity part of Proposition \ref{IntrinsicCurvatureImpliesEnergyDensity}, if $f^*g_X\equiv \frac{1}{c}\hat{g}$ then $\rho$ is conjugate to $(\tau_n\circ \hat{j})\cdot\mu_{(n)}$ for some $\mu_{(n)}$. Now suppose $\rho$ is not conjugate to $\hat{\rho}\cdot \mu_{(n)}$. Since the metrics can descend to $S$, there is a constant $0<\lam<1$ such that $f^*g_X< \frac{\lam^2}{c}\hat{g}$. Then from Lemma \ref{LengthSpectrum} and Lemma \ref{standard}, we obtain $l_{\rho}\leq  \frac{\lam}{\sqrt{c}}\cdot l_j=\lambda\cdot l_{\tau_n\circ j}$ and $\rho$ is strictly dominated by $\tau_n\circ \hat j$.
\end{proof}

\begin{rem}
Given a representation $\rho$ sharing the Hitchin fiber with an $n$-Fuchsian representation $\tau_n\circ j$ for some Riemann surface $\Sigma$ such that $\rho$ is not conjugate to $(\tau_n\circ j)\cdot\mu_{(n)}$, from Theorem \ref{dommain}, we see $\tau_n\circ j$ strictly dominates $\rho$. If we only aim to find an $n$-Fuchsian representation which dominates $\rho$, the representation $\tau_n\circ \hat j$ from Proposition \ref{domcon} is enough and easier to find. Note that the two $n$-Fuchsian representations $\tau_n\circ j$ and $\tau_n\circ \hat j$ never coincide unless $\rho$ is in the nilpotent cone for the Riemann surface $\Sigma$.
\end{rem}
\subsection{Domination results in other cases}\label{others}
Next we consider a more general family and show the domination results.
\begin{thm}\label{dommainUniReal}
Let $(E,\phi)\in \mc{M}_{\text{Higgs}}(\Sigma)$ corresponding to $(\rho,f)$. Suppose $\phi$ is diagonalizable on a dense set of $\Sigma$. Suppose the characteristic polynomial $\det(\lambda I_n-\phi)$ of $\phi$ has the form\\
(a) $(\lambda^2-a_1^2q)(\lambda^2-a_2^2q)\cdots (\lambda^2-a_{[\frac{n}{2}]}^2q)\lambda^{n-2[\frac{n}{2}]}$, for $q\in H^0(\Sigma, K_\Sigma^2)$ and $a_i\in \mb{R}$, $i=1,\cdots,[\frac{n}{2}]$, or\\
(b) $(\lambda-b_1\omega)(\lambda-b_2\omega)\cdots (\lambda-b_n\omega)$ satisfying $b_i\in \mb{R}$, $i=1,\cdots,n$ and $\sum\limits_{i=1}^nb_i=0$, for $\omega\in H^0(\Sigma, K_\Sigma)$.\\
Let $(\lambda_1,\cdots, \lambda_n)=(\pm a_1, \cdots, \pm a_{[\frac{n}{2}]}, (0))$ or $(b_1,\cdots, b_n)$ and $c=\frac{1}{2}\frac{\min\limits_{\lambda_i\neq \lambda_j}(\lambda_i-\lambda_j)^2}{\sum\limits_{i=1}^n\lambda_i^2}$.
Let $j$ be the Fuchsian representation which corresponds to $q_2=\frac{c}{4}\text{tr}(\phi^2)$. Then\\
(1) the energy density satisfies $e(f)<  \frac{1}{c}e(f_j)$;\\
(2) the pullback metric satisfies $g_{f}< \frac{1}{c}g_{f_j}$;\\
(3) the translation length spectrum satisfies $l_\rho< \lambda\cdot \frac{1}{\sqrt{c}}l_{j}$ for some positive constant $\lambda<1$;\\
(4) the energy satisfies $E(f)< \frac{1}{c}E(f_j);$ \\
(5) the entropy satisfies $h(\rho)>\sqrt{c}h(j)=\sqrt{c}$,\\
unless $(\lambda_1,\cdots, \lambda_n)$ has the form $t(\Lambda_{n_1},\cdots,\Lambda_{n_s})$, $t\in \mb{C}^*$, for some even partition $\pi=(n_1,\cdots,n_s)\in\mc{P}_n$, and $\rho$ is conjugate to $(\tau_{\pi}\circ j)\cdot\mu_{\pi}$ for some representation $\mu_{\pi}:\pi_1\rightarrow \mathfrak{G}_{\pi}$, in which case, it has the same harmonic map and the same translation length spectrum as $\tau_{\pi}\circ j$.
\end{thm}
\begin{proof}
Since $f$ is harmonic, from Sampson \cite{Sampson}, the set of the immersed points is either open and dense or empty. If it is empty, from the proof of Lemma \ref{HarmonicDomination} below, we have  $$e(f)=2|\text{Hopf}(f)|=\frac{2}{c}|\text{Hopf}(f_j)|<  \frac{1}{c}e(f_j).$$ So we assume the set of immersed points is open and dense, denoted as $U$. From Corollary \ref{eurhb}, at the point $p\in U$, at which $\phi$ is diagonalizable, we have $\kappa(p)\leq-c$. Since $\phi$ is diagonalizable on a dense set, we obtain $\kappa(p)\leq-c$ on $U$. Then by using the similar argument in the proof of Theorem \ref{dommain}, we finish the proof.
\end{proof}
\begin{rem}
If the eigenvalues $(\lambda_1,\cdots, \lambda_n)$ in Theorem \ref{dommainUniReal} are distinct, then the Higgs field is automatically diagonalizable. So the estimates in Theorem \ref{dommainUniReal} hold for the whole Hitchin fiber. In this case, the inequality is strictly unless the representation $\rho$ is $n$-Fuchsian.
\end{rem}
\begin{rem}
(1) The fiber at $(\lambda^2-a_1^2q)(\lambda^2-a_2^2q)\cdots (\lambda^2-a_{[\frac{n}{2}]}^2q)\lambda^{n-2[\frac{n}{2}]}$ contains the Higgs bundle with the corresponding representation $\text{diag}(\rho_1,\rho_2,\cdots, \rho_{[\frac{n}{2}]}, 1)$, where each $\rho_i: \pi_1(S)\rightarrow SL(2,\mathbb R)$ is the Fuchsian representation corresponding to the Higgs bundle parameterized by $a_i^2q$.\\
(2) The fiber at $(\lambda-a_1\omega)(\lambda-a_2\omega)\cdots (\lambda-a_n\omega)$ contains the Higgs bundle with the corresponding representation $\text{diag}(\rho_1,\cdots, \rho_n)$, where each $\rho_i: \pi_1(S)\rightarrow \mathbb C^*$ is the representation corresponding to the Higgs bundle $(L_i, a_i\omega)$ for some holomorphic line bundle $L_i$ of degree $0$ satisfying $\prod\limits_{i=1}^nL_i=\mathcal O$. \\
(3) The fiber at $(\lambda^2-(n-1)^2q)(\lambda^2-(n-3)^2q)\cdots (\lambda^2-(n+1-2[\frac{n}{2}])^2q)\lambda^{n-2[\frac{n}{2}]}$ contains the Higgs bundle with the corresponding representation $\tau_n\circ j$, where $j: \pi_1(S)\rightarrow SL(2,\mathbb R)$ is the Fuchsian representation corresponding to the Higgs bundle parameterized by $q$.
\end{rem}

In the next proposition, we show that the cases (a)(b) in Theorem \ref{dommainUniReal} are exactly the cases when the eigenvalues of the Higgs field $\phi$ are uni-real everywhere.

\begin{prop}\label{CharacterizationUniReal}
For a Higgs bundle $(E,\phi)$, the eigenvalues of $\phi$ are uni-real everywhere on $\Sigma$ if and only if the characteristic polynomial
\begin{equation}\label{uni1}
\det(\lambda I_n-\phi)=(\lambda^2-a_1^2q)(\lambda^2-a_2^2q)\cdots (\lambda^2-a_{[\frac{n}{2}]}^2q)\lambda^{n-2[\frac{n}{2}]},
\end{equation}
for $q\in H^0(\Sigma, K_{\Sigma}^2)$ and $a_i\in \mb{R}$, $i=1,\cdots,[\frac{n}{2}]$, or
\begin{equation}\label{unireal2}
\det(\lambda I_n-\phi)=(\lambda-b_1\omega)(\lambda-b_2\omega)\cdots (\lambda-b_n\omega)
\end{equation}
for $\omega\in H^0(\Sigma, K_{\Sigma})$ and $b_i\in \mb{R}$, $i=1,\cdots,n$, $\sum\limits_{i=1}^nb_i=0$.
\end{prop}
\begin{proof}
If the Hitchin fiber is uni-real, the coefficients in the characteristic polynomial
$$\det(\lambda I_n-\phi)=\lambda^n+\sum_{k=2}^nq_k\lambda^{n-k}$$ satisfies that for any $2\leq i<j\leq n$, there exists real numbers $c_1,c_2$, $c_1^2+c_2^2\neq 0$, such that $c_1q_i^j+c_2q_j^i=0$, since the real valued meromorphic function must be a constant.

Suppose at each point $z$, the eigenvalues (with multiple in the characteristic polynomial) of $\phi(z)$ equals $c(\lambda_1,\cdots, \lambda_n)$, where $\lambda_1,\cdots, \lambda_n\in \mathbb R$, $\sum\limits_{i=1}^{n}\lambda_i=0$ and $c\in \mathbb C^*$. Then $q_2(z)=c^2\sum\limits_{i=1}^{n}\lambda_i^2.$ Therefore, $q_2(z)=0$ if and only if $\lambda_1=\cdots =\lambda_n=0$. That is, $q_2(z)=0$ at point $z$ if and only if $q_3(z)=\cdots=q_n(z)=0$ at point $z$. So on $\Sigma$ either $q_2=0$ which implies $\phi$ is nilpotent or $q_2$ is nonzero.

We only need to check the case $q_2$ is nonzero. Clearly, $q_{2k}=c_kq_2^k$ for some real constant $c_k$.

If $q_{2k+1}=0$ for every $k$, then $\det(\lambda I_n-\phi)=\lambda^n+\sum_{k=1}^{[\frac{n}{2}]}c_kq_2^k\lambda^{n-2k}.$ So the solution at each point will be $(\pm a_1,\cdots, \pm a_{[\frac{n}{2}]}, (0))\sqrt{q_2},$ where $a_i$'s are either real or purely imaginary. Since the eigenvalue of $\phi$ is uni-real, we may assume $a_i$'s are real. So the characteristic polynomial has the expression (\ref{uni1}).

If there exists $k_0$ such that $q_{2k_0+1}\neq 0$, we have $q_{2k_0+1}^2=d_{k_0}q_2^{2k_0+1}$, for some real constant $d_{k_0}\neq 0$. So $\omega=q_{2k_0+1}/q_2^{k_0}$ is a well-defined nonzero holomorphic $1$-form and $q_2=\frac{1}{d_{k_0}}\omega^2$. Set $q_{k}=c_k\omega^{k}$ for some real constant $c_k$, $k=2,\cdots,n$. Then $\det(\lambda I_n-\phi)=\lambda^n+\sum_{k=2}^n c_k\omega^k\lambda^{n-k}. $ So the solution at each point will be $(b_1,\cdots, b_n)\cdot\omega,$ where $b_i\in \mathbb C$ satisfying $\sum\limits_{i=1}^{n}b_i=0$. Since the eigenvalue of $\phi$ is uni-real, we may assume $b_i$'s are real. So the characteristic polynomial has the expression (\ref{unireal2}).
\end{proof}
In the next Theorem, we impose an assumption on the rank of the Higgs field $\phi$ instead of the eigenvalues.
\begin{thm}\label{dommainRank2}
Let $(E,\phi)\in \mc{M}_{\text{Higgs}}(\Sigma)$ corresponding to $(\rho,f)$. Suppose the rank of $\phi$ is at most $2$ at every point on $\Sigma$. Let $j$ be the Fuchsian representation corresponding to the Higgs bundle parameterized by $q_2=\frac{1}{8}tr(\phi^2)$. Let $\pi_3=(3, 1,\cdots, 1)\in \mathcal P_n$. Then\\
(1) the energy density satisfies $e(f)< e(f_{\tau_{\pi_3}\circ j})$;\\
(2) the pullback metric satisfies $g_{f}< g_{f_{\tau_{\pi_3}\circ j}}$;\\
(3) the translation length spectrum satisfies $l_\rho\leq \lambda \cdot l_{\tau_{\pi_3}\circ j}$ for some positive constant $\lambda<1$;\\
(4) the energy satisfies $E(f)< E(f_{\tau_{\pi_3}\circ j})$;\\
(5) the entropy satisfies $h(\rho)> h({\tau_{\pi_3}\circ j})=\frac{1}{2}$,\\
unless $\rho$ is conjugate to $\rho_{\pi_3}\cdot\mu_{\pi_3}$ for some representation $\mu_{\pi_3}:\pi_1\rightarrow \mathfrak{G}_{\pi_3}=\{\text{diag}(cI_3, U(n-3))\bigcap SU(n)\}$, in which case, it has the same harmonic map and the same translation length spectrum as $\tau_{\pi_3}\circ j$.
\end{thm}
\begin{proof}
If the rank of the Higgs field is at most $2$, from Corollary \ref{rank2Inequality} and Lemma \ref{CurvatureAlgebraicFunction}, we obtain at each immersed point $p$, $\kappa(p)\leq -\frac{1}{2}C_{\pi_3}=-\frac{1}{4}$. The remaining proof is similar to Theorem \ref{dommain}.
\end{proof}

\section{Proof of Proposition \ref{IntrinsicCurvatureImpliesEnergyDensity}}\label{ProofOfProposition}
This whole section is devoted to prove Proposition \ref{IntrinsicCurvatureImpliesEnergyDensity}, which plays an important role in the proof of the theorems in the previous section. We use the same notations as in Section \ref{domination results}. 
\subsection{Intrinsic curvature}
Now we consider the Gauss curvature $k$ of the pullback metric $f^*g_X$ on $\ti{\Sigma}$. The pullback metric and the immersed points can descend to $\Sigma$ and we abuse the same notation if there is no confusion. 
\begin{lem}\label{indcurv}
Let $f:\ti{\Sigma} \rightarrow (X,g_X)$ be a harmonic map. Let $p$ be an immersed point. Suppose $\kappa(p)\leq -c$. Then $k(p)\leq -c$. The equality holds if and only if $\kappa(p)=-c$ and $f$ is totally geodesic at $p$.
\end{lem}
\begin{proof}
Let $e_1,e_2$ be an orthonormal basis of the induced metric at $p$. Then from the Gauss equation, at $p$
\begin{eqnarray*}
k=\kappa+\big<II(e_1,e_1),II(e_2,e_2)\big>-|II(e_1,e_2)|^2,
\end{eqnarray*}
where $II$ is the second fundamental form of $f$ defined by $II(X,Y)=(\nabla_{f_*X}f_*Y)^{\perp}$.

Let $\sigma_1,\sigma_2$ be an orthonormal basis of $g_0$ at $p$, where $g_0$ is the hyperbolic metric with respect to $\ti{\Sigma}$.
The harmonicity condition for $f$ means
\begin{equation*}
\text{tr}_{g_0}\nabla df=\nabla df(\sigma_1,\sigma_1)+\nabla df(\sigma_2,\sigma_2)=0,
\end{equation*}
where $\nabla df(X,Y)=\nabla_{f_*X}f_*Y-f_*(\nabla_XY)$.\\
By projection to the normal bundle, the equation above implies $$II(\sigma_1,\sigma_1)+II(\sigma_2,\sigma_2)=0.$$
Denote
\begin{\eq}
x=II(\sigma_1,\sigma_1)=-II(\sigma_2,\sigma_2),\quad y=II(\sigma_1,\sigma_2)=II(\sigma_2,\sigma_1).
\end{\eq}
Set $e_1=a\sigma_1+b\sigma_2, e_2=c\sigma_1+d\sigma_2.$
Note that $ad-bc\neq 0$. So
\begin{eqnarray*}
&&<II(e_1,e_1),II(e_2,e_2)>-|II(e_1,e_2)|^2\\
&=&<(a^2-b^2)x+2ab y,(c^2-d^2)+2cd y>-|(ac-bd)x+(bc+ad)y|^2\\
&=&-(bc-ad)^2(|x|^2+|y|^2)\leq 0.
\end{eqnarray*}
So $k(p)\leq \kappa(p)\leq -c$. The equality holds if and only if $\kappa(p)=-c$ and $II(p)=0$.
\end{proof}

\subsection{Energy density}\label{Domination of metrics}
In this subsection, we derive the domination of the energy density from the domination of the curvature. We leave the rigidity part to next subsection. Our proof is mainly adapted on Deroin-Tholozan \cite{DominationFuchsian} and Wan \cite{Wan}, except that we do not need the target manifold is negatively curved but the curvature of the tangent plane is negatived curved.

First we recall some calculations in \cite{SY}. Let $(\Sigma,\sigma(z)|dz|^2)$ be a Riemann surface with a K\"ahler metric. Let $(M,g)$ be a Riemannian manifold. Suppose $f:(\Sigma,\sigma(z)|dz|^2)\rightarrow (M,g)$ is a harmonic map. Let $p$ be an immersed point of $f$. Then locally, in a neighborhood $U$ of $p$, $f|_{U}:U\rightarrow f(U)$ is a diffeomorphism and also harmonic with respect to the induced metric. We can choose a local complex coordinate system $\{u\}$ in $f(U)$ such that the induced metric is $\mu(u)|du|^2$.

Then locally, the harmonicity of $f|_U$ reads as $$u_{z\bar{z}}+\frac{\partial \log \mu}{\partial u}u_{z}u_{\bar{z}}=0.$$
Denote $\partial u$ as the $(1,0),(1,0)$ part of $du\in T^{*}\Sigma\bigotimes f^*Tf(U)\bigotimes \mb{C}$ with respect to the complex structure of $\Sigma$ and $f(U)$. Similarly denote $\bar{\partial} u$ as the $(0,1),(1,0)$ part. Locally, $$\partial u=u_zdz\otimes \frac{\partial}{\partial u}\text{ and }\bar{\partial} u=u_{\bar{z}}d\bar{z}\otimes \frac{\partial}{\partial u}.$$
Set $H=||\partial u||^2=|u_z|^2\frac{\mu}{\sigma}$, $L=||\bar{\partial} u||^2=|u_{\bar{z}}|^2\frac{\mu}{\sigma}$. Then the energy density $e(f):=\frac{1}{2}||df||_{\sigma, g}^2=H+L$. The Jacobian $J(f)=H-L$. The Hopf differential $\text{Hopf}(f):=(f^*g)^{(2,0)}=u_z\bar{u}_z\mu dz\otimes d\bar{z}$, so $||\text{Hopf}(f)||_{\sigma}^2=HL$. Let $k_{\sigma},k_{\mu}$ be the Gauss curvature of $\Sigma, f(U)$ respectively. Then
\\
at nonzero of $H$,
\begin{equation}\label{H}
\triangle_{\sigma}\log H=-2k_{\mu} H+2k_{\mu} L+2k_{\sigma},
\end{equation}
at nonzero of $L$,
\begin{equation}\label{L}
\triangle_{\sigma}\log L=-2k_{\mu} L+2k_{\mu}H+2k_{\sigma}.
\end{equation}
\begin{lem}\label{HarmonicDomination}
For any $\rho$-equivariant harmonic map $f:\widetilde\Sigma\rightarrow X$, suppose the curvature $k_{\mu}$ of the corresponding pullback metric satisfies $k_{\mu}\leq -c$ for some constant $c>0$ at the immersed points. (If there is no immersed points, $c>0$ can be arbitrary.) Set $q=\text{Hopf}(f)$.
Then the energy density $e(f)\leq \frac{1}{c}e(f_j)$, where $f_j: \widetilde\Sigma\rightarrow \mathbb H^2$ is the unique $j$-equivariant harmonic map with Hopf differential $cq$ for a suitable Fuchsian representation $j:\pi_1\rightarrow SL(2,\mb{R})$. 
If the equality holds at one point, then $k_{\mu}\equiv -c$.
\end{lem}
\begin{proof}
Let $U$ be the set of the immersed points of $f$ in $\ti{\Sigma}$. Since $f$ is harmonic, from Sampson \cite{Sampson}, U is either open and dense or empty. Firs we suppose $U$ is nonempty. Let $g_0$ be the hyperbolic metric corresponding to $\ti{\Sigma}$. As the discussion above, we consider $\rho$-equivariant harmonic map $f:(U, g_0|dz|^2)\rightarrow (f(U), \mu(u)|du|^2)$ and the quantity $H,L$ satisfying Equation (\ref{H})(\ref{L}). We use the similar notation for $f_j$. Since $j$ is Fuchsian, $f_j$ is a diffeomorphism. By choosing a suitable orientation we assume $H_j>L_j$.

Since on $U$, the Jacobian $J$ is non-vanishing, i.e. $H-L\neq 0$. Consider a connected component of $U$, still denoted by $U$, then choosing a suitable orientation we assume $H>L$.
Notice that $H>L,H_j>L_j$ implies $H,H_j>0$. Set $H=e^w \frac{1}{c}H_j$, for $w$ a smooth function on $U$.
Then on $U$ by using $HL=||q||_{g_0}^2,H_jL_j=c^2||q||_{g_0}^2$,
\begin{\eq}
\triangle_0 w&=&-2k_{\mu} H+2k_{\mu} L-2H_j+2L_j\\
&=&2c(H-\frac{1}{c}H_j)-2c||q||_{g_0}^2(H^{-1}-cH_j^{-1})+(-2k_{\mu}-2c)(H-L).
\end{\eq}
Since $k_{\mu}\leq -c$,
\begin{\eq}
\triangle_0 w&\geq&2c(H-\frac{1}{c}H_j)-2c||q||_{g_0}^2(H^{-1}-cH_j^{-1})\\
&=&2(e^w-1)H_j-2c^2||q||_{g_0}^2H_j^{-1}(e^{-w}-1)
\end{\eq}
On the boundary of $U$, the Jacobian of $f$, $J_f=0$. So $$H=L=||q||_{g_0}=\frac{1}{c}\sqrt{H_jL_j}<\frac{1}{c}H_j.$$
So on the boundary of $U$, $w<0$. From the equivariancy, $\bar{U}$ descends to a compact domain. Therefore by maximum principle, we obtain either $w<0$ or $w\equiv 0$.

If $w<0$, $H<\frac{1}{c}H_j$, so $L>\frac{1}{c}L_1$, by $HL=\frac{1}{c^2}H_jL_j$. Therefore we have on $\bar{U}$
\begin{equation*}
0<H-L<\frac{1}{c}(H_j-L_j).
\end{equation*}
Again from $HL=\frac{1}{c^2}H_jL_j$, we obtain $e(f)=H+L<\frac{1}{c}(H_j+L_j)=\frac{1}{c}e(f_j).$

If $w\equiv 0$, then $k_{\mu}\equiv -c$ on $U$ and $e(f)=\frac{1}{c}e(f_j)$. If it is the case, since $w<0$ on $\partial U$, it must be $\partial U=\emptyset$. Then $k_{\mu}\equiv -c$ on $\ti{\Sigma}$. 

Suppose $U$ is empty, which means $J=H-L=0$ everywhere on $\Sigma$. Then from the discussion above, we have  $e(f)=2||q||_{g_0}<  \frac{1}{c}e(f_j)$. So we finish the proof.
\end{proof}

\begin{rem}
The curvature functions $\kappa$, $k$ and the energy density $e(f)$ only depend on the equivalent class in $\mc{M}_{\text{Higgs}}(\Sigma)$.
\end{rem}

\subsection{Rigidity}
In this subsection, we show the rigidity part of Proposition \ref{IntrinsicCurvatureImpliesEnergyDensity}. From Lemma \ref{indcurv} and Lemma \ref{HarmonicDomination}, if the equality holds at one point for the domination of the energy density, then the $\rho$-equivariant harmonic map $f:\ti{\Sigma}\rightarrow SL(n,\mb{C})/SU(n)$ must be a totally geodesic immersion and the curvature of its image must be a negative constant $-c$.
The rigidity means that such $(\rho,f)$ is unique in the moduli space of $PSL(n,\mb{C})$. More precisely, up to conjugate class, $(\rho,f)$ must be $\big((\tau_{\pi}\circ j)\cdot \mu_{\pi},\bar{\tau}_{\pi}\circ f_j\big)$ for some Fuchsian representation $j$, partition $\pi\in\mc{P}_n$ and representation $\mu_{\pi}:\pi_1\rightarrow\mathfrak{G}_{\pi}$. And then constant $c$ must be $\frac{1}{2}C_{\pi}$.
To show the rigidity, first we establish some Schur's type lemmas.
Let $j_{\pi}=d\tau_{\pi}|_{I}$.
\begin{lem}\label{SchurLemma}
Suppose $A$ is a complex matrix of size $n\times m$. Suppose either
(1) for any trace-free Hermitian matrix $X$ of size $2\times 2$,
\[j_n(X)\cdot A=A\cdot j_m(X);\]
or (2) for any $g\in SL(2,\mathbb R)$,
\[\tau_n(gg^*)\cdot A=A\cdot \tau_m(gg^*).\]
Then $A=0$ if $n\neq m$; $A=cI_n$ for some constant $c$ if $n=m$.
\end{lem}
\begin{proof}

First we see that Part (2) is equivalent to Part (1) by differentiating at identity or the exponential map.
So we only need to prove Part (1).

Take $X=\begin{pmatrix}1&0\\0&-1\end{pmatrix}$, then $\tau_n(X)=\text{diag}(n-1, n-3, \cdots, 3-n,1-n)$. Suppose $A=(a_{ij})$, $\tau_n(X)\cdot A=A\cdot \tau_m(X)$ implies $(n+1-2i)a_{ij}=a_{ij}(m+1-2j)$. If $m, n$ are not of the same parity, then $a_{ij}=0$ for every pair $i,j$. If $m, n$ are of the same parity, then $a_{ij}=0$ whenever $n+1-2i\neq m+1-2j$. Then if $n=m+2k, k\geq 0$, $A=\begin{pmatrix}0\\D\\0\end{pmatrix}$ where $D=\text{diag}(a_1,\cdots, a_m)$ and each $0$ in the matrix refers to a zero matrix of size $k\times m$; if $m=n+2k, k>0$,
$\begin{pmatrix}0&D&0\end{pmatrix}$, where $D=\text{diag}(a_1,\cdots, a_n)$ and each $0$ in the matrix refers to a zero matrix of size $n\times k$.

Take $X_0=\begin{pmatrix}0&1\\1&0\end{pmatrix}$, then
$$j_n(X_0)=j_n(\begin{pmatrix}0&1\\1&0\end{pmatrix})=\begin{pmatrix}0&r_1&&&&\\r_1&0&r_2&&&\\&r_2&0&\ddots&&\\&&\ddots&\ddots&\ddots&\\&&&
\ddots&0&r_{n-1}\\&&&&r_{n-1}&0\end{pmatrix},$$
where $r_i=\sqrt{i(n-i)}$ for $i=1,\cdots, n-1.$

Suppose $n\geq m$. Denote by $(j_n(X_0))^{(m)}$ the $m\times m$-minor of $j_n(X_0)$ formed by the entries $(i,j)$ satisfying $k+1\leq i, j\leq k+m$. So
$$(j_n(X_0))^{(m)}=\begin{pmatrix}0&r_{k+1}&&&&\\r_{k+1}&0&r_{k+2}&&&\\&r_{k+2}&0&\ddots&&\\&&\ddots&\ddots&\ddots&\\&&&
\ddots&0&r_{m+k-1}\\&&&&r_{m+k-1}&0\end{pmatrix}.$$
The equation $j_n(X_0)\cdot A=A\cdot j_m(X_0)$ implies $(j_n(X_0))^{(m)}D=Dj_m(X_0)$, which is
{\footnotesize $$\begin{pmatrix}0&a_2r_{k+1}&&&&\\a_1r_{k+1}&0&a_3r_{k+2}&&&\\&a_2r_{k+2}&0&\ddots&&\\&&\ddots&\ddots&\ddots&\\&&&
\ddots&0&a_mr_{m+k-1}\\&&&&a_{m-1}r_{m+k-1}&0\end{pmatrix}=\begin{pmatrix}0&a_1s_1&&&&\\a_2s_1&0&a_2r_2&&&\\&a_3s_2&0&\ddots&&
\\&&\ddots&\ddots&\ddots&\\&&&\ddots&0&a_{m-1}s_{m-1}\\&&&&a_ms_{m-1}&0\end{pmatrix},$$} where $s_i=\sqrt{i(m-i)}$ for $s=1,\cdots, m-1$.

If $n=m$, then $k=0$, $r_p=s_p$, $p=1,\cdots,n-1$. By comparing the entries of the above two matrices, we have $a_1=a_2=\cdots=a_n$. Hence $A=c I_n$ for some constant $c$.

If $n>m$, the $k>0$, we have $r_{k+1}a_2=s_1a_1, r_{k+1}a_1=s_1a_2$. Hence $r_{k+1}^2a_1a_2=s_1^2a_1a_2$ and thus $a_1a_2=0$ since $s_1=\sqrt{m-1}<\sqrt{(k+1)(m+k-1)}=r_{k+1}$. Then either $a_1=0$ or $a_2=0$. In either case, by comparing the above two matrices, it follows $a_1=\cdots=a_m=0.$ Hence $A=0$.

The proof of the other case $m=n+2k$ is similar. Therefore we finish the proof.
\end{proof}
\begin{lem}\label{Decomposition}
Let $\pi=(\lambda_1^{k_1},\cdots, \lambda_r^{k_r})\in \mathcal P_n$ with distinct $\lambda_i$'s. Suppose $A$ is a complex matrix of size $n\times n$. Suppose either
(1) for any trace-free Hermitian matrix $X$ of size $2\times 2$,
\[j_{\pi}(X)\cdot A=A\cdot j_{\pi}(X);\]
or (2) for any $g\in SL(2,\mathbb R)$,
\[\tau_{\pi}(gg^*)\cdot A=A\cdot \tau_{\pi}(gg^*).\]
Then
\[A=\text{diag}(A_1\otimes I_{\lambda_1}, A_2\otimes I_{\lambda_2},\cdots, A_r\otimes I_{\lambda_r}),\]
where $A_i$ is a complex matrix of size $k_i\times k_i$ for $i=1,\cdots, r$, and $I_k$ is the identity matrix of size $k$.
\end{lem}
\begin{proof}
It suffices to prove the result under Condition (2) since Condition (1) and Condition (2) are equivalent. First, we have $$j_{\pi}(X)=\text{diag}(\underbrace{j_{\lambda_1}(X), \cdots, j_{\lambda_1}(X)}_{\text{$k_1$ terms}}, \underbrace{j_{\lambda_2}(X), \cdots, j_{\lambda_2}(X)}_{\text{$k_2$ terms}}, \cdots, \underbrace{j_{\lambda_r}(X),\cdots, j_{\lambda_r}(X)}_{\text{$k_r$ terms}}).$$
Let $A=(U_{ij})$ where each $U_{ij}$ is a matrix of size $(\lambda_ik_i)\times (\lambda_jk_j)$ for $1\leq i, j\leq r$. Let  $U_{ij}=(U_{ij}^{kl})$, where each $U_{ij}^{kl}$ is a matrix of size $\lambda_i\times \lambda_j$ for $1\leq k\leq k_i, 1\leq l\leq k_j$. Therefore the equation $j_{\pi}(X)A=Aj_{\pi}(X)$ for any trace-free Hermitian matrix $X$ implies for any $1\leq i,j\leq r, $
\[j_{\lambda_i}(X)\cdot U_{ij}^{kl}=U_{ij}^{kl}\cdot j_{\lambda_j}(X)\]
for any trace-free Hermitian matrix $X$.
Note that for any $i\neq j$, $\lambda_i\neq \lambda_j$. By Lemma \ref{SchurLemma}, $U_{ij}^{kl}=0$ for any $i\neq j$ and $1\leq k\leq k_i, 1\leq l\leq k_j$; $U_{ii}^{kl}=c_{kl}\cdot I_{\lambda_i}$ for a constant $c_{kl}$.
So $U_{ij}=0$ when $i\neq j$; $U_{ij}=A_i\otimes I_{\lambda_i}$ for some $A_i\in gl(k_i,
\mathbb C)$ when $i=j$.
\end{proof}

The proof of the rigidity has two steps. First we show that $f$ can factor through $SL(2,\mb{R})/SO(2)$. Next we show that $\rho$ can factor through $SL(2,\mb{R})$.
\begin{lem}\label{rigiditymap}
Suppose the $\rho$-equivariant harmonic map $f:\ti{\Sigma}\rightarrow SL(n,\mb{C})/SU(n)$ is a totally geodesic immersion and the curvature of its image is a negative constant $-c$. Then there is an immersion $\hat{f}:\ti{\Sigma}\rightarrow SL(2,\mb{R})/SO(2)$, a partition $\pi\in\mc{P}_n$ and an element $x\in SL(n,\mb{C})$ such that
$$f=L_x\circ\bar{\tau}_{\pi}\circ\hat{f},$$
where $L_x$ is the left action on $SL(n,\mb{C})/SU(n)$. Furthermore $\hat{f}$ is immersed, harmonic and surjective.
\end{lem}
\begin{proof}
First we show $f$ can factor through $SL(2,\mb{R})/SO(2)$.

Step 1: We show that in the Lie algebra level the image of $f$ is standard. More precisely, let $p\in \ti{\Sigma}$, there exists $\pi\in\mc{P}_n$ and $x\in SL(n,\mb{C})$ such that $(L_{x}^{-1}\circ f)(p)=[I]$ and $\text{Image}(d(L_x^{-1}\circ f)|_{p})=\text{Image}(d\bar{\tau}_{\pi}|_{[I]})$.

Let $f(p)=[x_1]$, then $(L_{x_1}^{-1}\circ f)(p)=[I]$.
Recall the Cartan decomposition of $sl(n,\mathbb C)$ as $sl(n,\mathbb C)=\mathfrak k\oplus\mathfrak p$, where $\mathfrak k$ consists of trace-free skew-Hermitian  matrices and $\mathfrak p$ consists of trace-free Hermitian matrices. The direct sum is orthogonal with respect to the Killing form, $\big<\cdot,\cdot\big>$ on $sl(n,\mathbb C)$. Note that at $[I]$, $T_{[I]}(SL(n,\mathbb C)/SU(n))=\mathfrak p$. Denote by $\mathfrak s$ the tangent space of the image of $f$ at $[I]$, which is of real dimension $2$.

Since $f$ is totally geodesic, from Theorem 7.2 in \cite{Helgason}, $\mathfrak s\subseteq \mathfrak p$ is a Lie triple system, that is, $[\mathfrak s,[\mathfrak s,\mathfrak s]]\subseteq \mathfrak s.$
Also, $\mathfrak s$ generates a Lie subalgebra $\mathfrak g_1$ of $sl(n,\mathbb C)$ as follows:
$$\mathfrak g_1=[\mathfrak s,\mathfrak s]+\mathfrak s,\quad [\mathfrak s,\mathfrak s]\subseteq\mathfrak k, \quad \mathfrak s\subseteq\mathfrak p.$$
Since $\mathfrak s$ is of real dimension $2$, let $X, Y\in \mathfrak s$ be an orthonormal basis with respect to the Killing form. Since the curvature of the tangent plane spanned by $\mathfrak s$ is strictly negative, by the curvature formula Theorem 4.2 in \cite{Helgason}, $[X,Y]\neq 0$. Then $[\mathfrak s, \mathfrak s]$ is of $1$-dimensional spanned by $[X,Y].$ Let $H=[X,Y]\neq 0$. Since $H\in \mathfrak k, X, Y\in\mathfrak p$, then $[H, X], [H, Y]\in \mathfrak p\cap \mathfrak g_1=\mathfrak s=\text{span}\{X, Y\}$.
Suppose $[H,X]=c_1X+c_2Y$. We have
$$\big <[H,X],Y\big>=\big<H, [X, Y]\big>=\big<H, H\big><0, \text{ and }\big<[H,X],Y\big>=\big<c_1X+c_2Y, Y\big>=c_2|Y|^2=c_2.$$
So $c_2<0.$ And
$$\big <[H,X],X\big>=\big<H, [X, X]\big>=0, \text{ and }\big<[H,X],X\big>=\big<c_1X+c_2Y, X\big>=c_1|X|^2=c_1.$$
So $c_1=0.$ Hence $[H,X]=c_2Y$ where $c_2<0.$ Similarly, we have $[H,Y]=c_3X$ where $c_3>0$.

Let $a=\frac{2}{\sqrt{-c_2}}, b=\frac{2}{\sqrt{c_3}}$ and $X'=aX, Y'=bY, H'=\frac{ab}{2}H. $ Then we obtain a basis $\{H', X', Y'\}\subseteq \mathfrak g_1$ satisfying $H'\in [\mathfrak s, \mathfrak s]\subseteq\mathfrak k$, $X', Y'\in \mathfrak s\subseteq \mathfrak p$, $\big<X',Y'\big>=0,$ and
\[[X', Y']=2H', \quad [H',X']=-2Y', \quad [H', Y']=2X'.\]
We may then construct a Lie algebra isomorphism $\phi: sl(2,\mathbb R)\rightarrow \mathfrak g_1$ by
\[\phi: \begin{pmatrix}0&1\\-1&0\end{pmatrix}\mapsto H',\quad \begin{pmatrix}1&0\\0&-1\end{pmatrix}\mapsto X',\quad \begin{pmatrix}0&1\\1&0\end{pmatrix}\mapsto Y'.\]
So $\phi$ is in fact a Lie algebra homomorphism from $sl(2,\mathbb R)$ to $sl(n, \mathbb C)$ with image as $\mathfrak g_1$.
From Proposition \ref{SLstandard}, there exists a $g_0\in SL(n,\mathbb C)$ and a partition $\pi\in \mathcal P_n$ such that
\[\phi=Ad_{g_0}\circ j_{\pi}: sl(2,\mathbb R)\rightarrow sl(n,\mathbb C).\]

Since $X', Y'\in \mathfrak s\subseteq \mathfrak p$, then for any $S\in \text{span}\{\begin{pmatrix}1&0\\0&-1\end{pmatrix},\begin{pmatrix}0&1\\1&0\end{pmatrix}\}$, we have $\phi(S)\in\mathfrak{p}$,
$$(Ad_{g_0}\circ j_{\pi}(S))^*=Ad_{g_0}\circ j_{\pi}(S).$$
That is, $(g_0^{-1})^*j_{\pi}(S)g_0^*=g_0j_{\pi}(S)g_0^{-1}$. So $j_{\pi}(S)\cdot (g_0^*g_0)=(g_0^*g_0)\cdot j_{\pi}(S)$. Then by Lemma \ref{Decomposition}, suppose $\pi=(\lambda_1^{k_1},\cdots, \lambda_r^{k_r})$, we have $g_0^*g_0=\text{diag}(A_1\otimes I_{\lambda_1}, \cdots, A_r\otimes I_{\lambda_r})$ for some $A_i$'s. Then each $A_i$ is positive Hermitian for $i=1,\cdots,n$. Then we may choose a positive Hermitian  matrix $T_i$ such that $T_i^2=A_i$. Let $T=\text{diag}(T_1\otimes I_{\lambda_1}, \cdots, T_r\otimes I_{\lambda_r})$. Then $T$ is positive Hermitian and $g_0^*g_0=T^2$. Let $g_1=g_0T^{-1}$. Then $g_1^*g_1=I$. Notice that $T$ is in the centralizer of the image of $j_{\pi}$. Then $\phi=Ad_{g_1}\circ j_{\pi}$. We may further assume $\det(g_1)=1$, then $g_1\in SU(n)$. Let $x=x_1g_1^{-1}$. Then $x$ satisfies the requirements, so we finish the proof of Step 1.

Step 2: We show that in the Lie group level the image of $f$ is standard. More precisely, $\text{Image}(L_x^{-1}\circ f)\subseteq\text{Image}(\bar{\tau}_{\pi})$.

Consider $L^{-1}_{x}\circ f$, from Step 1, the image of $d(L^{-1}_{x}\circ f)$ at $p$ is the same as the image of $d\bar{\tau}_{\pi}$ at $[I]$. From Theorem 7.2 in \cite{Helgason}, the image of $\bar{\tau}_{\pi}$ is a complete totally geodesic submanifold. Since a totally geodesic submanifold is determined by its tangent space at one point, we have $\text{Image}(L_x^{-1}\circ f)\subseteq\text{Image}(\bar{\tau}_{\pi})$. We finish the proof of Step 2.

Step 3: We show the existence and the property of $\hat{f}$.

Since the image of $L_{x}^{-1}\circ f$ is in the image of $\bar{\tau}_{\pi}$ and the map $\bar{\tau}_{\pi}$ is injective, the map
$$\hat{f}:=\bar{\tau}_{\pi}^{-1}\circ f:\ti{\Sigma}\rightarrow SL(2,\mb{R})/SO(2)$$
is well-defined. Since $f$ is a harmonic immersion and $\bar{\tau}_{\pi}$ is a totally geodesic embedding satisfying $\bar\tau_{\pi}^*g_{SL(n,\mathbb C)/SU(n)}=\frac{1}{2}C_{\pi}g_{SL(2,\mathbb R)/SO(2)}$, $\hat{f}$ is also a harmonic immersion.
To show $\hat{f}$ is surjective, we consider the pullback metric $g_f$ of $f$, by $\rho$-equivariancy $g_f$ can descend to $\Sigma$, which is compact. So $g_f$ is complete. Since $f$ maps geodesics to geodesics, then the image of $f$ is geodesic complete, which also means complete. Since the image of $L_x\circ\bar{\tau}_{\pi}$ is complete, we obtain $\hat{f}$ is surjective. We finish the proof of Step 3 and then the whole proof.
\end{proof}

\begin{lem}\label{rigidityrep}
Under the assumption and the conclusion of Lemma \ref{rigiditymap}, then there is a Fuchsian representation $j$ and a representation $\mu_{\pi}:\pi_1\rightarrow \mathfrak{G}_{\pi}$ such that the corresponding $j$-equivariant harmonic map $f_j:\ti{\Sigma}\rightarrow SL(2,\mb{R})/SO(2)$ coincides with $\hat{f}$ and
$$\rho=\text{Ad}_{x^{-1}}\circ\big((\tau_{\pi}\circ j)\cdot\mu_{\pi}\big),$$
where $\text{Ad}$ is the adjoint action on $SL(n,\mb{C})$.
\end{lem}
\begin{proof}
From Lemma \ref{rigiditymap} and the $\rho$-equivariancy of $f$, we have the following commutative diagram, for any $\gamma\in \pi_1$,
$$\begin{array}{ccccccc}
\ti{\Sigma}&\stackrel{\hat{f}}{\longrightarrow}&SL(2,\mb{R})/SO(2)&\stackrel{\bar{\tau}_{\pi}}{\longrightarrow}&SL(n,\mb{C})/SU(n)&
\stackrel{L_x}{\longrightarrow}&SL(n,\mb{C})/SU(n)\\
\quad\downarrow \gamma & & & & \quad\downarrow L_{\rho_x(\gamma)} & & \quad\downarrow L_{\rho(\gamma)}\\
\ti{\Sigma}&\stackrel{\hat{f}}{\longrightarrow}&SL(2,\mb{R})/SO(2)&\stackrel{\bar{\tau}_{\pi}}{\longrightarrow}&SL(n,\mb{C})/SU(n)&
\stackrel{L_x}{\longrightarrow}&SL(n,\mb{C})/SU(n)
\end{array}$$
where $\rho_x(\gamma)=x^{-1}\rho(\gamma)x=(\text{Ad}_{x^{-1}}\circ\rho)(\ga)$. Since $\bar{\tau}_{\pi}$ is injective and $\hat{f}$ is surjective,
$$s_x(\ga):=\bar{\tau}_{\pi}^{-1}L_{\rho_x(\gamma)}\bar{\tau}_{\pi}:SL(2,\mb{R})/SO(2)\rightarrow SL(2,\mb{R})/SO(2)$$
is a well-defined map. Since $s_x(\ga)$ is also an isometry of the hyperbolic space $SL(2,\mb{R})/SO(2)$, there exists a unique element $\ti{\rho}_x(\ga)\in PSL(2,\mb{R})$, such that $s_x(\ga)=L_{\ti{\rho}_x(\ga)}$. From the uniqueness of $\ti{\rho}_x(\ga)$, it is easy to see $\ti{\rho}_x:\pi_1\rightarrow PSL(2,\mb{R})$ is a homomorphism. By the definition of $\ti{\rho}_x$, $\hat{f}$ is $\ti{\rho}_x$-equivariant. Since $\hat{f}$ is an immersed $\tilde \rho_x$-equivariant harmonic map, the pullback metric is a hyperbolic metric on $S$. So $\tilde \rho_x$ is the holomony of the hyperbolic metric, which is Fuchsian, and can be lifted to $\hat{\rho}_x: \pi_1\rightarrow SL(2,\mb{R})$. So we can complete the commutative diagram as follows:
$$\begin{array}{ccccccc}
\ti{\Sigma}&\stackrel{\hat{f}}{\longrightarrow}&SL(2,\mb{R})/SO(2)&\stackrel{\bar{\tau}_{\pi}}{\longrightarrow}&SL(n,\mb{C})/SU(n)&
\stackrel{L_x}{\longrightarrow}&SL(n,\mb{C})/SU(n)\\
\quad\downarrow \gamma & & \quad\downarrow L_{\hat{\rho}_x(\gamma)} & & \quad\downarrow L_{\rho_x(\gamma)} & & \quad\downarrow L_{\rho(\gamma)}\\
\ti{\Sigma}&\stackrel{\hat{f}}{\longrightarrow}&SL(2,\mb{R})/SO(2)&\stackrel{\bar{\tau}_{\pi}}{\longrightarrow}&SL(n,\mb{C})/SU(n)&
\stackrel{L_x}{\longrightarrow}&SL(n,\mb{C})/SU(n)
\end{array}$$
For any $y\in SL(2,\mb{R})$, we have $$L_{\rho_x(\ga)}\big(\bar{\tau}_{\pi}([y])\big)=\bar{\tau}_{\pi}\big(L_{\hat{\rho}_x(\ga)}([y])\big),$$
which implies
$$[\rho_x(\ga)\cdot\tau_{\pi}(y)]=[\tau_{\pi}(\hat{\rho}_x(\ga)\cdot y)]=[(\tau_{\pi}\circ\hat{\rho}_x)(\ga)\cdot \tau_{\pi}(y)].$$
Let $A=(\tau_{\pi}\circ\hat{\rho}_x)(\ga)^{-1}\cdot\rho_x(\ga)$. Then $\tau_{\pi}(y)^{-1}A\tau_{\pi}(y)\in SU(n)$ for $\forall y\in SL(2,\mb{R})$, which implies
$$I=\tau_{\pi}(y)^{-1}A\tau_{\pi}(y)(\tau_{\pi}(y)^{-1}A\tau_{\pi}(y))^*=\tau_{\pi}(y)^{-1}A\tau_{\pi}(yy^{*})A^*(\tau_{\pi}(y)^*)^{-1}$$
Then $ASA^*=S$ for every $S=\tau_{\pi}(yy^*)$, $y\in SL(2,R)$. Let $y=I$, then $AA^*=I$. So we obtain $AS=SA$ for every $S=\tau_{\pi}(yy^*)$, $y\in SL(2,R)$. Then from Lemma \ref{Decomposition}, we obtain $A=\text{diag}(A_1\otimes I_{\lambda_1},\cdots, A_r\otimes I_{\lambda_r})$. Since $A\in SU(n)$, we obtain $A\in \mathfrak{G}_{\pi}$. Since $A$ commutes with the image of $\tau_{\pi}$, $A$ gives a representation $\mu_{\pi}:\pi_1\rightarrow \mathfrak{G}_{\pi}$. So $\rho_x=(\tau_{\pi}\circ\hat{\rho}_x)\cdot \mu_{\pi}$. Let $j=\hat{\rho}_x$, we finish the proof.
\end{proof}

Finally, we are ready to give the proof of Proposition \ref{IntrinsicCurvatureImpliesEnergyDensity}.
\begin{proof}(of Proposition \ref{IntrinsicCurvatureImpliesEnergyDensity})
By Lemma \ref{indcurv}, we obtain that the Gaussian curvature $k$ of its pullback metric satisfies $k\leq \kappa\leq -c$.  After obtaining the curvature estimate, from Lemma \ref{HarmonicDomination}, we obtain the energy density $e(f)\leq \frac{1}{c}e(f_j)$. And if the equality holds at one point, then $\kappa\equiv -c$ and the map is totally geodesic. The rigidity part of the statement follows from Lemma \ref{rigiditymap} and Lemma \ref{rigidityrep}. We finish the proof.
\end{proof}

\section{Geometric applications}\label{applications}
In this section, we will derive two main applications from Section \ref{domination results} and \ref{ProofOfProposition} to equivariant minimal surfaces and maximal surfaces in product spaces. We normalize the induced Riemannian metric on $X=SL(n,\mb{C})/SU(n)$ from the Killing form of $sl(n,\mb{C})$ to $g_{n}$, such that $\bar{\tau}_n^*g_n=g_{\mb{H}^2}$. 

Let $j:\pi_1\rightarrow SL(2,\mathbb R)$ be a Fuchsian representation, $\rho:\pi_1\rightarrow SL(n,\mb{C})$ be a reductive representation. Let $\Sigma=(S,J)$ be a Riemann surface. By the non-Abelian Hodge theory, we obtain a $j$-equivariant harmonic map $f_j:\ti{\Sigma}\rightarrow (\mathbb H^2, g_{\mb{H}^2})$  and  a $\rho$-equivariant harmonic map $f_{\rho}:\ti{\Sigma}\rightarrow (X, g_n)$. Since $j$ is Fuchsian, $f_j$ is a diffeomorphism. 
\subsection{Minimal surfaces}\label{min}
The map $(f_j,f_{\rho})$ gives a $(j,\rho)$-equivariant harmonic embedding
$$(f_j,f_{\rho}):\ti{\Sigma}\rightarrow \big(\mathbb H^2\times X,g_{\mb{H}^2}+g_n\big).$$
It is also the graph of $f_{\rho}\circ f_j^{-1}:\mathbb H^2\rightarrow X.$ The Hopf differential of $(f_j,f_{\rho})$ is $\text{Hopf}\big((f_j,f_{\rho})\big)=\text{Hopf}(f_j)+\text{Hopf}(f_{\rho}).$ Suppose $\text{Hopf}(f_j)=-\text{Hopf}(f_{\rho})$, then $(f_j,f_{\rho})$ is conformal. Together with the harmonicity, the map $(f_j,f_{\rho})$ gives a $(j,\rho)$-equivariant embedded minimal surface.
\begin{prop}\label{ministable}
Let $\Sigma=(S,J)$ be a Riemann surface and $q_2\in H^0(\Sigma, K_\Sigma^2)$. Let $j, \hat j$ be the Fuchsian representations which correspond to $q_2, -q_2$ respectively. Suppose $\rho:\pi_1\rightarrow SL(n,\mb{C})$ is a reductive representation in the same Hitchin fiber as the $n$-Fuchsian representation $\tau_n\circ \hat{j}$ in the moduli space of Higgs bundles over $\Sigma$. Then
$$(f_j,f_{\rho}):\ti{\Sigma}\rightarrow \big(\mathbb H^2\times X, g_{\mb{H}^2}+g_n\big)$$
gives a stable $(j,\rho)$-equivariant embedded minimal surface.
\end{prop}
\begin{proof}
One only needs to show the minimal surface is stable. Since $\bar{\tau}_{n}:\big(\mathbb H^2,g_{\mb{H}^2}\big)\rightarrow\big(X,g_{n}\big)$ is isometric, the energy density of $f_{\tau_n\circ j}$ is just the energy density of $f_j$. 
Suppose the Fuchsian representation corresponds to the Higgs bundle $(E,\phi)=\big(K_\Sigma^{\frac{1}{2}}\oplus K_\Sigma^{-\frac{1}{2}},\left(
\begin{array}{cc}
0 & q_2\\
1 & 0
\end{array}
\right)\big).$ Suppose $h$ solves the Hitchin equation for $(E,\phi)$. In this case, $h=\text{diag}(h_1,h_1^{-1})$ and locally, the Hitchin equation reduces to $\partial_z\partial_{\bar z}\log h_1+h_1^{-2}-|q_2|^2h_1^2=0.$ From Equation $(\ref{expression})$, $e(f_{j})\cdot g_0=2\text{tr}(\phi\phi^{*_h})=2(|q_2|^2h_1^2+h_1^{-2})$. Notice that the energy density and the Hitchin equation are invariant under the $U(1)$ action on $q_2$. In particular, $e(f_{j})=e(f_{ \hat{j}}).$ Since $\rho$ and $\tau_n\circ \hat j$ share the same Hitchin fiber of $\mathcal M_{Higgs}(\Sigma)$, it follows from Theorem \ref{dommain} that $e(f_{\tau_n\circ \hat{j}})\geq e(f_{\rho})$. So $e(f_j)=e(f_{\hat j})=e(f_{\tau_n\circ \hat{j}})\geq e(f_{\rho})$.

We claim the map $f_{\rho}\circ f_{j}^{-1}$ is area-decreasing. It is enough to consider the immersion points of $f_{\rho}$. At the immersion points, locally $f_{\rho}\circ f_{j}^{-1}$ maps a surface to a surface. Fix a basis, the area-decreasing condition is equivalent to the Jacobian $|J(f_{\rho}\circ f_{j}^{-1})|\leq 1$, which is $|J(f_{\rho})|\leq |J(f_{j})|$. We use the notation in Section \ref{Domination of metrics}. Then $J(f)=H-L$, $e(f)=H+L$, $||\text{Hopf}(f)||_{g_0}^2=HL$. So $|J(f)|^2=|e(f)|^2-4||\text{Hopf}(f)||_{g_0}^2$. Therefore
\begin{\eq}
|J(f_{j})|^2=|e(f_{j})|^2-4||q_2||_{g_0}^2=|e(f_{\hat{j}})|^2-4||q_2||_{g_0}^2\geq |e(f_{\rho})|^2-4||q_2||_{g_0}^2=|J(f_{\rho})|^2.
\end{\eq}
So $f_{\rho}\circ f_{j}^{-1}$ is area-decreasing. Then from Lee-Wang \cite{LeeWang} Theorem 4.1, we obtain the stability of the minimal surface.
\end{proof}
For the completeness, we prove the following existence result, which is well-known. Let $G$ be a reductive Lie group, $K$ be a maximal compact subgroup of $G$. Denote $g_{G/K}$ as the Riemannian metric of $G/K$ induced from the Killing form of $Lie(G)$. 
\begin{prop}\label{miniexist}
For any Fuchsian representation $j$, reductive representation $\rho:\pi_1\rightarrow G$, constant $c>0$, there exists a Riemann surface $\Sigma=(S,J)$, such that the corresponding $j$-equivariant harmonic map $f_j:\ti{\Sigma}\rightarrow \mathbb H^2$, $\rho$-equivariant harmonic map $f_{\rho}:\ti{\Sigma}\rightarrow G/K$ gives a $(j,\rho)$-equivariant embedded minimal surface $$(f_j,f_{\rho}):\ti{\Sigma}\rightarrow \big(\mathbb H^2\times G/K,g_{\mb{H}^2}+cg_{G/K}\big).$$
\end{prop}
\begin{proof}
It is enough to show there exists a Riemann surface $\Sigma=(S,J)$ such that the $(f_j,f_{\rho})$ is conformal with respect to $\Sigma$. Consider the energy function $E_{(j,\rho)}$ on the Teichm\"uller space $\mathcal T(S)$, $E_{(j,\rho)}([\Sigma]):=\int_{\Sigma} e(f^{\Sigma}_{(j,\rho)})dV$. The $(j,\rho)$-equivariant harmonic map $f^{\Sigma}_{(j,\rho)}=(f_j^{\Sigma},f_{\rho}^{\Sigma})$ is from the non-Abelian Hodge theory with respect to $\Sigma$, which is unique up to isometry. And it is clear that $E_{(j,\rho)}$ only depends on the class $[\Sigma]\in\mathcal T(S)$. So $E_{(j,\rho)}$ is well-defined on $\mathcal T(S)$. By the classical results of Sacks-Uhlenbeck \cite{SacksUhlenbeck1, SacksUhlenbeck2} and Schoen-Yau \cite{SchoenYau}, if the Riemann surface $\Sigma$ is a critical of $E_{(j,\rho)}$, then the corresponding harmonic map $f^{\Sigma}_{(j,\rho)}$ is conformal. Notice that $E_{(j,\rho)}\geq 0$. We show $E_{(j,\rho)}$ is proper, then it has a minimum point. In fact, $E_{(j,\rho)}=E_j+E_{\rho}\geq E_j$. From Tromba \cite{Tromba}, $E_j$ is proper. So $E_{(j,\rho)}$ is proper and we finish the proof.
\end{proof}
Suppose $G$ is a semisimple Lie group of rank $1$, then the sectional curvature of $G/K$ is strictly negative. For a constant $c>0$, denote $g_{-c}$ as the rescaling metric of $g_{G/K}$ such that the maximum of the sectional curvature of $g_{-c}$ is $-c$.
\begin{prop}\label{minirank1}
Let $j$ be a Fuchsian representation. Let $G$ be a reductive Lie group of rank $1$, $\rho:\pi_1\rightarrow G$ be an irreducible representation. Suppose $\rho$ does not preserve any geodesic arc in $G/K$. Then for $c\geq 1$, there is a unique $(j,\rho)$-equivariant minimal surface $f:\ti{S} \rightarrow \big(\mathbb H^2\times G/K,g_{\mb{H}^2}+g_{-c}\big)$. Moreover, it is an embedding.
\end{prop}
\begin{proof}
The existence part follows from Proposition \ref{miniexist}. For the uniqueness, let $f$ be such a minimal surface. Then $f=(f_j,f_{\rho})$ is a pair a harmonic maps with respect to the pullback metric. Since $j$ is Fuchsian, from the discussion at the beginning of this section, the minimal surface is the graph of $\ti{f}=f_{\rho}\circ f_j^{-1}$. Since the sectional curvature $k_{G/K}$ of $G/K$ satisfies $k_{G/K}\leq -c\leq -1$, from Proposition \ref{IntrinsicCurvatureImpliesEnergyDensity}, we obtain $e(f_{\rho})\geq e(f_j)$. So as the same proof as in Proposition \ref{ministable}, we have $\ti{f}$ is area-decreasing. Then from Theorem 4.1 in Lee-Wang \cite{LeeWang}, $f$ gives a stable minimal surface.

By carefully checking the proof of Theorem 4.1 in \cite{LeeWang}, for a variational vector field $V$ along $\ti{f}$, a deformation family $\ti{f}_t$ with respect to $V$, $t\in (-\epsilon,\epsilon)$, we have
$$\frac{d^2A_t}{dt^2}\big|_{t=0}\geq\int_\Sigma\sum\limits_{i=1}^{2}-\big<R\big(V,d{f}(a_i)\big)d{f}(a_i),V\big>dvol_{g_0},$$
where $A$ is the area functional, $R$ is the Riemannian curvature tensor on $G/K$, and $a_i$'s form an orthonormal basis on the tangent space of $\mathbb H^2$ with respect to the metric $g_0$. If $f_{\rho}$ has no immersion point, from Sampson \cite{Sampson}, the image of $f_{\rho}$ lies in a geodesic arc. By the $\rho$-equivariancy, $\rho$ preserves this geodesic arc, which cannot happen by the assumption. So $f_{\rho}$ has at least one immersion point and thus the set of immersion points is open and dense following from \cite{Sampson}. So if $\frac{d^2A_t}{dt^2}\big|_{t=0}=0$, $G/K$ being negatively curved forces $V=0$. Hence the minimal surface $f$ is strictly stable.

Fix a conformal class $\Sigma\in \mathcal T(S)$, since $j,\rho$ are irreducible, there is a unique $j$-equivariant harmonic map $f^{\Sigma}_{j}: \ti{\Sigma}\rightarrow \mathbb H^2$ and a unique $\rho$-equivariant harmonic map $f^{\Sigma}_{\rho}: \ti{\Sigma}\rightarrow X.$ To show the uniqueness of the minimal surface, one only needs to show in $\mathcal T(S)$ there is a unique $[\Sigma]$ such that $(f^{\Sigma}_{j}, f^{\Sigma}_{\rho})$ is conformal. As in the proof of Proposition \ref{miniexist}, it is equivalent to the uniqueness of the critical point of $E_{(j,\rho)}([\Sigma])=E(f_{(j,\rho)}^{\Sigma})$. Consider the restriction of the area functional $A$ on the maps  $f_{(j,\rho)}^{\Sigma}$ parameterized by $\mathcal T(S)$. Notice that the critical point of $E$ is minimal, which is also the critical point of $A$. And $E=\int_\Sigma (H+L)dvol_{g_0}\geq \int_\Sigma |H-L|dvol_{g_0}=A$, the equality holds if and only if $||\text{Hopf}(f_{(j,\rho)}^{\Sigma})||_{g_0}^2=HL=0$, which is the critical point of $E$. Since for every critical point of $A$, it is strictly stable, which means strictly locally minimal, then we have for every critical point of $E$, it is strictly locally minimal. Since $\mathcal T(S)$ is of finite dimension, then the critical point of $E$ is unique. So we finish the proof.
\end{proof}
\begin{rem}
For $G=SL(2,\mb{C})$, which is simple and of rank $1$, the irreducible representations not preserving any geodesic arc are precisely the non-elementary representations. The elementary representation $\rho$ is of the following three types: (1) $\rho$ is reducible; (2) the image of $\rho$ lies in $SU(2)$; (3) the image of $\rho$ lies in the subgroup generated by $\left(
\begin{array}{cc}
\lambda & 0\\
0 & \lambda^{-1}
\end{array}
\right), \lambda\in \mb{C}^*$
and
$\left(
\begin{array}{cc}
0 & 1\\
-1 & 0
\end{array}
\right).$
\end{rem}
The harmonic maps are unique up to the centralizer of the representation. For the product representation $(j,\rho)$ with $j$ Fuchsian, the centralizer of $(j,\rho)$ is just the centralizer of $\rho$.

\subsection{Maximal surfaces and structure of $n$-Fuchsian fibers}
The map $(f_j,f_{\rho})$ gives a $(j,\rho)$-equivariant harmonic embedding
$$(f_j,f_{\rho}): \ti\Sigma\rightarrow (\mathbb H^2\times X, g_{\mb{H}^2}-g_{n}).$$ Suppose $f_j^*g_{\mb{H}^2}>f_{\rho}^*g_{n}$, then the pullback metric $(f_j,f_{\rho})^*(g_{\mb{H}^2}-g_{n})$ is Riemannian, which means the image of $\ti{\Sigma}$ is spacelike.
The Hopf differential of $(f_j,f_{\rho})$ is $\text{Hopf}\big((f_j,f_{\rho})\big)=\text{Hopf}(f_j)-\text{Hopf}(f_{\rho})$. Suppose $\text{Hopf}(f_j)=\text{Hopf}(f_{\rho})$, then $(f_j,f_{\rho})$ is conformal. Together with the harmonicity, we obtain $(f_j,f_{\rho})$ gives a $(f_j,f_{\rho})$-equivariant embedded spacelike maximal surface. For the basic materials on maximal surfaces, one may refer \cite{MaximalSurface}.

We recall a result from Tholozan \cite{Tholozan}, Section 2.
\begin{prop}\label{Nico}(Tholozan \cite{Tholozan})
For a Fuchsian representation $j$ and a reductive representation $\rho:\pi_1\rightarrow SL(n,\mb{C})$, suppose there is a $(j,\rho)$-equivariant smooth map
$$f:\big(\mathbb H^2, g_{\mb{H}^2}\big)\rightarrow \big(X, g_{n}\big)$$
with Lipschitz constant strictly less than 1. Then there is a unique conformal class $[\Sigma]\in \mathcal T(S)$, such that $(f_j,f_{\rho}):\ti{\Sigma}\rightarrow \big(\mathbb H^2\times X, g_{\mb{H}^2}-g_{n}\big)$ gives a $(j,\rho)$-equivariant embedded spacelike maximal surface satisfying the conformal class of the induced metric is $[\Sigma]$.
\end{prop}
From Theorem \ref{dommain}, the $n$-Fuchsian representation dominates the representations in the same Hitchin fiber, which implies the assumption in Proposition \ref{Nico}. So we can construct maximal surfaces as follows.
\begin{prop}\label{maximalsurface}
Let $\Sigma=(S,J)$ be a Riemann surface. Suppose $\rho\in \mathcal M_{Bettis}(S)$ is in the $n$-Fuchsian fiber containing $\tau_n\circ j$ in $\mathcal M_{Higgs}(\Sigma)$. Suppose $\rho$ is not conjugate to $(\tau_n\circ j)\cdot \mu_n$ for any representation $\mu_n:\pi_1\rightarrow \mathfrak{G}_n$. Then $(f_j,f_{\rho}):\ti{\Sigma}\rightarrow \big(\mathbb H^2\times X, g_{\mb{H}^2}-g_{n}\big)$ gives a $(j,\rho)$-equivariant embedded spacelike maximal surface. 

Moreover, the conformal class $\Sigma$ is unique among all the $(j,\rho)$-equivariant maximal space-like surfaces.
\end{prop}
\begin{proof}
From Theorem \ref{dommain}, $f_{j}^*g_{\mb{H}^2}>f_{\rho}^*g_n$, which means $f\circ f_j^{-1}$ is a $(j,\rho)$-equivariant distance-decreasing map. So the assumption of Proposition \ref{Nico} holds. Then the statement follows from the uniqueness in Proposition \ref{Nico}.
\end{proof}
\begin{rem}
From Proposition \ref{maximalsurface}, the uniqueness of the $(j,\rho)$-equivariant spacelike maximal surfaces is reduced to the uniqueness of the $(j,\rho)$-equivariant harmonic maps with respect to a Riemann surface $\Sigma=(S,J)$. From the non-Abelian Hodge theory, the uniqueness of the harmonic maps is up to the centralizer of the representation $(j,\rho)$. Since $j$ is Fuchsian, the uniqueness is just up to the centralizer of $\rho$.
\end{rem}
\begin{rem}
In fact, one may also carry a similar computation as in Lee-Wang \cite{LeeWang} Theorem 3.1 and show the spacelike maximal surface in Proposition \ref{maximalsurface} is automatically stable.
\end{rem}
Combining with Proposition \ref{Nico} and Theorem \ref{dommain}, we obtain the following description on the space of $n$-Fuchsian fibers.
\begin{prop}\label{n-fiber}
Let $\tau_n\circ j$ be an $n$-Fuchsian representation. Let $\rho:\pi_1\rightarrow SL(n,\mb{C})$ be a reductive representation, not conjugate to $(\tau_n\circ j)\cdot \mu_n$ for any representation $\mu_n:\pi_1\rightarrow \mathfrak{G}_n$. Suppose $\tau_n\circ j$ and $\rho$ are in the same Hitchin fiber of $\mathcal M_{Higgs}(\Sigma)$ for some $\Sigma$. Then for another $\Sigma^{\pr}$, $[\Sigma^{\pr}]\neq [\Sigma]$ in $\mathcal T(S)$, $\tau_n\circ j$ and $\rho$ can not be in the same Hitchin fiber of $\mathcal M_{Higgs}(\Sigma^{\pr})$.
\end{prop}
\begin{proof}
As in the proof of Proposition \ref{maximalsurface}, if $\tau_n\circ j$ and $\rho$ are in the same Hitchin fiber for some conformal class, then the assumption of Proposition \ref{Nico} holds. So Proposition \ref{n-fiber} follows from the uniqueness in Proposition \ref{Nico}.
\end{proof}

\end{document}